\documentclass[11pt]{amsart}
\usepackage{amsmath}
\usepackage[psamsfonts]{amssymb}
\usepackage[all]{xy}
\usepackage{graphicx}
\usepackage[T1]{fontenc}
\usepackage[bitstream-charter]{mathdesign}
\usepackage{hyperref}
\usepackage{verbatim}
\usepackage{color}
\usepackage{mathtools}
\makeatletter
\DeclareRobustCommand\widecheck[1]{{\mathpalette\@widecheck{#1}}}
\def\@widecheck#1#2{%
    \setbox\z@\hbox{\m@th$#1#2$}%
    \setbox\tw@\hbox{\m@th$#1%
       \widehat{%
          \vrule\@width\z@\@height\ht\z@
          \vrule\@height\z@\@width\wd\z@}$}%
    \dp\tw@-\ht\z@
    \@tempdima\ht\z@ \advance\@tempdima2\ht\tw@ \divide\@tempdima\thr@@
    \setbox\tw@\hbox{%
       \raise\@tempdima\hbox{\scalebox{1}[-1]{\lower\@tempdima\box
\tw@}}}%
    {\ooalign{\box\tw@ \cr \box\z@}}}
\makeatother

\theoremstyle{definition}
\newtheorem{theorem}{Theorem}[section]
\newtheorem{prop}[theorem]{Proposition}
\newtheorem{lemma}[theorem]{Lemma}
\newtheorem{cor}[theorem]{Corollary}

\newtheorem{question}[theorem]{Question}
\newtheorem{notation}[theorem]{Notation}

\newtheorem{ex}[theorem]{Example}
\theoremstyle{remark}
\newtheorem{dfn}[theorem]{Definition}
\newtheorem{remark}[theorem]{Remark}

\def\co{\colon\thinspace}

\def\ep{\epsilon}

\def\Q{\mathbb{Q}}
\def\R{\mathbb{R}}
\def\Z{\mathbb{Z}}
\def\N{\mathbb{N}}
\def\C{\mathbb{C}}

\def\dell{\delta_{\mathrm{ell}}}
\def\dellu{\delta_{\mathrm{ell}}^{\mathrm{u}}}

\DeclareMathOperator{\Img}{Im}

\numberwithin{equation}{section}

\newtheorem{proposition}[theorem]{Proposition}

\newtheorem{lemma-definition}[theorem]{Lemma-Definition}

\newtheorem{definition}[theorem]{Definition}

\newcommand{\op}{\operatorname}

\newcommand{\too}{\longrightarrow}

\newcommand{\norm}[1]{\left\| #1 \right\|}

\newcommand{\eps}{\varepsilon}

\newcommand{\bpm}{\begin{pmatrix}}
\newcommand{\epm}{\end{pmatrix}}

\renewcommand{\epsilon}{\varepsilon}

\makeatletter
\newcommand*{\rom}[1]{\expandafter\@slowromancap\romannumeral #1@}
\makeatother
\title{Symplectically knotted codimension-zero embeddings of domains in $\mathbb{R}^4$}
\author{Jean Gutt \and Michael Usher}
\address{Department of Mathematics\\University of Georgia\\Athens, GA 30602}
\email{guttjean@gmail.com, usher@uga.edu}

\begin{document}

\begin{abstract}
   We show that many toric domains $X$ in $\R^4$ admit symplectic embeddings $\phi$ into dilates of themselves which are knotted in the strong sense that there is no symplectomorphism of the target that takes $\phi(X)$ to $X$.  For instance $X$ can be taken equal to a polydisk $P(1,1)$, or to any convex toric domain that both is contained in $P(1,1)$ and properly contains a ball $B^4(1)$; by contrast a result of McDuff shows that $B^4(1)$ (or indeed any four-dimensional ellipsoid) cannot have this property.  The embeddings are constructed based on recent advances on symplectic embeddings of ellipsoids, though in some cases a more elementary construction is possible.  The fact that the embeddings are knotted is proven using filtered positive $S^1$-equivariant symplectic homology.
\end{abstract}

\maketitle
\section{Introduction} \label{intro}

Recent years have seen a significant improvement in our understanding of when one region in $\mathbb{R}^4$ symplectically embeds into another, see \emph{e.g.} \cite{M}, \cite{MS}, \cite{CG}.  Complementing this existence question, one can ask whether embeddings are unique up to an appropriate notion of equivalence; in particular, if $A\subset U\subset\R^4$ this entails asking whether every symplectic embedding $A\hookrightarrow U$ is equivalent to the inclusion.  Somewhat less is known about this uniqueness question, though there are positive results in \cite{M},\cite{CG} and negative results in \cite{FHW}, \cite{Hi}.  We show in this paper that modern techniques of constructing symplectic embeddings $B\hookrightarrow U$ often give rise, when restricted to certain subsets $A\subset B\cap U$, to embeddings $A\hookrightarrow U$ that are distinct from the inclusion in a strong sense. 

The subsets of $\mathbb{R}^4$ (and in some cases more generally in $\R^{2n}\cong \C^n$) that we consider are toric domains; let us set up some notation and recall basic definitions.

 Define $\mu\co \mathbb{C}^n\to [0,\infty)^n$ by \[ \mu(z_1,\ldots,z_n)=(\pi|z_1|^2,\ldots,\pi|z_n|^2).\]  A \textbf{toric domain} is by definition a set of the form $X_{\Omega}=\mu^{-1}(\Omega)$ where $\Omega$ is a domain in $[0,\infty)^n$.  Throughout the paper the term ``domain'' will always refer to the closure of a bounded open subset of $\R^n$ or $\C^n$; in particular domains are by definition compact.

Given $\Omega\subset[0,\infty)^n$, we define
\begin{equation}\label{omegahat}
\widehat{\Omega} = \left\{(x_1,\ldots,x_n)\in\R^n \;\big|\; (|x_1|,\ldots,|x_n|)\in\Omega\right\}.
\end{equation}

Symplectic embedding problems for toric domains are currently best understood when the domains are concave or convex according to the following definitions, which follow \cite{GuH}.

\begin{definition}
A {\bf convex toric domain\/} is a toric domain $X_\Omega$ such that $\widehat{\Omega}$ is a convex domain in $\R^n$.
\end{definition}
\begin{definition}
A {\bf concave toric domain\/} is a toric domain $X_\Omega$ where $\Omega\subset [0,\infty)^n$ is a domain and  $[0,\infty)^n\setminus\Omega$ is convex.
\end{definition}

\begin{ex}
If $n=2$, a convex or concave toric domain $X_{\Omega}$ arises from a ``region under a graph'' $\Omega=\{(x,y)|0\leq x\leq a,0\leq y\leq f(x)\}$ where $f\co [0,a]\to[0,\infty)$ is a monotone decreasing function.  The corresponding toric domain $X_{\Omega}$ is convex if $f$ is concave, and is concave if $f$ is convex and $f(a)=0$.
\end{ex}

\begin{ex}
If $a_1,\ldots,a_n>0$, the \emph{ellipsoid} $E(a_1,\ldots,a_n)$ is defined as $X_{\Omega}$ where $\Omega=\big\{(x_1,\ldots,x_n)\in[0,\infty)^n|\sum_{i=1}^{n}\frac{x_i}{a_i}\leq 1\big\}$. As a special case, the ball of capacity $a$ is $B^{2n}(a)=E(a,\ldots,a)$.  Note that $E(a_1,\ldots,a_n)$ is both a concave toric domain and a convex toric domain.  We will occasionally find it convenient to extend this to the case that some $a_i=0$ by taking $E(\ldots,0,\ldots)=\varnothing$.
\end{ex}

\begin{ex}
If $a_1,\ldots,a_n>0$, the \emph{polydisk} $P(a_1,\ldots,a_n)$ is defined as $X_{\Omega}$ where $\Omega=\{(x_1,\ldots,x_n)\in [0,\infty)^n\,|\,(\forall i)(0\leq x_i\leq a_i)\}$.  Equivalently, $P(a_1,\ldots,a_n)=B^2(a_1)\times\cdots\times B^2(a_n)$.  Polydisks are convex toric domains.
\end{ex}

We use the following standard notational convention:
\begin{dfn}
If $A\subset \C^n$ and $\alpha>0$, we define $\alpha A=\{\sqrt{\alpha}a|a\in A\}$.
\end{dfn}

(The square root ensures that any capacity $c$ will obey $c(\alpha A)=\alpha c(A)$, and also that we have $E(\alpha a_1,\ldots,\alpha a_n)=\alpha E(a_1,\ldots,a_n)$ and similarly for polydisks.)

For any subset $B\subset\C^n$ let $B^{\circ}$ denote the interior of $B$.
This paper is largely concerned with symplectic embeddings $X\hookrightarrow \alpha X^{\circ}$ where $X$ is a concave or convex toric domain and $\alpha >1$.  The definitions imply that  concave or convex toric domains $X$ always satisfy $X\subset \alpha X^{\circ}$ for all $\alpha>1$ (see Proposition \ref{cvxstar}),  so one such embedding is given by the inclusion of $X$ into  $\alpha X^{\circ}$. However we will find that in many cases there are other such embeddings that are inequivalent to the inclusion in the following sense:

\begin{dfn}\label{knotdef} Let $A\subset U\subset\C^n$, with $A$ closed and $U$ open, and let $\phi\co A\to U$ be a symplectic embedding.\footnote{Since $A$ may not be a manifold or even a manifold with boundary we should say what it means for $\phi\co A\to U$ to be a symplectic embedding; our convention will be that it means that there is an open neighborhood of $A$ to which $\phi$ extends as a symplectic embedding.  When $A$ is a manifold with boundary it is not hard to see using a relative Moser argument that this is equivalent to the statement that $\phi\co A\to U$ is a smooth embedding of manifolds with boundary which preserves the symplectic form.}  We say that $\phi$ is \textbf{unknotted} if there is a symplectomorphism $\Psi\co U\to U$ such that $\Psi(A)=\phi(A)$.  We say that $\phi$ is \textbf{knotted} if it is not unknotted.
\end{dfn}

Note that we do not require the map $\Psi$ to be compactly supported, or Hamiltonian isotopic to the identity, or even to extend continuously to the closure of $U$; accordingly our definition of knottedness is in principle more restrictive than others that one might use.

In Section \ref{4proof} (based on results from Sections \ref{sh} and \ref{construct}) we will prove the existence of knotted embeddings from $X$ to $\alpha X^{\circ}$ for many toric domains $X\subset \C^2$ and suitable $\alpha>1$.  

\begin{theorem}\label{4d}  Let $X\subset \C^2$ belong to any of the following classes of domains:
\begin{itemize}\item[(i)] All convex toric domains $X$ such that, for some $c>0$, $B^4(c)\subsetneq X\subset P(c,c)$.
\item[(ii)] All concave toric domains $X_{\Omega}$ such that, for some $c>0$, \[ \{(x,y)\in [0,\infty)^2|\min\{2x+y,x+2y\}\leq c\}\subset \Omega\subsetneq\{(x,y)\in [0,\infty)^2|x+y\leq c\}.\]
\item[(iii)] All complex $\ell^p$ balls $\{(w,z)\in \C^2||w|^p+|z|^p\leq r^p\}$ for $p>\frac{\log 9}{\log 6}\approx 1.23$, except for $p=2$.
\item[(iv)] All polydisks $P(a,b)$ for $a\leq b<2a$.
\end{itemize}  Then there exist $\alpha >1$ and a knotted embedding $\phi\co X\to \alpha X^{\circ}$.
\end{theorem}

\begin{figure}\includegraphics[width=\textwidth]{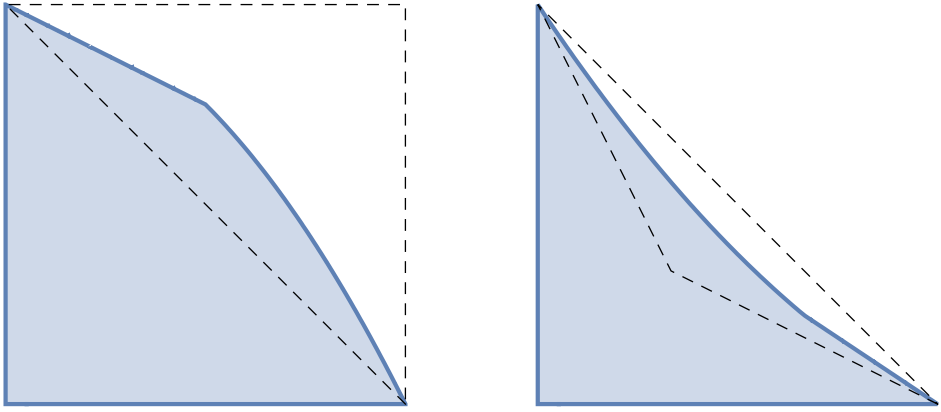}
\caption{The shaded regions are examples of choices of $\Omega$ such that Theorem \ref{4d} gives knotted embeddings $X_{\Omega}\to\alpha X_{\Omega}^{\circ}$ for suitable $\alpha>1$.  The dashed lines delimit the regions which are assumed to contain $(\partial\Omega)\cap (0,\infty)^2$ in, respectively, Cases (i) and (ii) of the theorem.}
\label{exfigure}
\end{figure}

For context, recall that McDuff showed in \cite{M91} that the space of symplectic embeddings from one four-dimensional ball to another is always connected; by the symplectic isotopy extension theorem this implies that symplectic embeddings $B^4(c)\to\alpha B^4(c)^{\circ}$ can never be knotted. (In particular the exclusion of $B^4(c)$ from each of the classes (i),(ii),(iii) above is necessary.) McDuff's result was later extended to establish the connectedness of the space of embeddings of one four-dimensional ellipsoid into another \cite{M} or of a four-dimensional concave toric domain into a convex toric domain \cite{CG}.  So Theorem \ref{4d} reflects that embeddings from concave toric domains into concave ones, or convex toric domains into convex ones, can behave differently than embeddings from concave toric domains into convex ones.

We do not know whether the bound $b<2a$ in part (iv) of Theorem \ref{4d} is sharp.  The bound $p>\frac{\log 9}{\log 6}$ in part (iii) is not sharp; we are aware of extensions of our methods that lower this bound slightly, though in the interest of brevity we do not include them.  Note that the domains in part (iii) are concave when $p<2$ and convex when $p>2$ (in the latter case the result follows directly from part (i)).  

While our primary focus in this paper is on domains in $\mathbb{R}^4$, we show in Theorem \ref{anyd} that the embeddings from Cases (i) and (iv) of Theorem \ref{4d} remain knotted after being trivially extended to the product of $X_{\Omega}$ with an ellipsoid of sufficiently large Gromov width. It remains an  interesting problem to find knotted embeddings involving broader classes of high-dimensional domains that do not arise from lower-dimensional constructions.

By the way, embeddings such as those in Theorem \ref{4d} can only be knotted for a limited range of $\alpha$, since the extension-after-restriction principle \cite[Proposition A.1]{S} implies that for any compact set $X\subset \C^n$ which is star-shaped with respect to the origin and contains the origin in its interior and any symplectic embedding $\phi\co X\to \C^n$, there is $\alpha_0>1$ such that $\phi(X)\subset \alpha_0 X^{\circ}$ and such that $\phi$ is unknotted when considered as a map to $\alpha X^{\circ}$ for all $\alpha\geq \alpha_0$.  The values for $\alpha$ that we find in the proof of Theorem \ref{4d} vary from case to case, but in each instance lie between $1$ and $2$. This suggests the:
\begin{question}
Do there exist a domain $X\subset \R^{2n}$, a number $\alpha>2$, and a knotted symplectic embedding $\phi\co X\to \alpha X^{\circ}$?
\end{question}

Theorem \ref{4d} concerns embeddings of a domain $X$ into the interior of a dilate $\alpha X^{\circ}$ of $X$; of course it is also natural to consider embeddings in which the source and target are not simply related by a dilation.  Our methods in principle allow for this, though the proofs that the embeddings are knotted become more subtle.  In Section \ref{polysect} we carry this out for embeddings of four-dimensional polydisks into other polydisks, and in particular we prove the following as Corollary \ref{allpoly}:

\begin{theorem} \label{polythm} Given any $y\geq 1$, there exist polydisks $P(a,b)$ and $P(c,d)$ and knotted embeddings of $P(a,b)$ into $P(1,y)^{\circ}$ and of $P(1,y)$ into $P(c,d)^{\circ}$.
\end{theorem}

Theorem \ref{polythm} and Case (iv) of Theorem \ref{4d}
should be compared to \cite[Section 3.3]{FHW}, in which it is shown that, if $a\leq b<c$ but $a+b>c$, then the embeddings $\phi_1,\phi_2\co P(a,b)\to P(c,c)^{\circ}$ given by $\phi_1(w,z)=(w,z)$ and $\phi_2(w,z)=(z,w)$ are not isotopic through compactly supported symplectomorphisms of $P(c,c)^{\circ}$. However our embeddings are different from these; in fact the embeddings from \cite{FHW} are not even knotted in our (rather strong) sense since there is a symplectomorphism of the open polydisk $P(c,c)^{\circ}$ mapping $P(a,b)$ to $P(b,a)$.  If one instead considers embeddings into $P(c,d)$ with $c< d$ chosen such that $P(c,d)^{\circ}$ contains both $P(a,b)$ and $P(b,a)$ and $a+b>d$, then $P(a,b)$ and $P(b,a)$ are inequivalent to each other under the symplectomorphism group of $P(c,d)^{\circ}$.  However in situations where this construction and the  construction underlying  Theorem \ref{4d} (iv) and Theorem \ref{polythm} both apply, our knotted  embeddings represent different knot types than both $P(a,b)$ and $P(b,a)$, see Remark \ref{flip}. 

Let us be a bit more specific about how we prove Theorem \ref{4d}; the proof of Theorem \ref{polythm} is conceptually similar.  The knotted embeddings $\phi\co X\to\alpha X^{\circ}$ described in Theorem \ref{4d} are obtained as compositions of embeddings $X\to E\to\alpha X^{\circ}$ where $E$ is an ellipsoid.  In the cases that $X$ is convex, the first map $X\to E$ is just an inclusion, while the second map $E\to\alpha X^{\circ}$ is obtained by using recent developments from \cite{M},\cite{CG} that ultimately have their roots in Taubes-Seiberg-Witten theory, see Section \ref{construct}. (For a limited class of convex toric domains $X$ that are close to a cube $P(c,c)$, we provide a much more elementary and explicit construction in Section \ref{explicitsection}.)  In the cases that $X$ is concave the reverse is true: $E\to\alpha X^{\circ}$ is an inclusion while $X\to E$ is obtained from these more recent methods.  Meanwhile, we use the properties of transfer maps in filtered  $S^1$-equivariant symplectic homology to obtain a lower bound on possible values $\alpha$ such that there can exist any unknotted embedding $X\to \alpha X^{\circ}$ which factors through an ellipsoid $E$.  In each case in Theorem \ref{4d}, we will find compositions $X\to E\to\alpha X^{\circ}$ arising from the constructions in Section \ref{construct} for which $\alpha$ is less than this symplectic-homology-derived lower bound, leading to the conclusion that the composition must be knotted. Figure \ref{L5} and its caption explain this more concretely in a representative special case. 

\begin{figure} 
\includegraphics[width=\textwidth]{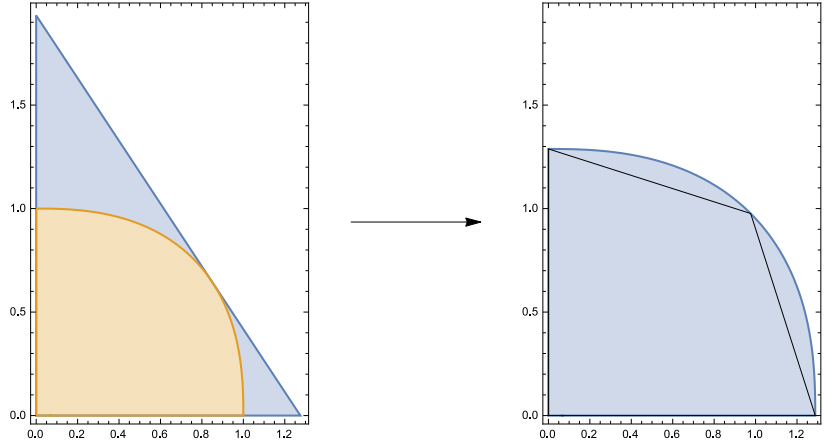}
\caption{The strategy underlying our knotted embedding in the case that $X$ is the $\ell^5$ ball of capacity $1$, as in Case (i) or (iii) of Theorem \ref{4d}.  $X$ is the toric domain associated to the smaller region on the left; the toric domain associated to the triangle on the left is the ellipsoid $E=E((3/2)^{3/5},3^{3/5})$, which in particular contains $X$. The larger region at right is obtained by dilating $X$ by $\alpha=(1+\ep)(3/2)^{3/5}$ for a small $\ep>0$, and Proposition \ref{cvxaxy} shows that there is a symplectic embedding $\phi\co E\to\alpha X^{\circ}$ (in fact, $\phi$ has image contained in the preimage under $\mu$ of the inscribed quadrilateral on the right).  Our knotted embedding is $\phi|_X$; Theorem \ref{dellubounds}(a) implies that any unknotted embedding $X\to\alpha X^{\circ}$ that extends to a symplectic embedding  $E\to\alpha X^{\circ}$ would have $\alpha\geq 2^{3/5}$, whereas in this construction $\alpha$ can be taken arbitrarily close to $(3/2)^{3/5}$.}
\label{L5}
\end{figure}

To carry this out systematically, let us introduce the following two quantities associated to a star-shaped domain $X\subset\C^n$, where the symbol $\hookrightarrow$ always denotes a symplectic embedding: \begin{equation}\label{delldef}
\dell(X)=\inf\{\alpha\geq 1|(\exists a_1,\ldots,a_n)(X\hookrightarrow E(a_1,\ldots,a_n)\hookrightarrow \alpha X^{\circ})\} \end{equation} and \begin{equation}\label{delludef} \dellu(X)=\inf\left\{\alpha\geq 1\left|\begin{array}{c}(\exists a_1,\ldots,a_n,\,f\co X\hookrightarrow E(a_1,\ldots,a_n),\\ g\co E(a_1,\ldots,a_n)\hookrightarrow \alpha X^{\circ})(g\circ f\mbox{ is unknotted.})\end{array}\right.\right\} \end{equation} (The $u$ in $\dellu$ stands for ``unknotted.'')  To put this into a different context, as was suggested to us by Y. Ostrover and L. Polterovich, one can define a pseudometric on the space of star-shaped domains in $\mathbb{C}^n$ by declaring the distance between two domains $X$ and $Y$ to be the logarithm of the infimal $\alpha\in\R$ such that there is a sequence of symplectic embeddings $\alpha^{-1/2}X\hookrightarrow Y\hookrightarrow\alpha^{1/2}X^{\circ}$; a more refined version of this pseudometric would additionally ask that neither of the resulting compositions $X\to\alpha X^{\circ}$ and $Y\to\alpha Y^{\circ}$ be knotted.  Then (at least if $n=2$) the logarithm of $\dell(X)$ or of $\dellu(X)$ is the distance from $X$ to the set of ellipsoids with respect to such a pseudometric.  (In the case of $\dellu$ this statement depends partly on the result from \cite{M} that when $E$ is an ellipsoid in $\mathbb{R}^4$ a symplectic embedding $E\hookrightarrow\alpha E^{\circ}$ is never knotted.)

We will prove Theorem \ref{4d} by proving, for each $X$ as in the statement, a strict inequality $\dell(X)<\dellu(X)$.  This entails finding upper bounds for $\dell(X)$ by exhibiting particular compositions of embeddings $X\hookrightarrow E\hookrightarrow \alpha X^{\circ}$, and finding lower bounds for $\dellu(X)$ using filtered positive $S^1$-equivariant symplectic homology. As it happens, for convex or concave toric domains both our upper bounds and our lower bounds can be conveniently expressed in terms of the following notation:

\begin{notation}
For a domain $\Omega\subset [0,\infty)^n$ we define functions $\|\cdot\|^{*}_{\Omega}$ and $[\,\cdot\,]_{\Omega}$ from $\R^n$ to $\R$ as follows:
\begin{itemize}
\item For $\vec{\alpha}\in \R^n$, $\|\vec{\alpha}\|^{*}_{\Omega}=\sup\{\vec{\alpha}\cdot\vec{v}\,|\,\vec{v}\in\Omega\}$.
\item For $\vec{\alpha}\in \R^n$, $[\vec{\alpha}]_{\Omega}=\inf\{\vec{\alpha}\cdot\vec{v}\,|\,\vec{v}\in[0,\infty)^n\setminus\Omega\}$.
\end{itemize}
\end{notation}

The estimates for $\dellu$ that are relevant to Theorem \ref{4d} are given by the following result, proven in Section \ref{sh}:
\begin{theorem} \label{dellubounds} (a) If $X_{\Omega}\subset \C^2$ is a convex toric domain, then \[ \dellu(X_{\Omega})\geq \frac{\|(1,1)\|^{*}_{\Omega}}{\max\{\|(1,0)\|^{*}_{\Omega},\|(0,1)\|^{*}_{\Omega}\}}.\]

(b) If $X_{\Omega}\subset \C^2$ is a concave toric domain, then 
\[ \dellu(X_{\Omega})\geq \frac{\min\{[(2,1)]_{\Omega},[(1,2)]_{\Omega}\}}{[(1,1)]_{\Omega}}.\]
\end{theorem}

As for upper bounds on $\dell$, in Section \ref{dellproof} we prove the following:
\begin{theorem}\label{dellbounds}
(a) Suppose that $\Omega\subset[0,\infty)^2$ is a domain such that $\hat{\Omega}$ is convex and such that $\Omega$ contains  points $(a,0),(0,b),(x,y)$ with $0<x\leq  a\leq b\leq x+y$.  Then \[ \dell(X_{\Omega})\leq \left\|\left(\frac{1}{a},\frac{1}{x+y}\right)\right\|^{*}_{\Omega}.\]

(b) Suppose that $\Omega\subset [0,\infty)^2$ is a domain that contains  $(0,0)$ in its interior and whose complement in $[0,\infty)^2$ is convex, and such that points $(a,0),(0,b),(x,y)$ with $0<x+y\leq a\leq b$ all belong to $\overline{[0,\infty)^2\setminus\Omega}$.  Then \[ \dell(X_{\Omega})\leq \frac{1}{\left[\left(\frac{1}{b},\frac{1}{x+y}\right)\right]_{\Omega}}.\]

(c) For a polydisk $P(a,b)$ with $a\leq b\leq 2a$ we have \[
\dell(P(a,b))\leq \left\|\left(\frac{3}{a+b},\frac{1}{2a+b}\right)\right\|^{*}_{[0,a]\times[0,b]}.
\]
\end{theorem}

Assuming Theorems \ref{dellubounds} and \ref{dellbounds} for the time being, we now show how they lead to Theorem \ref{4d}.

\subsection{Proof of Theorem \ref{4d}} \label{4proof}

In each of the four cases it suffices to prove a strict inequality $\dell(X)<\dellu(X)$.  

First let $X=X_{\Omega}$ be a convex toric domain with $B^4(c)\subsetneq X_{\Omega}\subset P(c,c)$.  Thus  $\Omega\subset [0,c]\times [0,c]$ (since $X_{\Omega}\subset P(c,c)$), and $\Omega$ is a convex region containing the points $(c,0),(0,c),$ and (due to the strict inclusion $B^4(c)\subsetneq X_{\Omega}$) some point $(x_0,y_0)$ having $x_0+y_0>c$.  The fact that $(c,0),(0,c)\in\Omega\subset [0,c]\times[0,c]$ implies that $\|(1,0)\|^{*}_{\Omega}=\|(0,1)\|^{*}_{\Omega}=c$.  Consequently by Theorem \ref{dellubounds}(a), \[ \dellu(X_{\Omega})\geq \frac{1}{c}\|(1,1)\|^{*}_{\Omega}.\]  Meanwhile Theorem \ref{dellbounds}(a) gives \[ \dell(X_{\Omega})\leq \left\|\left(\frac{1}{c},\frac{1}{x_0+y_0}\right)\right\|^{*}_{\Omega}=\frac{1}{c}\left\|\left(1,\frac{c}{x_0+y_0}\right)\right\|^{*}_{\Omega}.\]  So to prove that $\dell(X_{\Omega})<\dellu(X_{\Omega})$ it suffices to show that $\|(1,1)\|^{*}_{\Omega}>\|(1,a)\|^{*}_{\Omega}$ where $a:=\frac{c}{x_0+y_0}<1$.  Choose $(v_1,v_2)\in \Omega$ such that $v_1+av_2=(1,a)\cdot(v_1,v_2)=\|(1,a)\|^{*}_{\Omega}$; it suffices to find $(w_1,w_2)\in \Omega$ with $(1,1)\cdot(w_1,w_2)=w_1+w_2>v_1+av_2$.  Bearing in mind that $(v_1,v_2)\in\Omega\subset[0,\infty)^2$ and $a<1$, if $v_2\neq 0$ we can simply take $(w_1,w_2)=(v_1,v_2)$.  On the other hand if $v_2=0$ then since $\Omega\subset [0,c]\times [0,c]$ we have $v_1+av_2\leq c$, so taking $(w_1,w_2)=(x_0,y_0)$ gives $w_1+w_2> c\geq v_1+av_2$.  So in any case $\|(1,1)\|^{*}_{\Omega}\geq w_1+w_2>v_1+av_2=\|(1,a)\|^{*}_{\Omega}$, proving that $\dellu(X_{\Omega})>\dell(X_{\Omega})$ and thus completing the proof of Case (i) of Theorem \ref{4d}.

Case (ii) is rather similar.  The hypothesis implies that all points $(x,y)$ of $[0,\infty)^2\setminus\Omega$ have $\min\{2x+y,x+2y\}\geq c$ and so Theorem \ref{dellubounds}(b) yields $\dellu(X_{\Omega})\geq \frac{c}{[(1,1)]_{\Omega}}$.  The hypothesis also implies that $\overline{[0,\infty)^2\setminus\Omega}$  contains a point $(x_0,y_0)$ with $x_0+y_0<c$, and also contains the points $(c,0)$ and $(0,c)$, so by Theorem \ref{dellbounds}(b) we have \[ \dell(X_{\Omega})\leq \frac{1}{\left[\left(\frac{1}{c},\frac{1}{x_0+y_0}\right)\right]_{\Omega}}=\frac{c}{\left[\left(1,\frac{c}{x_0+y_0}\right)\right]_{\Omega}}.\]  So to show that $\dell(X_{\Omega})<\dellu(X_{\Omega})$ it suffices to show that $[(1,b)]_{\Omega}>[(1,1)]_{\Omega}$ where $b:=\frac{c}{x_0+y_0}>1$.  This is established in basically the same way as the similar inequality in Case (i): let $(v_1,v_2)\in \overline{[0,\infty)^2\setminus\Omega}$  minimize $\vec{v}\mapsto (1,b)\cdot \vec{v}$.  Then either $v_2\neq 0$, in which case $(1,1)\cdot(v_1,v_2)<(1,b)\cdot (v_1,v_2)$, or else $v_2= 0$, in which case $v_1=c$ by our assumptions on $\Omega$, and so $[(1,1)]_{\Omega}\leq x_0+y_0<c=(1,b)\cdot(v_1,v_2)$.  So in Case (ii) we indeed have $\dell(X_{\Omega})<\dellu(X_{\Omega})$. 

We now turn to Case (iii) concerning complex $\ell^p$ balls $X=\{(w,z)\in\C^2\,|\,|w|^p+|z|^p\leq r^p\}$.  Using appropriate rescalings it suffices to prove the result in the case that $r=\frac{1}{\sqrt{\pi}}$, so that $X=X_{\Omega}$ where $\Omega=\{(x,y)\in[0,\infty)^2\,|\,x^{p/2}+y^{p/2}\leq 1\}$.  When $p>2$, $X_{\Omega}$ is a convex toric domain contained in $P(1,1)$ and strictly containing $B^4(1)$, so the result follows from Case (i).  From now on assume that $0<p<2$, so that $X_{\Omega}$ is a concave toric domain.  Since $p/2<1$, the reverse H\"older inequality (and the fact that it is sharp) implies that for any $(v,w)\in [0,\infty)^2$ we have $[(v,w)]_{\Omega}=(v^q+w^q)^{1/q}$ where $q=\frac{p}{p-2}<0$.  So from Theorem \ref{dellubounds}(b) we obtain \[ \dellu(X)\geq \frac{(2^q+1)^{1/q}}{2^{1/q}}=\left(2^{q-1}+\frac{1}{2}\right)^{-1/|q|}.\]  Meanwhile $\overline{[0,\infty)^2\setminus\Omega}$  contains the points $(0,1),(1,0),(2^{-2/p},2^{-2/p})$, so Theorem \ref{dellbounds}(b) yields \[ \dell(X)\leq \frac{1}{\left(1+\left(\frac{1}{2^{1-2/p}}\right)^{q}\right)^{1/q}}=\left(1+\frac{1}{2}\right)^{-1/q}=\left(\frac{2}{3}\right)^{-1/|q|}. \]  So we will have $\dell(X)<\dellu(X)$ provided that $2^{q-1}+\frac{1}{2}<\frac{2}{3}$, where $q=\frac{p}{p-2}$.  This condition is equivalent to $2^q<\frac{1}{3}$, \emph{i.e.}, $\frac{p}{2-p}>\frac{\log 3}{\log 2}$, \emph{i.e.}, $p>\frac{\log 9}{\log 6}$.

Turning finally to Case (iv), let $X=P(a,b)=X_{\Omega}$ where $\Omega=[0,a]\times [0,b]$ and we assume that $a\leq b< 2a$.  Clearly for $(v,w)\in [0,\infty)^2$ we have $\|(v,w)\|^{*}_{\Omega}=av+bw$.  Hence Theorem \ref{dellubounds}(a) gives \[ \dellu(P(a,b))\geq \frac{a+b}{b} \] while Theorem \ref{dellbounds}(c) gives \[ \dell(P(a,b))\leq \frac{3a}{a+b}+\frac{b}{2a+b}.\]  In other words, writing $s=\frac{b}{a}$, we have $\dellu(P(a,b))\geq 1+\frac{1}{s}$ and $\dell(P(a,b))\leq 1+\frac{4+s}{2+3s+s^2}$.  So $\dell(P(a,b))<\dellu(P(a,b))$ provided that $4s+s^2<2+3s+s^2$, \emph{i.e.} provided that $\frac{b}{a}=s<2$, as we have assumed.    \qed

\subsection{Organization of the paper} The following Section \ref{sh} will recall some facts about $S^1$-equivariant symplectic homology and extend these using an inverse limit construction to open subsets of $\R^{2n}$ in order to prove Theorem \ref{dellubounds}, which is the key to showing that our embeddings are indeed knotted.  The point of the argument, roughly speaking, is that the filtered positive $S^1$-equivariant symplectic homology of an ellipsoid $E$ is ``as simple as possible'' given the total (unfiltered) homology, while that of the domains in Theorem \ref{4d} has additional features in the form of elements that persist over certain finite action intervals before disappearing in the total homology.  The ratios of the endpoints of these intervals are related to the bounds that we prove on the quantity $\dellu$ in Theorem \ref{dellubounds}.     We also show that our knotted embeddings remain knotted in certain products in Section \ref{prodsect}. 

The embeddings appearing in our main results are constructed in Section \ref{construct} using methods derived from Taubes-Seiberg-Witten theory in work of McDuff \cite{M} and Cristofaro-Gardiner \cite{CG}.  While these sophisticated methods seem to be necessary to obtain results as broad as Theorems \ref{4d} and \ref{polythm}, we show in Section \ref{explicitsection} that for certain domains that are close to a cube the embeddings can be obtained by much more elementary methods, leading to explicit formulas which we provide.  Section \ref{polysect} extends the work in the rest of the paper to obtain the knotted polydisks from Theorem \ref{polythm}.

The appendix contains a proof of a lemma concerning filtered positive $S^1$-equivariant symplectic homology, showing that it can be identified as the filtered homology of a certain filtered complex generated by good Reeb orbits.  This lemma probably will not surprise experts (in particular it was anticipated in \cite[Remark 3.2]{GuH}), and is similar to \cite[Proposition 3.3]{GG}, but we have not seen full details of a proof of a result as sharp as this one elsewhere.

\subsection{Acknowledgements}
This work grew out of our consideration of a question of Yaron Ostrover and Leonid Polterovich.
We are grateful to Richard Hind, Mark McLean, Yaron Ostrover, Leonid Polterovich, and Felix Schlenk for very useful discussions at various stages of this project.
The work was partially supported by the NSF through grant DMS-1509213 and by an AMS-Simons travel grant.

\section{Obstructions to unknottedness from filtered positive $S^1$-equivariant symplectic homology}\label{sh}

The goal of this section is to prove Theorem \ref{dellubounds}, which gives lower bounds on the quantity $\dellu$ defined in (\ref{delludef}).
The main tool for proving this theorem is the positive $S^1$-equivariant symplectic homology which was introduced by Viterbo \cite{V} and developed by Bourgeois and Oancea \cite{bo, BOjta, BOind, BOjems}.
We refer to \cite{bo, BOjta, gutt, GuH} for a precise definition, but describe here some of the key features.

Let $(X,\lambda)$ be a Liouville domain, so that $X$ is a compact smooth manifold with boundary and $\lambda\in \Omega^1(X)$ has the properties that $d\lambda$ is non-degenerate and that $\lambda|_{\partial X}$ is a contact form.  We  say that $(X,\lambda)$ is non-degenerate if the linearized return map of the Reeb flow at each closed Reeb orbit on $\partial X$, acting on the contact hyperplane $\ker\lambda$, does not have $1$ as an eigenvalue.  We will also assume that the first Chern class of $TX$ vanishes on $\pi_2(X)$.

 In this situation, as in \cite{GuH}, for each $L\in \R$ we have an $L$-filtered positive  $S^1$-equivariant symplectic homology, denoted by $CH^L(X,\lambda)$; these are $\mathbb{Q}$-vector spaces that come equipped with maps $\imath_{L_1,L_2}\co CH^{L_1}(X,\lambda)\to CH^{L_2}(X,\lambda)$ for $L_1\leq L_2$ such that $\imath_{L,L}$ is the identity and $\imath_{L_2,L_3}\circ \imath_{L_1,L_2}=\imath_{L_1,L_3}$.\footnote{Warning: In \cite{GuH} the map that we denote by $\imath_{L_1,L_2}$ is denoted by $\imath_{L_2,L_1}$.}  The assumption on $c_1(TX)$ implies that the $CH^L(X,\lambda)$ are $\Z$-graded.
The (unfiltered) positive $S^1$-equivariant symplectic homology of $(X,\lambda)$ is $CH(X,\lambda)=\varinjlim_L CH^L(X,\lambda)$ where the direct limit is constructed using the maps $\imath_{L_1,L_2}$.

The analysis of the spaces $CH^L(X,\lambda)$ is significantly simplified by the following, which is proven in the appendix.  A slightly weaker version for a different version of $S^1$-equivariant symplectic homology is given in \cite[Proposition 3.3]{GG}.
\begin{lemma}\label{complex-exists} Assume as above that $(X,\lambda)$ is a non-degenerate Liouville domain with $c_1(TX)|_{\pi_2(X)}=0$.
There is an $\mathbb{R}$-filtered chain complex $\bigl(CC_*(X,\lambda),\partial\bigr)$, freely generated over $\Q$ by the good\footnote{Recall that a Reeb orbit $\gamma$ is bad if it is an even degree multiple cover of another Reeb orbit $\gamma'$ such that the Conley-Zehnder indices of $\gamma$ and $\gamma'$ have opposite parity. Otherwise, $\gamma$ is good.} Reeb orbits of $\lambda|_{\partial X}$ with the generator corresponding to a Reeb orbit $\gamma$ having filtration level equal to the action $\int_{\gamma}\lambda$ and grading equal to the Conley-Zehnder index of $\gamma$, such that for each $k\in \Z$ and $L\in \R$ the space $CH^{L}_{k}(X,\lambda)$ is isomorphic to the $k$th homology of the subcomplex $CC^{L}_{*}(X,\lambda)$ of $CC_*(X,\lambda)$ consisting of elements with filtration level at most $L$, and such that for $L_1\leq L_2$ the image of the map $\imath_{L_1,L_2}\co CH^{L_1}_{k}(X,\lambda)\to CH^{L_2}_{k}(X,\lambda)$ is isomorphic to the image of the inclusion-induced map $H_k\big(CC^{L_1}_{*}(X,\lambda)\big)\to H_k\big(CC^{L_2}_{*}(X,\lambda)\big)$. 

Moreover, the boundary operator $\partial$ on $CC_*(X,\lambda)$ strictly decreases filtration, in the sense that if $x\in CC^{L}_{*}(X,\lambda)$ then there is $\epsilon>0$ such that $\partial x\in CC^{L-\epsilon}_{*}(X,\lambda)$.
\end{lemma}

\begin{dfn}
A \textbf{tame domain} in $\R^{2n}$ is a $2n$-dimensional Liouville domain $(X,\lambda)$ where: \begin{itemize} \item $X$ is a compact submanifold with boundary of $\mathbb{R}^{2n}$; \item  $d\lambda=\omega_0$, where $\omega_0=\sum_{i=1}^{n}dx_i\wedge dy_i$ is the (restriction of the) standard symplectic form on $\R^{2n}$; and \item the Reeb flow of $\lambda|_{\partial X}$ is non-degenerate. \end{itemize} 
A \textbf{tame star-shaped domain} is a subset $X\subset \R^{2n}$ such that $(X,\lambda_0|_X)$ is a tame domain, where \[ \lambda_0=\frac{1}{2}\sum_i(x_idy_i-y_idx_i).\]
\end{dfn}

Said differently, a tame star-shaped domain is a smooth star-shaped domain such that the radial vector field on $\R^{2n}$ is transverse to the boundary, and such that the characteristic flow on the boundary is non-degenerate.

\begin{remark}
If $U\subset \R^{2n}$ is open and  $\lambda\in \Omega^1(U)$ with $d\lambda=\omega_0$, and if $X\subset U$ has the property that $(X,\lambda|_X)$ is a tame domain, we will typically write $CH^L(X,\lambda)$ instead of $CH^{L}(X,\lambda|_X)$.  It should be noted however that $CH^L(X,\lambda)$ depends only on the restriction of $\lambda$ to $X$.  In fact, more specifically, given that we always assume that $d\lambda=\omega_0$ the only dependence of $CH^L(X,\lambda)$ on $\lambda$ (as opposed to $d\lambda$) arises from the germ of $\lambda|_X$ along $\partial X$; this feature is part of what allows for the construction of transfer maps associated to generalized Liouville embeddings in \cite{GuH}.
\end{remark}

Let $(X,\lambda)$ and $(X',\lambda')$ be two non-degenerate Liouville domains. If $\phi: X\hookrightarrow (X')^{\circ}$ is a symplectic embedding with the property that $(\phi^{\star}\lambda'-\lambda)|_{\partial X}$ is exact\footnote{Such embeddings in general are called ``generalized Liouville embeddings'' of $X$ into $X'$.}, then for all $L\in\R$, there exists a map

\[
	\Phi^{L}_{\phi}:CH^L(X',\lambda') \longrightarrow CH^L(X,\lambda)
\]
called the transfer map.
This map is defined in \cite[Section 8.1]{GuH}.  If $X\subset (X')^{\circ}$, we will simply write $\Phi^L$ for the transfer map induced by the inclusion of $X$ into $X'$.

Such a transfer map $\Phi^{L}_{\phi}$ also exists in the case that, instead of being a generalized Liouville embedding into the interior of $X'$, $\phi\co X\to X'$ is simply an isomorphism of Liouville domains (\emph{i.e.} $\phi$ is a diffeomorphism with $\phi^{\star}\lambda'=\lambda$). In this case $\Phi_{\phi}^{L}$ is an isomorphism.

The transfer map is functorial in the sense that if $(X_1,\lambda_1)$, $(X_2,\lambda_2)$, and $(X_3,\lambda_3)$ are tame domains and if $\phi: X_1\hookrightarrow X_2$ and $\psi: X_2\hookrightarrow X_3$ are either  generalized Liouville embeddings or isomorphisms of Liouville domains, then the following diagram is commutative:
\begin{equation}\label{functor}
	\xymatrix{
		CH^{L}(X_3,\lambda_3) \ar[r]^{\Phi_{\psi}^{L}} \ar@/_2pc/[rr]^{\Phi_{\psi\circ\phi}^{L}} & CH^{L}(X_2,\lambda_2) \ar[r]^{\Phi_{\phi}^{L}} & CH^{L}(X_1,\lambda_1).
	}
\end{equation} (This is proven in the unfiltered context for Liouville embeddings in \cite[Theorem 4.12]{gutt}, and the same argument proves the result in our more general situation.)

Recall that a tame star-shaped domain $W$ by definition has the property that $(W,\lambda_0)$ is a non-degenerate Liouville domain, where $\lambda_0$ is  the standard Liouville primitive $\frac{1}{2}\sum_{i}(x_idy_i-y_idx_i)$, so in this case we obtain graded vector spaces $CH^L(W,\lambda_0)$.  In this case, for any $\zeta>0$, the 
scaled domain $\zeta W=\{\sqrt{\zeta}\vec{x}\,|\,\vec{x}\in W\}$ is likewise a tame domain with respect to $\lambda_0$.  By pulling back the ingredients in the construction
of $CH^{L}(W,\lambda_0)$ by appropriate rescalings, we obtain an identification of $CH^{L}(W,\lambda_0)$ with $CH^{\zeta L}(\zeta W,\lambda_0)$ (on the level of the Reeb orbits that generate the complex $CC_*(W,\lambda_0)$, this 
sends an orbit $\gamma\co S^1\to \R^{2n}$ to the orbit $\sqrt{\zeta}\gamma$, which has the effect of multiplying the action by $\zeta$). We call this isomorphism
$CH^{L}(W,\lambda_0)\cong CH^{\zeta L}(\zeta W,\lambda_0)$ the ``rescaling isomorphism.''  The following gives useful relations between this rescaling isomorphism and the other maps in the theory.

\begin{lemma}\label{lem:commutativetriangle}
	Let $W$ be a tame star-shaped domain,  $\zeta>1$, and $0<s<t$.
	Then the diagrams
\begin{equation}\label{scaletransfer}
		\xymatrix{
			CH^{\zeta^{-1}t}(W,\lambda_0)\ar[rr]^{\imath_{\zeta^{-1}t,t}} \ar[d]^{\textrm{rescaling}}_{\cong} & &CH^{t}(W,\lambda_0)\\
			CH^{t}(\zeta W,\lambda_0)\ar[rru]_{\Phi^t}
		}
\end{equation} and \begin{equation} \label{scaleincl}
\xymatrix{
CH^{s}(W,\lambda_0) \ar[r]^{\imath_{s,t}} &  CH^{t}(W,\lambda_0) \ar[d]_{\textrm{rescaling}}^{\cong} \\ CH^{\zeta s}(\zeta W,\lambda_0)\ar[r]^{\imath_{\zeta s,\zeta t}}  \ar[u]_{\textrm{rescaling}}^{\cong} & CH^{\zeta t}(\zeta W,\lambda_0)
}
\end{equation} are both commutative. 
\end{lemma}

\begin{proof} The commutativity of (\ref{scaleincl}) follows by conjugating the various ingredients involved in the construction of $CH$ by rescalings, see \cite[Lemma 4.15]{gutt}.  The commutativity of (\ref{scaletransfer}) follows from the description of the transfer morphism $\Phi^t\co CH^t(\zeta W,\lambda_0)\to CH^t(W,\lambda_0)$ in \cite[Lemma 4.16]{gutt}; indeed it is shown there that the chain map which induces $\Phi^t$ on filtered homology can be chosen to be the one that sends an orbit near the boundary of $\zeta W$ to its image under the rescaling $\zeta W\to W$, and this correspondence multiplies actions by $\zeta^{-1}$.
   \end{proof}
 
\begin{lemma}\label{lem:commsquare}
	Let $X$ be a tame star-shaped domain. Let $b\geq a>0$.
	Then the following diagram is commutative:
	\[
		\xymatrix{
			CH^L(bX,\lambda_0)\ar[rr]^{\Phi^L}\ar[d]_{\simeq}^{\textrm{rescaling}} && CH^L(aX,\lambda_0)\ar[d]^{\simeq}_{\textrm{rescaling}}\\
			CH^{b^{-1}L}(X,\lambda_0)\ar[rr]_{\imath_{b^{-1}L,a^{-1}L}} && CH^{a^{-1}L}(X,\lambda_0)
		}.
	\]
\end{lemma}

\begin{proof}

Consider the diagram \[ \xymatrix{
CH^{L}(bX,\lambda_0)\ar[rr]^{\Phi^L} & & CH^L(aX,\lambda_0) \ar[dd]_{\cong} \\
 & CH^{\frac{a}{b}L}(aX,\lambda_0) \ar[ul]_{\cong} \ar[ur]^{\imath_{\frac{a}{b}L,L}} 
 &
\\ CH^{b^{-1}L}(X,\lambda_0) \ar[uu]^{\cong}  \ar[ur]_{\cong} \ar[rr]_{\imath_{b^{-1}L,a^{-1}L}} & & CH^{a^{-1}L}(X,\lambda_0) 
}\] where all of the indicated isomorphisms are given by rescaling.  The left triangle commutes trivially, the upper triangle commutes as a special case of (\ref{scaletransfer}), and the lower right quadrilateral commutes as a special case of (\ref{scaleincl}). Hence the entire diagram commutes, which implies the result since the left map is an isomorphism.\end{proof}

\begin{lemma}\label{lem:hamdiffeo}
	Let $X, X'\subset\R^{2n}$ and $\lambda\in \Omega^1(X')$ be such that $X\subset (X')^{\circ}$ and both $(X,\lambda|_X)$ and $(X',\lambda|_{X'})$ are tame domains, and let $\Psi$ be a symplectomorphism between open subsets of  $\R^{2n}$ whose  domain contains $X'$.
	Then the following diagram is commutative:
	\begin{equation} \label{isocomm}
		\xymatrix{
			CH^L(X',\lambda)\ar[r]^{\Phi^L}\ar[d]_{\simeq}^{\Phi^{L}_{\Psi^{-1}}} & CH^L(X,\lambda)\ar[d]^{\simeq}_{\Phi^{L}_{\Psi^{-1}}} \\
			CH^L\bigl( \Psi(X'),{\Psi^{-1}}^{\star}\lambda\bigr)\ar[r]_{\Phi^L}& CH^L\bigl( \Psi(X),{\Psi^{-1}}^{\star}\lambda\bigr)
		}
	\end{equation}
\end{lemma}

\begin{proof}
This is a direct consequence of the functoriality (\ref{functor}): writing $i\co X\to X'$ and $j\co \Psi(X)\to\Psi(X')$ for the inclusions, we have a commutative diagram \[ \xymatrix{ \Psi(X) \ar[r]^{j}\ar[d]^{\Psi^{-1}} & \Psi(X') \ar[d]^{\Psi^{-1}} \\ X\ar[r]_{i} & X'} \] and (\ref{isocomm}) is obtained by taking transfer maps.
\end{proof}

In proving Theorem \ref{dellubounds} it will be helpful to know that the image of the  map $\imath_{L_1,L_2}$ is not too small in certain situations.  The following two lemmas give our first results in this direction.
\begin{lemma}\label{lem:convexebarcode}
	Let $X_\Omega$ be a  convex toric domain in $\C^2$. Then for any $\delta,\eps >0$ there is a tame star-shaped domain $X_{\Omega}^{\delta,\eps}$ such that $(1-\eps)X_{\Omega}\subset X_{\Omega}^{\delta,\eps}\subset X_{\Omega}^{\circ}$ and such that, for any $L_1,L_2$ with \[ \max\left\{\norm{(1,0)}^*_\Omega, \norm{(0,1)}^*_\Omega \right\}+\delta\leq L_1<L_2\leq \norm{(1,1)}^{*}_{\Omega}-\delta,\] the map
	\[
		\imath_{L_1,L_2}\co CH_3^{L_1}(X_{\Omega}^{\delta,\eps},\lambda_0) \too CH_3^{L_2}(X_{\Omega}^{\delta,\eps},\lambda_0)
	\] is an isomorphism of two-dimensional vector spaces.
\end{lemma}

\begin{proof}

The constructions in steps 1, 2, and 3 of \cite[Proof of Lemma 2.5]{GuH} use a Morse-Bott perturbation of a suitable smoothing of $X_{\Omega}$ to obtain a tame star-shaped domain $X_{\Omega}^{\delta,\eps}$ that can be arranged to have the properties that $(1-\eps)X_{\Omega}\subset X_{\Omega}^{\delta,\eps}\subset X_{\Omega}^{\circ}$ and such that the Reeb orbits of $\lambda_0|_{\partial X_{\Omega}^{\delta,\eps}}$ having action at most $\|(1,1)\|_{\Omega}^{*}$ and Conley-Zehnder index at most $4$ consist of:

	\begin{itemize}
	\item no orbits of index $2$;
	\item two orbits  of index $3$, with actions in the intervals $\left(\|(1,0)\|_{\Omega}^{*}-\delta, \|(1,0)\|_{\Omega}^{*}+\delta\right)$ and$\left(\|(0,1)\|_{\Omega}^{*}-\delta, \|(0,1)\|_{\Omega}^{*}+\delta\right)$, respectively; and 
	\item at most one orbit  of index $4$, with action greater than $\|(1,1)\|_{\Omega}^{*}-\delta$.
	\end{itemize} 
	So letting $CC_{*}^{L}(X_{\Omega}^{\delta,\eps},\lambda_0)$ be as in Lemma \ref{complex-exists} (so that in particular $CH^{L}_k(X_{\Omega}^{\delta,\eps},\lambda_0)\cong H_k\bigl(CC_{*}^{L}(X_{\Omega}^{\delta,\eps},\lambda_0),\partial\bigr)$), for any $L$ in  $\left[\max\{\|(1,0)\|_{\Omega}^{*},\|(0,1)\|_{\Omega}^{*}\}+\delta\,,\,\|(1,1)\|_{\Omega}^{*}-\delta\right]$ we have $CC_{2}^{L}(X_{\Omega}^{\delta,\eps},\lambda_0)=CC_{4}^{L}(X_{\Omega}^{\delta,\eps},\lambda_0)=\{0\}$ and $CC_{3}^{L}(X_{\Omega}^{\delta,\eps},\lambda_0)\cong \mathbb{Q}^2$, and moreover if $L_1,L_2$ both lie in this interval with $L_1<L_2$ then the inclusion of complexes $CC_{3}^{L_1}(X_{\Omega}^{\delta,\eps},\lambda_0)\to CC_{3}^{L_2}(X_{\Omega}^{\delta,\eps},\lambda_0)$ is an isomorphism.  So passing to homology shows that, for $\max\{\|(1,0)\|_{\Omega}^{*},\|(0,1)\|_{\Omega}^{*}\}+\delta\leq L_1<L_2\leq \|(1,1)\|_{\Omega}^{*}-\delta$, the inclusion-induced map $\imath_{L_1,L_2}\co CH^{L_1}_{3}(X_{\Omega}^{\delta,\eps},\lambda_0)\to CH^{L_2}_{3}(X_{\Omega}^{\delta,\eps},\lambda_0)$ is an isomorphism of two-dimensional vector spaces.
\end{proof}

\begin{lemma}\label{lem:concavebarcode}
		Let $X_\Omega$ be a concave toric domain in $\C^2$.  Then for any $\delta, \eps >0$ there is a tame star-shaped domain $X_{\Omega}^{\delta,\eps}$ such that $(1-\eps)X_{\Omega}\subset X_{\Omega}^{\delta,\eps}\subset X_{\Omega}^{\circ}$ and such that, if  \[ [(1,1)]_{\Omega}+\delta\leq L_1<L_2\leq \min\{[(1,2)]_{\Omega},[(2,1)]_{\Omega}\}-\delta,\] the map
	\[
\iota_{L_1,L_2}\co CH_4^{L_1}(X_{\Omega}^{\delta,\eps},\lambda_0) \too CH_{4}^{L_2}(X_{\Omega}^{\delta,\eps},\lambda_0) 	\] is an isomorphism of one-dimensional vector spaces.
\end{lemma}
\begin{proof}
We argue analogously to the proof of Lemma \ref{lem:convexebarcode}. By \cite[Proof of Lemma 2.7]{GuH}, there is a tame star-shaped domain $X_{\Omega}^{\delta,\eps}$ such that $(1-\eps)X_{\Omega}\subset X_{\Omega}^{\delta,\eps}\subset X_{\Omega}^{\circ}$ and such that the part of $CC_{*}(X_{\Omega}^{\delta,\eps},\lambda_0)$ of filtration level at most $\max\{[(1,2)]_{\Omega},[(2,1)]_{\Omega}\}$ and degree at most five is generated by:
	\begin{itemize} \item one generator, denoted $a_{1,1}$, in degree $3$, with filtration level in the interval  $\left([(1,1)]_{\Omega}-\delta, [(1,1)]_{\Omega}+\delta\right)$;
	\item one generator, denoted $b_{1,1}$, in degree $4$, with filtration level in the interval  $\left([(1,1)]_{\Omega}-\delta, [(1,1)]_{\Omega}+\delta\right)$; and 
	\item at most two generators $c_{1,2}$ and $c_{2,1}$ in degree $5$, with respective filtration levels in the intervals   $\bigl([(1,2)]_{\Omega}-\delta, [(1,2)]_{\Omega}+\delta\bigr)$ and  $\bigl([(2,1)]_{\Omega}-\delta, [(2,1)]_{\Omega}+\delta\bigr)$.  \end{itemize}
	
	Moreover it is a standard fact (see \emph{e.g.} \cite[Proposition 3.1]{GuH}) that the full degree-$3$ homology $CH_{3}(X_{\Omega}^{\delta,\eps},\lambda_0)$ of this complex is isomorphic to $\mathbb{Q}$; indeed this statement holds for arbitrary tame star-shaped domains in $\R^4$.  Also, \cite[Theorem 1.14]{GuH} shows that a generator for $CH_{3}(X_{\Omega}^{\delta,\eps},\lambda_0)$ is represented by a chain having filtration level at most $[(1,1)]_{\Omega}$.  So since the generator $a_{1,1}$ spans the part of $CC_3(X_{\Omega}^{\delta,\eps},\lambda_0)$ with filtration level at most $\max\{[(1,2)]_{\Omega},[(2,1)]_{\Omega}\}$ (which is greater than $[(1,1)]_{\Omega}$), it follows that $a_{1,1}$ must not be in the image of the boundary operator $\partial$.  Since the boundary operator preserves the filtration, we must then have $\partial b_{1,1}=0$.
	
	Thus for $[(1,1)]_{\Omega}+\delta\leq L\leq \min\{[(1,2)]_{\Omega},[(2,1)]_{\Omega}\}-\delta$, the element $b_{1,1}$ is a degree-four cycle in the subcomplex $CC_{*}^{L}(X_{\Omega}^{\delta,\eps},\lambda_0)$, which is not a boundary for the trivial reason that, for this range of $L$, $CC_{5}^{L}(X_{\Omega}^{\delta,\eps},\lambda_0)=\{0\}$.  Thus $b_{1,1}$ descends to homology to generate the one-dimensional vector space $CH_{4}^{L}(X_{\Omega}^{\delta,\eps},\lambda_0)$ for any such $L$, and the map $\imath_{L_1,L_2}\co CH_{4}^{L_1}(X_{\Omega}^{\delta,\eps},\lambda_0)\to CH_{4}^{L_2}(X_{\Omega}^{\delta,\eps},\lambda_0)$ is an isomorphism whenever $[(1,1)]_{\Omega}+\delta\leq L_1<L_2\leq \min\{[(1,2)]_{\Omega},[(2,1)]_{\Omega}\}-\delta$.
\end{proof}

We are now going to extend the definition of $CH^L$ to open subsets of $\R^{2n}$. This is part of what makes it possible to prove knottedness in the strong sense of Defnition \ref{knotdef}, which considers arbitrary symplectomorphisms of the open set that serves as the codomain for the embedding.  Working with open sets also allows us to consider domains with poorly-behaved boundaries, to which the standard definition of $CH^L$ does not apply.

We continue to denote by $\omega_0$ the standard symplectic form $\sum_{i=1}^{n}dx_i\wedge dy_i$ on open subsets of $\R^{2n}$.
\begin{definition}\label{dfn:open}
	Let $U$ be an open subset of $\R^{2n}$ and let $\lambda\in \Omega^1(U)$ be such that $d\lambda=\omega_0$.
	We define the positive $S^1$-equivariant symplectic homology of $(U,\lambda)$ as
	\begin{equation}\label{eq:defCHopendomains}
		CH^L(U,\lambda) := \lim_{\substack{\longleftarrow\\ (X,\lambda|_X) \textrm{ tame domain}\\ X\subset U\\ }}CH^L(X,\lambda).
	\end{equation}
	Here the inverse limit is taken over transfer maps $\Phi^L$ associated to inclusions.
\end{definition}

Given open sets $U\subset V\subset \R^{2n}$ and $\lambda\in \Omega^1(V)$ with $d\lambda=\omega_0$, we define a transfer map $\Phi^L\co CH^L(V,\lambda)\to CH^L(U,\lambda)$ as the inverse limit of transfer maps $\Phi^L\co CH^L(Y,\lambda)\to CH^L(X,\lambda)$ as $X,Y$ vary through sets such that $(X,\lambda),(Y,\lambda)$ are both tame with $X\subset U\cap Y^{\circ}$ and $Y\subset V$.  This construction will be extended to certain other symplectic embeddings of open subsets (not just inclusions) in Lemma \ref{openfunctor}.

\begin{lemma}\label{closedopen}
        If $X$ is a  tame star-shaped domain and if $L$ is not the action of any periodic Reeb orbit on $\partial X$ then the natural map $CH^L(X,\lambda_0)\to CH^{L}(X^{\circ},\lambda_0)$ is an isomorphism.
\end{lemma}

\begin{proof} The system of tame star-shaped domains $\left\{(1-\eps)X|\,\eps>0\right\}$ is cofinal in the system of all tame star-shaped domains $Y$ with $Y\subset X^{\circ}$, so there is a natural isomorphism \[ CH^L(X^{\circ},\lambda_0)\cong \lim_{\substack{\longleftarrow\\ \eps >0\\ }}CH^L\bigl((1-\eps)X,\lambda_0\bigr). \]

Lemma \ref{lem:commsquare} then induces a natural isomorphism \[ \lim_{\substack{\longleftarrow\\ \eps >0\\ }}CH^L\bigl((1-\eps)X,\lambda_0\bigr)\cong \lim_{\substack{\longleftarrow\\ \eps >0\\ }}CH^{(1-\eps)^{-1}L}(X,\lambda_0)\] where the inverse limit on the right is constructed from the maps $\imath_{s,t}$ that are identified by Lemma \ref{complex-exists} with the maps induced by  inclusions of subcomplexes $CC^{s}_{*}\hookrightarrow CC^{t}_{*}$.  Since $L$ is not  the action of any periodic Reeb orbit on $\partial X$, it follows from Lemma \ref{complex-exists} that the map $\imath_{L,(1-\eps)^{-1}L}: CH^{L}(X,\lambda_0)\to CH^{(1-\eps)^{-1}L}(X,\lambda_0)$ is an isomorphism for all sufficiently small $\eps$, from which the lemma immediately follows. \end{proof}

Let $U$ be an open subset of $\C^n$ and $\lambda\in \Omega^1(U)$ with $d\lambda=\omega_0$, and let $L_1<L_2\in\R$.
We define the map $\imath_{L_1,L_2}:CH^{L_1}(U,\lambda)\to CH^{L_2}(U,\lambda)$ as the inverse limit of the maps $\imath_{L_1,L_2}:CH^{L_1}(X_U,\lambda|_{X_U})\to CH^{L_2}(X_U,\lambda|_{X_U})$ where $(X_U,\lambda|_{X_U})$ is a tame domain, $X_U\subset U$.

Since the inverse limit is a functor from the category of diagrams of abelian groups to the category of abelian groups (see \cite[Application 2.6.7]{weibel}),
we have a similar statement to Lemma \ref{lem:commutativetriangle}:

\begin{lemma}\label{opentriangle}
	Let $U$ be an open set in $\R^{2n}$, let $\zeta>1$, and let $\lambda\in \Omega^1(U)$ with $d\lambda=\omega_0$.
	Then the following diagram is commutative:
	\begin{equation}
		\xymatrix{
			CH^{\zeta^{-1}L}(U,\lambda)\ar[rr]^{\imath_{\zeta^{-1}L,L}} \ar[d]_{\textrm{rescaling}} && CH^{L}(U,\lambda)\\
			CH^{L}(\zeta U,\lambda)\ar[rru]_{\Phi^L}
		}.
	\end{equation}
\end{lemma}

The following calculation related to the maps $\imath_{L_1,L_2}$ will be very helpful.

\begin{lemma}\label{lem:keyconvex} Let $X_{\Omega}$ be a convex toric domain in $\C^2$. \begin{itemize}
    \item[(i)] If $\max\{\|(1,0)\|_{\Omega}^{*},\|(0,1)\|_{\Omega}^{*}\}<L<\|(1,1)\|_{\Omega}^{*}$, then 
$CH^{L}_{3}(X_{\Omega}^{\circ},\lambda_0)$ is a two-dimensional vector space. 
\item[(ii)]  If $\max\{\|(1,0)\|_{\Omega}^{*},\|(0,1)\|_{\Omega}^{*}\}<L_1<L_2<\|(1,1)\|_{\Omega}^{*}$, then $\imath_{L_1,L_2}\co CH^{L_1}_{3}(X_{\Omega}^{\circ},\lambda_0)\to CH^{L_2}_{3}(X_{\Omega}^{\circ},\lambda_0)$ is an isomorphism. 
\end{itemize} 
\end{lemma}

\begin{proof}
Choose $\delta>0$ such that \[ \max\{\|(1,0)\|_{\Omega}^{*},\|(0,1)\|_{\Omega}^{*}\}+2\delta < L,L_1,L_2 < \|(1,1)\|_{\Omega}^{*}-2\delta. \]  
For this fixed value of $\delta$ and varying $\eps>0$, the non-degenerate domains $X_{\Omega}^{\delta,\eps}$ from Lemma \ref{lem:convexebarcode}
 form a cofinal system in the inverse system defining $CH^L(X_{\omega}^{\circ},\lambda_0)$.
Choose a sequence $\eps_m\searrow 0$ such that $X_{\Omega}^{\delta,\eps_m}\subset (X_{\Omega}^{\delta,\eps_{m+1}})^{\circ}$ for each $m$, so we have transfer maps 
$(\Phi^L)_{m}\co CH^{L}_{3}( X_{\Omega}^{\delta,\eps_{m+1}},\lambda_0)\to CH^{L}_{3}(X_{\Omega}^{\delta,\eps_m},\lambda_0)$; this gives a cofinal subsystem within our inverse system.  We claim that these transfer maps are isomorphisms of 
two-dimensional vector spaces once $m$ is sufficiently large (and hence $\eps_m$ is sufficiently small).  

To prove this, we first note that the domain and codomain both have dimension two by Lemma \ref{lem:convexebarcode}, so it is enough to  show that $(\Phi^L)_m$ 
is injective for all large $m$.  But we have inclusions \[ 
(1-\eps_{m})X_{\Omega}^{\delta,\eps_{m+1}}\subset (1-\eps_m)X_{\Omega}^{\circ}\subset (X_{\Omega}^{\delta,\eps_m} )^{\circ}\subset X_{\Omega}^{\delta,\eps_m}  \subset (X_{\Omega}^{\delta,\eps_{m+1}})^{\circ}\] 
and so the transfer map $(\Phi^L)_m$ fits into a sequence of transfer maps \begin{equation}\label{seqtrans} \xymatrix{
CH^{L}_{3}( X_{\Omega}^{\delta,\eps_{m+1}},\lambda_0)\ar[r]^{(\Phi^L)_m} &  CH^{L}_{3}(X_{\Omega}^{\delta,\eps_m},\lambda_0)\ar[r] &  CH^{L}_{3}\bigl((1-\eps_m)X_{\Omega}^{\delta,\eps_{m+1}},\lambda_0\bigr)  }
\end{equation}
whose composition is identified  up to isomorphism by Lemma \ref{lem:commutativetriangle} with the inclusion-induced map \[ \imath_{L,(1-\eps_m)^{-1}L}\co  
CH^{L}_{3}( X_{\Omega}^{\delta,\eps_{m+1}},\lambda_0)\to CH^{(1-\eps_m)^{-1}L}_{3}( X_{\Omega}^{\delta,\eps_{m+1}},\lambda_0).  
\]  Provided that $m$ is chosen so large that $(1-\eps_m)^{-1}L<\|(1,1)\|_{\Omega}^{*}-\delta$,  Lemma \ref{lem:convexebarcode} shows that the above map
$\imath_{L,(1-\eps)^{-1}L}$ is an isomorphism.  Thus for $m$ sufficiently large the first map $(\Phi^L)_m$ in the sequence (\ref{seqtrans}) must be injective, and hence
is also an isomorphism by counting dimensions.

Since the $(\Phi^L)_m$ are all isomorphisms for $m$ sufficiently large, and since they form the structure maps in a cofinal system within the inverse 
system defining $CH^L(X_{\Omega}^{\circ},\lambda_0)$, it follows that the canonical map $CH^L(X_{\Omega}^{\circ},\lambda_0)\to CH^{L}_{3}(X_{\Omega}^{\delta,\eps_m},\lambda_0)$
is an isomorphism for $m$ sufficiently large.  So by Lemma \ref{lem:convexebarcode} $CH^{L}_{3}(X_{\Omega}^{\circ},\lambda_0)$ is two-dimensional, proving statement (i) of the lemma.    Moreover this argument works uniformly for  all $L$ in  the interval from   $\max\{\|(1,0)\|_{\Omega}^{*},\|(0,1)\|_{\Omega}^{*}\}+2\delta$ to   $\|(1,1)\|_{\Omega}^{*}-2\delta$, and in particular for $L=L_1$ or $L=L_2$ where $L_1,L_2$ are as in statement (ii) of the lemma.  So for sufficiently large $m$ we have  a commutative diagram \[ \xymatrix{ 
CH^{L_1}_{3}(X_{\Omega}^{\circ},\lambda_0)\ar[r]^{\imath_{L_1,L_2}} \ar[d] & CH^{L_2}_{3}(X_{\Omega}^{\circ},\lambda_0)\ar[d] \\ CH^{L_1}_{3}(X_{\Omega}^{\delta,\eps_m},\lambda_0) \ar[r]^{\imath_{L_1,L_2}}\ar[r]^{\imath_{L_1,L_2}} &
CH^{L_2}_{3}(X_{\Omega}^{\delta,\eps_m},\lambda_0)
} \] where the vertical arrows are isomorphisms by what we have just shown, and the bottom horizontal arrow is an isomorphism by Lemma \ref{lem:convexebarcode}.  Hence the top horizontal arrow is an isomorphism, proving statement (ii) of the lemma.
\end{proof}

\begin{lemma}\label{lem:keyconcave}
		Let $X_\Omega$ be a concave toric domain in $\C^2$ such that $[(1,1)]_{\Omega}<\max\{[(1,2)]_{\Omega},[(2,1)]_{\Omega}\}$. Then for $[(1,1)]_{\Omega}<L_1<L_2<\min\{[(1,2)]_{\Omega},[(2,1)]_{\Omega}\}$,
	\[
\imath_{L_1,L_2}\co CH_4^{L_1}(X_{\Omega}^{\circ},\lambda_0) \too CH_4^{L_2}(X_{\Omega}^{\circ},\lambda_0)
	\] is an isomorphism of one-dimensional vector spaces.
\end{lemma}
\begin{proof} This follows by the exact same argument as Lemma \ref{lem:keyconvex}, using Lemma \ref{lem:concavebarcode} instead of Lemma \ref{lem:convexebarcode}.
\end{proof}

\begin{remark} In the case that $X_{\Omega}$ is an ellipsoid $E(a,b)$ (and hence in particular is both a concave toric domain and a convex toric domain), Lemmas \ref{lem:keyconvex} and \ref{lem:keyconcave} have no content when applied to $X_{\Omega}$. Indeed in  this case, assuming without loss of generality that $a\leq b$, \[
\|(1,1)\|_{\Omega}^{*}=\max\{\|(1,0)\|_{\Omega}^{*},\|(0,1)\|_{\Omega}^{*}\}=b, \quad [(1,1)]_{\Omega}=\min\{[(1,2)]_{\Omega},[(2,1)]_{\Omega}=a\] and so there are no choices of $L_1,L_2$ that satisfy the hypotheses.  For each of the domains appearing in our main theorem, on the other hand, Lemma \ref{lem:keyconvex} or Lemma \ref{lem:keyconcave} gives important information.
\end{remark}

\begin{prop}\label{prop:open} 
Let $U\subset \R^{2n}$ be a star-shaped open set, and let $\phi\co U\to V$ be a symplectomorphism where $V$ is an open subset of $\R^{2n}$.  Then $\phi$ determines an isomorphism $\Phi_{\phi}^{L}\co CH^{L}(V,{\phi^{-1}}^{\star}\lambda_0)\to CH^{L}(U,\lambda_0)$ such that the diagram
\begin{equation}\label{opencomm} \xymatrix{ CH^{L}(V,{\phi^{-1}}^{\star}\lambda_0)\ar[r]^{\Phi^L}\ar[d]^{\Phi_{\phi}^{L}} & CH^{L}\bigl(\phi(W),{\phi^{-1}}^{\star}\lambda_0\bigr) \ar[d]^{\Phi_{\phi|_W}^{L} }\\ CH^{L}(U,\lambda_0) \ar[r]^{\Phi^L} & CH^{L}(W,\lambda_0) } \end{equation}  commutes when $W\subset U$ is  an open subset.
\end{prop}

\begin{proof}
For $X\subset U$, it is straightforward to see that $(X,\lambda_0)$ is a non-degenerate Liouville domain if and only if $\bigl(\phi(X),{\phi^{-1}}^{\star}\lambda_0\bigr)$ is a non-degenerate Liouville domain.  So in view of Lemma \ref{lem:hamdiffeo}, we obtain an isomorphism of the inverse systems defining $CH^{L}(V,{\phi^{-1}}^{\star}\lambda_0)$ and $CH^L(U,\lambda_0)$.  This induces the desired isomorphism $\Phi_{\phi}^{L}$ between the inverse limits $ CH^{L}(V,{\phi^{-1}}^{\star}\lambda_0)$ and  $CH^{L}(U,\lambda_0)$, and the fact that (\ref{opencomm}) commutes follows by taking inverse limits of the diagrams (\ref{isocomm}) from Lemma \ref{lem:hamdiffeo}.
\end{proof}

\begin{dfn}
Let $U\subset \R^{2n}$ be an open subset and let $\lambda\in \Omega^1(U)$ obey $d\lambda=\omega_0$.  We say that the pair $(U,\lambda)$ is \textbf{tamely exhausted} if for every compact subset $K\subset U$ there is a set $X$ with $K\subset X\subset U$ such that $(X,\lambda)$ is a tame Liouville domain and such that the natural map $H^1(X;\R)\to H^1(\partial X;\R)$ is zero.
\end{dfn}

\begin{ex} \label{ex: exhausted}
In any dimension $m$, let us say that a nonempty compact subset $X\subset \R^{m}$ is \textbf{strictly star-shaped} if for all $x\in X$ and all $t\in [0,1)$ it holds that $tx\in X^{\circ}$.  We claim that if $X\subset \R^{2n}$ is strictly star-shaped then $(X^{\circ},\lambda_0)$ is tamely exhausted.

To see this, first note that for any $a\in S^{2n-1}$ the set $I_a=\{t\geq 0\,|\,ta\in X\}$ is a closed interval of the form $[0,f(a)]$ where $0<f(a)<\infty$.  Indeed $I_a$ contains all sufficiently small positive numbers because the definition implies that $0\in X^{\circ}$, and $I_a$ is closed and bounded because $X$ is compact.  So we can take $f(a)=\sup I_a=\max I_a$; the fact that $I_a$ contains all numbers between $0$ and $f(a)$ is an obvious consequence of the assumption that $X$ is star-shaped.  Moreover we then have $tf(a)a\in X^{\circ}$ for all $t\in [0,1)$.  

So we have defined a function $f\co S^{2n-1}\to (0,\infty)$ with the properties that \[ X=\big\{sa|a\in S^{2n-1},0\leq s \leq f(a)\big\} \] and \[ X^{\circ}=\big\{sa|a\in S^{2n-1},0\leq s <f(a)\big\}.\]  

We will now show that $f$ is continuous.  Let $a\in S^{2n-1}$ and let $\ep>0$ be small enough that $f(a)>\ep$. Then $\big(f(a)-\ep\big)a\in X^{\circ}$, so by considering a small ball around $\big(f(a)-\ep\big)a$ that is contained in $X^{\circ}$ we see that, for $b\in S^{2n-1}$ sufficiently close to $a$, it will hold that  $\big(f(a)-\ep\big)b\in X^{\circ}$ and hence that $f(b)>f(a)-\ep$.  Thus $f$ is lower semi-continuous. To see that $f$ is upper semi-continuous  note that if it were not then we could find $a_k,a\in S^{2n-1}$ with $a_k\to a$ and each $f(a_k)\geq f(a)+\ep$ for some $\ep>0$ independent of $k$. Since $X$ is compact, after passing to a subsequence the $f(a_k)a_k$ would converge to a point of the form $sa$ where both $s>f(a)$ and $sa\in X$, contradicting the defining property of $f$.  So $f$ is indeed continuous.

With this in hand it is not hard to see that our strictly star-shaped domain $X$ is tamely exhausted. Indeed, if $K$ is a compact subset of $X^{\circ}$ then there is $\eps>0$ such that, for all $a\in S^{2n-1}$ and $t\geq 0$ with $ta\in K$, we have $t<f(a)-\eps$.  Choose a $C^{\infty}$ function $g\co S^{2n-1}\to (0,\infty)$ such that, for all $a\in S^{2n-1}$, $f(a)-\eps<g(a)<f(a)$.  Then defining $Y=\big\{ta|a\in S^{2n-1},0\leq t\leq g(a)\big\}$, $Y$ will be a smooth manifold with boundary such that $\lambda_0|_{\partial Y}$ is a contact form and such that $K\subset Y\subset X^{\circ}$.  Possibly after a further perturbation of $g$, the Reeb flow of $\lambda_0|_{\partial Y}$ will be non-degenerate so that $(Y,\lambda_0)$ is tame. Because $Y$ is star-shaped, it obviously has $H^1(Y;\R)=0$.  Since $K$ is an arbitrary compact subset of $X^{\circ}$ this proves our claim that $(X^{\circ},\lambda_0)$ is tamely exhausted.
\end{ex}

\begin{lemma}\label{openfunctor} To each symplectic embedding $\phi\co U\hookrightarrow V$ between open subsets $U,V\subset \R^{2n}$ equipped with one-forms $\lambda,\lambda'$ such that $(U,\lambda)$ and $(V,\lambda')$ are tamely exhausted, we may associate a map $\Phi^{L}_{\phi}\co CH^L(V,\lambda')\to CH^L(U,\lambda)$ such that:
\begin{itemize}
\item[(i)] In the case that $\phi$ is the inclusion of $U$ into $V$ and $\lambda=\lambda'|_U$, $\Phi_{\phi}^{L}$ coincides with the transfer map $\Phi^L\co CH^L(V,\lambda')\to CH^L(U,\lambda')$ described just before  Lemma \ref{closedopen}.
\item[(ii)] If $(U,\lambda),(V,\lambda'),(W,\lambda'')$ are tamely exhausted and if $\phi\co U\hookrightarrow V$ and $\psi\co V\hookrightarrow W$ are symplectic embeddings then we have a commutative diagram \begin{equation}\label{opendiag} \xymatrix{CH^L(W,\lambda'')\ar[rr]^{\Phi^{L}_{\psi\circ\phi}}
\ar[dr]_{\Phi^{L}_{\psi}} & & CH^L(U,\lambda) \\ & CH^L(V,\lambda')\ar[ru]_{\Phi^{L}_{\phi}} & } \end{equation}
\end{itemize}
\end{lemma}

\begin{proof}

Since $(U,\lambda)$ is tamely exhausted, the subsets $X\subset U$ with $(X,\lambda)$ tame and $H^1(X;\R)\to H^1(\partial X;\R)$ zero form a cofinal system in the inverse system defining $CH^L(U,\lambda)$.  So in order to construct $\Phi^{L}_{\phi}$ it suffices to define maps $\Phi^{L}_{VX}\co CH^L(V,\lambda')\to CH^L(X,\lambda)$ for all such $X$ in such a way that the  diagrams \begin{equation}\label{phivx} \xymatrix{ CH^L(V,\lambda')\ar[d]_{\Phi^{L}_{VX}}\ar[dr]^{\Phi^{L}_{VX'}} & \\ CH^L(X,\lambda)\ar[r]_{\Phi^L} & CH^L(X',\lambda) } \end{equation}
 commute for subsets $X,X'\subset U$ as above with $X'\subset X^{\circ}$.

To define $\Phi^{L}_{VX}$, note that the fact that $(V,\lambda')$ is tamely exhausted implies that there is $Y$ with $\phi(X)\subset Y^{\circ}\subset Y\subset V$ such that $(Y,\lambda')$ is tame, and define $\Phi_{VX}$ as a composition $CH^L(V,\lambda')\to CH^L(Y,\lambda')\to CH^L(X,\lambda)$ where the first map is the structure map of the inverse limit and the second map is the transfer map associated to $\phi|_{X}\co X\hookrightarrow Y^{\circ}$.  (The fact that $\phi|_X$ is a generalized Liouville embedding follows from the facts that $\phi$ preserves $\omega_0$ and that $H^1(X;\R)\to H^1(\partial X;\R)$ vanishes.)  

We claim that this map $\Phi_{VX}^{L}$ is independent of the choice of $Y$ involved in its construction.  Indeed if $Y'\subset V$ is another set satisfying the same properties, then the fact that $(V,\lambda')$ is tamely exhausted shows that there is $Z$ such that $Y\cup Y'\subset Z^{\circ}\subset Z\subset V$ and such that $(Z,\lambda')$ is tame.  We can then form a commutative diagram 
\[
    \xymatrix{
        CH^L(V,\lambda')\ar[dr]\ar[drr]\ar[ddr]&&\\
        &CH^L(Z,\lambda')\ar[d]\ar[r]&CH^L(Y,\lambda')\ar[d]\\
        &CH^L(Y',\lambda')\ar[r]&CH^L(X,\lambda)
    }.
\]
Every piece of the above diagram (the square and the two triangles) is commutative by definition of the inverse limit and by functoriality of the transfer map.
Therefore the two compositions $CH^L(V,\lambda')\to CH^L(X,\lambda)$ passing respectively through $CH^L(Y,\lambda')$ and $CH^L(Y',\lambda')$ are equal to each other.

To see that (\ref{phivx}) commutes, notice that, by what we have just shown, we may use the same subdomain $Y\subset V$ in the constructions of $\Phi_{VX}^{L}$ and of $\Phi_{VX'}^{L}$, yielding a commutative diagram \[ \xymatrix{ & CH^L(V,\lambda')\ar[ddl]_{\Phi_{VX}^{L}}\ar[d]\ar[ddr]^{\Phi_{VX'}^{L}} & \\ & CH^L(Y,\lambda')\ar[dl]\ar[dr] & \\ CH^L(X,\lambda)\ar[rr]^{\Phi^L} & & CH^L(X',\lambda)          }\] where the bottom triangle is an instance of (\ref{functor}).  So passing to the inverse limit over $X$ indeed yields our desired map $\Phi_{\phi}^{L}\co CH^L(V,\lambda')\to CH^L(U,\lambda)$.

It remains to show that the various maps $\Phi_{\phi}^{L}$ construced in this way satisfy properties (i) and (ii) in the statement of the lemma.  However, given the  validity of the above construction of $\Phi^{L}_{\phi}$ and the functoriality (\ref{functor}) for transfer maps associated to generalized Liouville embeddings, both of these are straightforward exercises with inverse limits and so we leave them to the reader.
\end{proof}

\begin{cor} \label{rankcor}
Let $X,E,Y\subset \mathbb{R}^{2n}$ be strictly star-shaped domains with $X\subset Y^{\circ}$, and let $f\co X\to E$, $g\co E\to Y^{\circ}$ be symplectic embeddings.  If the composition $g\circ f$ is unknotted, then for all $k\in \Z,L\in \R$ it holds that \[ \op{Rank}\left(\Phi^{L}\co CH^{L}_{k}(Y^{\circ},\lambda_0)\to CH^{L}_{k}(X^{\circ},\lambda_0)\right)\leq \dim CH_{k}^{L}(E^{\circ},\lambda_0).\]
\end{cor}

\begin{proof}

The assumption that $g\circ f$ is unknotted implies that there is a symplectomorphism $\phi\co Y^{\circ}\to Y^{\circ}$ such that $\phi(g(f(X)))=X$.  Example \ref{ex: exhausted} shows that each of $(X^{\circ},\lambda_0),(E^{\circ},\lambda_0)$, and $(Y^{\circ},\lambda_0)$ is tamely exhausted.  It is clear from the definition that if $\psi\co Z\hookrightarrow \R^{2n}$ is a symplectic embedding with image $Z'$, then $(Z,\psi^{\star}\lambda)$ is tamely exhausted if and only if $(Z',\lambda)$ is tamely exhausted.  So since our symplectomorphism $\phi$ maps $Y^{\circ}$ to $Y^{\circ}$ and   $g\big(f(X^{\circ})\big)$ to $X^{\circ}$ it follows that $(Y^{\circ},\phi^{\star}\lambda_0)$ and $\big(g(f(X^{\circ})),\phi^{\star}\lambda_0\big)$ are also tamely exhausted.
Consider the diagram \[ \xymatrix{ &  CH^{L}_{k}(E^{\circ},\lambda_0) \ar[rd]^{\Phi_{(g|_{f(X^{\circ})})^{-1}}^{L}} & \\ CH^{L}_{k}(Y^{\circ},\phi^{\star}\lambda_0)\ar[rr]^{\Phi^L}\ar[ru]^{\Phi_{g}^{L}}\ar[d]^{\Phi_{\phi^{-1}}^{L}} & & CH^{L}_{k}\bigl(g(f(X^{\circ})),\phi^{\star}\lambda_0\bigr)\ar[d]^{\Phi_{\phi^{-1}}^{L}} \\ CH^{L}_{k}(Y^{\circ},\lambda_0)\ar[rr]^{\Phi^L} & & CH^{L}_{k}(X^{\circ},\lambda_0)                   }
\]   We see that the top triangle commutes since it is an instance of (\ref{opendiag}) (as $g\circ (g|_{f(X^{\circ})})^{-1}$ is just the inclusion of $g(f(X^{\circ}))$ into $Y^{\circ}$); the  square commutes by Corollary \ref{prop:open}; and the vertical arrows are isomorphisms. Hence for each $k\in \Z$,
\begin{align*}
    \op{Rank}&\left(\Phi^{L}\co CH^{L}_{k}(Y^{\circ},\lambda_0)\to CH^{L}_{k}(X^{\circ},\lambda_0)\right)=\\
    &\phantom{coucou}\op{Rank}\left(\Phi^L\co CH^{L}_{k}(Y^{\circ},\phi^{\star}\lambda_0)\to CH^{L}_{k}(g(f(X^{\circ})),\phi^{\star}\lambda_0)\right).
\end{align*}
But since $\Phi^L\co CH^{L}_{k}(Y^{\circ},\phi^{\star}\lambda_0)\to CH^{L}_{k}(g(f(X^{\circ})),\phi^{\star}\lambda_0)$ factors through $CH^{L}_{k}(E^{\circ},\lambda_0)$, its rank is at most the dimension of $CH^{L}_{k}(E^{\circ},\lambda_0)$.
 \end{proof}

Throughout the rest of the paper Corollary \ref{rankcor} will be our main tool for showing that embeddings are knotted.  First we need the following to show that it applies to the domains appearing in our main theorems.

\begin{prop}\label{cvxstar}
Let $X$ be either a convex toric domain or a concave toric domain in $\mathbb{R}^{2n}$.  Then $X$ is strictly star-shaped.
\end{prop}

\begin{proof}
First suppose that $X=X_{\Omega}$ is a convex toric domain; thus $\Omega\subset [0,\infty)^n$ has the property that $\hat{\Omega}$ (as defined in (\ref{omegahat})) is a convex domain in $\R^{n}$.  It is easy to see that $X_{\Omega}$ is strictly star-shaped if and only if $\hat{\Omega}$ is strictly star-shaped.  

Let us re-emphasize that ``domains'' are by definition closures of bounded open sets.  Consequently if $x=(x_1,\ldots,x_n)\in \hat{\Omega}$ and $0<\ep<1$, we can find $(y_1,\ldots,y_n)\in \hat{\Omega}^{\circ}$ such that $\frac{y_i}{x_i}>1-\ep$ for all $i$ such that $x_i\neq 0$.  Now $\hat{\Omega}^{\circ}$ is convex and is invariant under reversal of the sign of any subset of the coordinates of $\R^n$, so it follows that $\big\{(z_1,\ldots,z_n)\,|\,|z_i|\leq y_i\big\}\subset \hat{\Omega}^{\circ}$. In particular this implies that $(1-\ep)x\in \hat{\Omega}^{\circ}$.  Since $\ep$ can be taken arbitrarily small this proves that $\hat{\Omega}$ is strictly star-shaped and hence that $X_{\Omega}$ is strictly star-shaped.

Now let us turn to the case that $X=X_{\Omega}$ is a concave toric domain, so that $\Omega\subset [0,\infty)^n$ has the property that $[0,\infty)^n\setminus \Omega$ is convex.  It is easy to see that $X_{\Omega}$ is strictly star-shaped if and only if $\Omega$ has the property that $t\Omega\subset \Omega^{\circ}$ for all $t\in [0,1)$, where the interior $\Omega^{\circ}$ is taken relative to $[0,\infty)^n$.  Suppose for contradiction that $x=(x_1,\ldots,x_n)\in \Omega$ and $tx\notin \Omega^{\circ}$ where $0\leq t<1$.  Then \[ tx\in [0,\infty)^n\setminus \Omega^{\circ}=\overline{[0,\infty)^n\setminus \Omega}.\]  Now $\overline{[0,\infty)^n\setminus \Omega}$ is a convex set which (since $\Omega$ is compact) contains all points sufficiently far from the origin in addition to containing $tx$, in view of which \[ \big\{(y_1,\ldots,y_n)\,|\,y_i\geq tx_i\mbox{ for all }i\big\} \subset \overline{[0,\infty)^n\setminus \Omega}.\]  The set on the left hand side above contains our point $x$ in its interior, so we would have \[ x\in \Omega\cap \left(\overline{[0,\infty)^n\setminus \Omega}\right)^{\circ}.\]  
But \[ \left(\overline{[0,\infty)^n\setminus \Omega}\right)^{\circ}=\left([0,\infty)^n\setminus \Omega^{\circ}\right)^{\circ}=[0,\infty)^n\setminus \overline{\Omega^{\circ}}, \] so we would have $x\in \Omega\cap \left([0,\infty)^n\setminus \overline{\Omega^{\circ}}\right)$, which is impossible since $\Omega$ is the closure of an open subset.
\end{proof}

We now fulfill the main goal of this section by proving Theorem \ref{dellubounds}.

\begin{proof}[Proof of Theorem \ref{dellubounds} (a)]
We will show that $\alpha>\dellu(X)$ implies that $\alpha\geq \frac{\|(1,1)\|^{*}_{\Omega}}{\max\{\|(1,0)\|^{*}_{\Omega},\|(0,1)\|^{*}_{\Omega}}$.  
Let $\alpha>\dellu(X)$.  Then there is an ellipsoid $E$ and embeddings $f\co X\hookrightarrow E$ and $g\co E\hookrightarrow \alpha X^{\circ}$ such that $g\circ f$ is unknotted. By slightly perturbing $E$ we may assume that $E$ is irrational (\emph{i.e.} $E=E(a,b)$ where $\frac{b}{a}\notin \mathbb{Q}$); this ensures that $E$ is a tame star-shaped domain. We will apply Corollary \ref{rankcor} with $k=3$.  Note that, for each $L\in \R$, $\dim CH_{3}^{L}(E^{\circ},\lambda_0)\leq 1$ by Lemma \ref{closedopen} and \cite[Section 3]{BCE}.  So by Corollary \ref{rankcor}, we must have $\op{Rank}(\Phi^L\co CH_{3}^{L}(\alpha X^{\circ},\lambda_0)\to CH_{3}^{L}(X^{\circ},\lambda_0))\leq 1$ for all $L\in \R$.  By Lemma \ref{opentriangle}, then, $\imath_{\alpha^{-1}L,L}\co CH^{\alpha^{-1}L}(X^{\circ},\lambda_0)\to CH^{L}(X^{\circ},\lambda_0)$ has rank at most one.  

If we had $\alpha < \frac{\|(1,1)\|^{*}_{\Omega}}{\max\{\|(1,0)\|^{*}_{\Omega},\|(0,1)\|^{*}_{\Omega}\}}$, then Lemma \ref{lem:keyconvex}  would allow us to find a real number $L$ such that $\imath_{\alpha^{-1}L,L}\co CH^{\alpha^{-1}L}(X^{\circ},\lambda_0)\to CH^{L}(X^{\circ},\lambda_0)$ is an isomorphism of two-dimensional vector spaces, a contradiction which proves that $\alpha\geq \frac{\|(1,1)\|^{*}_{\Omega}}{\max\{\|(1,0)\|^{*}_{\Omega},\|(0,1)\|^{*}_{\Omega}\}}$, as desired.
\end{proof}

\begin{proof}[Proof of Theorem \ref{dellubounds} (b)] This follows by essentially the same argument, using $k=4$ in the application of Corollary \ref{rankcor} in place of $k=3$, and appealing to Lemma \ref{lem:keyconcave} instead of Lemma \ref{lem:keyconvex}.  This yields the result since any irrational ellipsoid $E$ has no periodic Reeb orbits on its boundary with Conley-Zehnder index equal to $4$ and hence obeys $CH_{4}^{L}(E^{\circ},\lambda_0)=\{0\}$ for all $L\in \R$.\end{proof}

\subsection{Products}\label{prodsect}
The goal of this section is to show that Theorem \ref{4d} extends to products of convex toric domains with large ellipsoids of arbitrary even dimension.
\begin{theorem}\label{anyd}
Let $X\subset \C^2$ belong to any of the following classes of domains:
\begin{itemize}\item[(i)] All convex toric domains $X$ such that, for some $c>0$, $B^4(c)\subsetneq X\subset P(c,c)$.
\item[(ii)] All polydisks $P(a,b)$ for $a\leq b<2a$.
\end{itemize}  Then there exist numbers $\alpha >1$ and $R>0$  and a knotted symplectic embedding $\phi\co X\times E(b_1,\ldots,b_{n-2})\to \alpha \bigl(X\times E(b_1,\ldots,b_{n-2})\bigr)^{\circ}$ for any $b_1,\ldots,b_{n-2}$ with each $b_i\geq R$.
\end{theorem}

(Specific values for $R$ in the various cases will appear in the proof.)

In order to prove this we will first establish some basic facts concerning the relationship of the filtered positive $S^1$-equivariant symplectic homology of a product of two convex toric domains to that of the factors.  Observe that the product of two convex toric domains is a convex toric domain: we have $X_{\Omega_1}\times X_{\Omega_2}=X_{\Omega_1\times \Omega_2}$.  Also notice that, if $\Omega_1\subset \R^m$ and $\Omega_2\subset \R^n$ and if we express general elements of $\R^{n+m}$ as $(\alpha,\beta)$ where $\alpha\in \R^m$ and $\beta\in \R^n$, then $\|(\alpha,\beta)\|_{\Omega_1\times \Omega_2}^{*}=\|\alpha\|_{\Omega_1}^{*}+\|\beta\|_{\Omega_2}^{*}$.

\begin{proposition} \label{prodapprox}   Let $X_{\Omega_1}\subset \C^2$ and $X_{\Omega_2}\subset \C^{n-2}$ be two convex toric domains, and assume that $\min\{\|e_i\|_{\Omega_2}^{*}\}> \|(1,1)\|_{\Omega_1}^{*}$ where $\{e_1,\ldots,e_{n-2}\}$ is the standard basis for $\R^{n-2}$.  Then for any $\delta,\eps >0$ there is a tame star-shaped domain $Z^{\delta,\eps}$ such that $(1-\eps)X_{\Omega_1\times \Omega_2}\subset Z^{\delta,\eps}\subset X_{\Omega_1\times \Omega_2}^{\circ}$ and such that, for \[ \max\left\{\norm{(1,0)}^*_{\Omega_1}, \norm{(0,1)}^*_{\Omega_1} \right\}+\delta\leq L_1<L_2\leq \norm{(1,1)}^{*}_{\Omega_1}-\delta,\] the map
	\[
		\imath_{L_1,L_2}\co CH_{n+1}^{L_1}(Z^{\delta,\eps},\lambda_0) \too CH_{n+1}^{L_2}(Z^{\delta,\eps},\lambda_0)
	\] is an isomorphism of two-dimensional vector spaces.
\end{proposition}

\begin{proof}
As in the proof of Lemma \ref{lem:convexebarcode}, Steps 1, 2, and 3 of the proof of \cite[Lemma 2.5]{GuH} provide a tame star-shaped domain $Z^{\delta,\eps}$ such that $(1-\eps)X_{\Omega_1\times \Omega_2}\subset Z^{\delta,\eps}\subset X_{\Omega_1\times \Omega_2}^{\circ}$ and such that the Reeb orbits of $\lambda_0|_{\partial Z^{\delta,\eps}}$ having action at most  $\|(1,1)\|_{\Omega_1}^{*}$ and Conley-Zehnder index at most $n+2$ consist of:
	\begin{itemize}
	\item no orbits of index $n$;
	\item two orbits $(1,0),(0,1)$ in degree $n+1$, with actions in the intervals $\left(\|(1,0)\|_{\Omega_1}^{*}-\delta, \|(1,0)\|_{\Omega_1}^{*}+\delta\right)$ and$\left(\|(0,1)\|_{\Omega_1}^{*}-\delta, \|(0,1)\|_{\Omega_1}^{*}+\delta\right)$, respectively; and 
	\item at most one orbit $(1,1)$ of index $n+2$, with filtration level greater than $\|(1,1)\|_{\Omega_1}^{*}-\delta$.
	\end{itemize} 
	(In general we would potentially obtain orbits with actions approximately $\|(\alpha,\beta)\|_{\Omega_1\times \Omega_2}^{*}= \|\alpha\|_{\Omega_1}^{*}+\|\beta\|_{\Omega_2}^{*}$ for arbitrary $\alpha\in \N^2$ and $\beta\in \N^{n-2}$, but our restriction to filtration levels less than or equal to $\|(1,1)\|_{\Omega_1}^{*}$, which is assumed to be less than each $\|e_i\|_{\Omega_2}^{*}$ forces $\beta$ to be zero.)
	
	So as in the proof of Lemma \ref{lem:convexebarcode}, for   $L$ in the interval  $\left[\max\{\|(1,0)\|_{\Omega}^{*},\|(0,1)\|_{\Omega}^{*}\}+\delta\,,\,\|(1,1)\|_{\Omega}^{*}-\delta\right]$ we have $CC_{n}^{L}(Z^{\delta,\eps},\lambda_0)=CC_{n+2}^{L}(Z^{\delta,\eps},\lambda_0)=\{0\}$ and $CC_{n+1}^{L}(Z^{\delta,\eps},\lambda_0)\cong \mathbb{Q}^2$, and moreover if $L_1,L_2$ both lie in this interval with $L_1<L_2$ then the inclusion of complexes $CC_{n+1}^{L_1}(Z^{\delta,\eps},\lambda_0)\to CC_{n+1}^{L_2}(Z^{\delta,\eps},\lambda_0)$ is an isomorphism.  So passing to homology shows that, for $\max\{\|(1,0)\|_{\Omega}^{*},\|(0,1)\|_{\Omega}^{*}\}+\delta\leq L_1<L_2\leq \|(1,1)\|_{\Omega}^{*}-\delta$, the inclusion-induced map $\imath_{L_1,L_2}\co CH^{L_1}_{n+1}(Z^{\delta,\eps},\lambda_0)\to CH^{L_2}_{n+1}(Z^{\delta,\eps},\lambda_0)$ is an isomorphism of two-dimensional vector spaces.
\end{proof}

\begin{lemma}\label{lem:keyconvexprod}
	Let $X_{\Omega_1}$ be a convex toric domain in $\C^2$ with the property that $\max\left\{\norm{(1,0)}^*_{\Omega_1}, \norm{(0,1)}^*_{\Omega_1} \right\}<\norm{(1,1)}_{\Omega_1}^{*}$, and let $X_{\Omega_2}$ be a convex toric domain in $\C^{n-2}$ such that $\min_{1\leq i\leq n-2}\|e_i\|_{\Omega_2}^{*}>\|(1,1)\|_{\Omega_1}^{*}$. Then for all small $\eta>0$,
	\[
		\op{Rank}\left(CH_{n+1}^{\max\left\{\norm{(1,0)}^*_{\Omega_1}, \norm{(0,1)}^*_{\Omega_1} \right\}+\eta}\left((X_{\Omega_1}\times X_{\Omega_2})^{\circ},\lambda_0\right) \too CH_{n+1}^{\norm{(1,1)}^*_\Omega-\eta}\left((X_{\Omega_1}\times X_{\Omega_2})^{\circ},\lambda_0\right)\right)=2.
	\]
\end{lemma}

\begin{proof}
Given Proposition \ref{prodapprox}, this is proven in exactly the same way as Lemma \ref{lem:keyconvex}.
\end{proof}

\begin{lemma}\label{ellipsoidproduct} If $E(a_1,a_2)\subset \R^4$ is an ellipsoid with $1<\frac{a_2}{a_1}\notin \Q$, and if $E(b_1,\ldots,b_{n-2})\subset \R^{2n-4}$ is any ellipsoid, then
$\dim CH^{L}_{n+1}\Bigl(\bigl(E(a_1,a_2)\times E(b_1,\ldots,b_{n-2})\bigr)^{\circ},\lambda_0\Bigr)\leq 1$ for all $L<\min\{b_i\}$.
\end{lemma}

\begin{proof}
    Given $p>1$, consider the Hamiltonian
    \[
        H:\C^2\times\C^{n-2}\to\R \co H(z,w):=\left(\left(\frac{\pi|z_1|^{2}}{a_1}+\frac{\pi|z_2|^{2}}{a_2}\right)^{p} + \left(\sum_{i=1}^{n-2}\frac{\pi |w_i|^{2}}{b_i} \right)^{p}\right)^{\frac{1}{p}}
    \]
    
    A computation shows the Hamiltonian vector field $X_H$ of $H$ obeys $\lambda_0(X_H)=-H$, from which one deduces that the Reeb vector field of $\lambda_0$ along the boundary of $Z_p:=\{H\leq 1\}$ is equal to $-X_H$.  (Here we use the sign convention that defines $X_H$ by $d\lambda_0(X_H,\cdot)=dH$.)  
    
    Note that $Z_p\subset E(a_1,a_2)\times E(b_1,\ldots,b_{n-2})$, and that (because the $\ell^p$ norm on $\R^2$ converges uniformly on compact subsets to the $\ell^{\infty}$ norm as $p\to\infty$), for any $\eps>0$ we have $(1-\eps)\bigl(E(a_1,a_2)\times E(b_1,\ldots,b_{n-2})\bigr)\subset Z_p$ for all sufficiently large $p$.

    The Reeb flow on $\partial Z_p$ rotates the $w_i$ coordinates with period $b_i\left(\sum_{i=1}^{n-2}\frac{\pi|w_i|^2}{b_i}\right)^{-(p-1)}$, which is greater than or equal to $b_i$ since, on $\partial Z_p$, we have $\sum_{i=1}^{n-2}\frac{\pi|w_i|^2}{b_i}\leq 1$.  Hence any closed Reeb orbit on $\partial Z_p$ having action less than $\min\{b_i\}$ must have all $b_i$ identically zero.  

    Because $\frac{a_2}{a_1}\notin \Q$, it is easy to check that any closed Reeb orbit on $\partial Z_p$ must have one or both of $z_1,z_2$ identically equal to zero.  Such an orbit which also has all $b_i$ equal to zero has action $ka_1$ or $ka_2$ where $k\in \N$.  Moreover the Conley-Zehnder index of such an orbit is given by $2k + 2\left\lfloor\frac{ka_1}{a_2}\right\rfloor+n-1$ or by $2k + 2\left\lfloor\frac{ka_2}{a_1}\right\rfloor+n-1$. Indeed the linearized flow splits into the symplectic sum of the linearized flows on $E(a_1,a_2)$ and on $E(b_1,\ldots,b_{n-2})$. Thus the Conley-Zehnder index is the sum of the Conley-Zehnder indices of each individual linearized flow.
    
    In particular, there is only one such orbit of Conley-Zehnder index $n+1$, namely the one which rotates once in the $z_1$ plane and has all other coordinates equal to zero.  It follows that $Z_p$ is arbitrarily well-approximated by non-degenerate star-shaped domains $Z_{p}^{\eps}\subset Z_{p}^{\circ}$ such that, for  $L<\min\{b_i\}$, we have  $\dim CH_{n+1}^{L}(Z_{p}^{\eps},\lambda_0)\leq 1$.  By using these $Z_{p}^{\eps}$ for $p\gg 1$ to approximate $\bigl(E(a_1,a_2)\times E(b_1,\ldots,b_{n-2})\bigr)^{\circ}$ it is not hard to see (using arguments like the one in the proof of Lemma \ref{lem:keyconvex}) that $\dim CH^{L}_{n+1}\Bigl(\bigl(E(a_1,a_2)\times E(b_1,\ldots,b_{n-2})\bigr)^{\circ},\lambda_0\Bigr)\leq 1$.
\end{proof}

\begin{proof}[Proof of Theorem \ref{anyd}]
    In Case (a), the proof of Theorem \ref{4d} shows that $\dell(X_{\Omega})<\frac{1}{c}\|(1,1)\|_{\Omega}^{*}$.  Hence there is a sequence of symplectic embeddings $X_{\Omega}\hookrightarrow E(a_1,a_2)\hookrightarrow \alpha X_{\Omega}^{\circ}$ where (without loss of generality) $1<\frac{a_2}{a_1}\notin \Q$ and $1<\alpha< \frac{1}{c}\|(1,1)\|_{\Omega}^{*}=\frac{\|(1,1)\|_{\Omega}^{*}}{\max\{\|(1,0)\|_{\Omega}^{*},\|(0,1)\|_{\Omega}^{*}\}}$.  By taking a product with the identity, this yields symplectic embeddings   \[ X_{\Omega}\times E(b_1,\ldots,b_{n-2})\hookrightarrow E(a_1,a_2)\times E(b_1,\ldots,b_{n-2})\hookrightarrow \alpha \bigl(X_{\Omega}\times E(b_1,\ldots,b_{n-2})\bigr)^{\circ}. \]
If the composition of these embeddings were unknotted, then Corollary \ref{rankcor} (applied with $L$ slightly smaller than $\|(1,1)\|_{\Omega}^{*}$) and Lemma \ref{lem:keyconvexprod} would show that $\dim CH_{n+1}^{L}\left((E(a_1,a_2)\times E(b_1,\ldots,b_{n-2})^{\circ},\lambda_0\right)\geq 2$, a contradiction with Lemma \ref{ellipsoidproduct} provided that we choose $R\geq \|(1,1)\|_{\Omega}^{*}$.

In Case (b), the proof of Theorem \ref{4d} likewise shows that there is a sequence of symplectic embeddings $P(a,b)\hookrightarrow  E(a_1,a_2)\hookrightarrow \alpha P(a,b)^{\circ}$ where $1<\frac{a_2}{a_1}\notin \Q$ and $1<\alpha<\frac{a+b}{b}=\frac{\|(1,1)\|_{\Omega}^{*}}{\max\{\|(1,0)\|_{\Omega}^{*},\|(0,1)\|_{\Omega}^{*}\}}$.  (Here we write $\Omega=[0,a]\times [0,b]$).  Then the same argument as in Case (a) applies to show that the product of the composition of these embeddings with the identity on $E(b_1,\ldots,b_{n-2})$ will be knotted provided that $b_i\geq R:=a+b$ for all $i$.
\end{proof}

\section{Some embeddings of four-dimensional ellipsoids} \label{construct}
The main goal of this section is to prove Theorem \ref{dellbounds}, which asserts the existence of certain symplectic embeddings to and from four-dimensional ellipsoids.  The machinery for constructing (or, perhaps more accurately, ascertaining the existence of) such embeddings has its roots in Taubes-Seiberg-Witten theory and in papers such as \cite{MP}, \cite{M}, \cite{CG} which relate the question of whether certain four-dimensional domains symplectically embed into certain other domains to questions about symplectic ball-packing problems and then to questions about the symplectic cones of blowups of $\mathbb{C}P^2$, which are then converted to elementary problems by results from \cite{LL}.  We will presently recall some of these results, rephrasing them in a way suitable for our applications.

 In this section we will consider a limited class of toric domains in $\mathbb{C}^2$, given as the preimage under the standard moment map $\mu\co (w,z)\mapsto (\pi|w|^2,\pi|z|^2)$ of a quadrilateral having a right-angled vertex at the origin and satisfying a couple of other conditions, see Figure \ref{quads}.  More specifically:

\begin{dfn}\label{quaddfn}  Let $a,b,x,y\in [0,\infty)$ satisfy the following properties: \begin{itemize} \item[(i)] $x\leq a$ and $y\leq b$.
\item[(ii)] If $\frac{x}{a}+\frac{y}{b}<1$, then $x+y\leq \min\{a,b\}$.
\item[(iii)] If $\frac{x}{a}+\frac{y}{b}>1$, then $x+y\geq \max\{a,b\}$.  \end{itemize}
We denote by $T(a,b,x,y)$ the preimage under $\mu$ of the quadrilateral in $\mathbb{R}^2$ having vertices $(0,0),(a,0),(x,y),(0,b)$.  

Any such set $T(a,b,x,y)$ is said to be a \textbf{toric quadrilateral}; it is said to be concave if $\frac{x}{a}+\frac{y}{b}\leq 1$ and convex if $\frac{x}{a}+\frac{y}{b}\geq 1$. 

If $a,b,x,y\in\mathbb{Q}$, then $T(a,b,x,y)$ is said to be a \textbf{rational toric quadrilateral}.

\end{dfn}

\begin{figure}\label{quads} 
\begin{center}
\includegraphics[width=2 in]{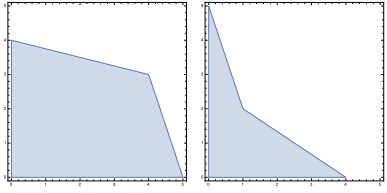}
\caption{The images under $\mu$ of the toric quadrilaterals $T(5,4,4,3)$ and $T(4,5,1,2)$.}
\end{center}
\end{figure}

We allow the possibility that $(x,y)$ lies on the line segment from $(a,0)$ to $(0,b)$ so that the relevant quadrilateral degenerates to a triangle; indeed in this case $T(a,b,x,y)$ is the ellipsoid $E(a,b)$ (and is both concave and convex).

For any rational concave toric quadrilateral $T(a,b,x,y)$ (and indeed for somewhat more general toric domains), \cite{CG} (generalizing \cite{M}) explains how to construct the so-called ``weight sequence'' $w\bigl(T(a,b,x,y)\bigr)$ of $T(a,b,x,y)$, which is a finite unordered sequence of positive numbers.  We will rephrase this as follows.  Given two unordered sequences of positive numbers $a=[a_1,\ldots,a_k],b=[b_1,\ldots,b_l]$ we write $a\sqcup b$ for the union with repetitions: $a\sqcup b=[a_1,\ldots,a_k,b_1,\ldots,b_l]$.  We will abbreviate the weight sequence $w\bigl(E(a,b)\bigr)=w\bigl(T(a,b,a,0)\bigr)$ as $\mathcal{W}(a,b)$.  Then for a general rational concave toric quadrilateral the weight sequence is determined recursively by the following prescriptions: 
\begin{itemize} \item For any $a\geq 0$, $\mathcal{W}(a,0)=\mathcal{W}(0,a)=[]$ (the empty sequence).
\item For $0<a\leq b$, $\mathcal{W}(a,b)=\mathcal{W}(b,a)=[a]\sqcup\mathcal{W}(a,b-a)$.
\item If $\frac{x}{a}+\frac{y}{b}<1$ (which by our assumptions in Definition \ref{quaddfn} imply that $x+y\leq\min\{a,b\}$) then \[ w\bigl(T(a,b,x,y)\bigr)=[x+y]\sqcup\mathcal{W}(a-x-y,y)\sqcup\mathcal{W}(b-x-y,x).\]
\end{itemize} 

For instance, \begin{align*} 
w(T(4,5,1,2))&=[3]\sqcup\mathcal{W}(1,2)\sqcup \mathcal{W}(2,1)=[3]\sqcup[1]\sqcup\mathcal{W}(1,1)\sqcup[1]\sqcup \mathcal{W}(1,1)
\\ &=[3,1,1,1,1]\sqcup \mathcal{W}(1,0)\sqcup \mathcal{W}(1,0)=[3,1,1,1,1].
\end{align*}

Dually, to a convex toric quadrilateral $T$ is associated a ``weight expansion,''  which takes the form of a pair $\bigl(h(T);\hat{w}(T)\bigr)$ where $h(T)\in[0,\infty)$ is called the ``head'' and $\hat{w}(T)$ is a possibly-empty unordered sequence of positive numbers and is called the ``negative weight sequence.''  For a general rational convex toric quadrilateral  the weight expansion is determined 
as follows: 
\begin{itemize}\item If $a\leq b$ then $h\bigl(E(a,b)\bigr)=h\bigl(E(b,a)\bigr)=b$ and $\hat{w}\bigl(E(a,b)\bigr)=\hat{w}\bigl(E(b,a)\bigr)=\mathcal{W}(b,b-a)$.
\item If $\frac{x}{a}+\frac{y}{b}>1$ (which by our assumptions in Definition \ref{quaddfn} imply that $x+y\geq\max\{a,b\}$), then $h\bigl(T(a,b,x,y)\bigr)=x+y$ and $\hat{w}\bigl(T(a,b,x,y)\bigr)=\mathcal{W}(x+y-a,y)\sqcup \mathcal{W}(x+y-b,x)$.
\end{itemize} (This is a complete prescription, since by definition any convex toric quadrilateral $T(a,b,x,y)$ has $\frac{x}{a}+\frac{y}{b}\geq 1$, with equality implying that $T(a,b,x,y)=E(a,b)$.  A more obviously-consistent phrasing is that the head $h(T)$ is equal to the capacity of the smallest ball containing $T$, and that the negative weight sequence is the union of the weight sequences of ellipsoids whose interiors are equivalent under the action of translations and $SL_2(\Z)$ to  the components of $B^4\bigl(h(T)\bigr)^{\circ}\setminus T$.)

The deep result that we need is:
\begin{theorem}\cite[Theorem 1.4]{CG}\label{cgthm}  Let $S_1,\ldots,S_k$ be a rational concave toric quadrilaterals and $T$ be a rational convex toric quadrilateral.  Then the following are equivalent:
\begin{itemize} \item[(i)] For all $\alpha >1$ there is a symplectic embedding $\coprod_{i=1}^{k}S_i\hookrightarrow \alpha T$.
\item[(ii)] For all $\alpha>1$ there is a symplectic embedding \[ \left(\coprod_{i=1}^{k}\coprod_{c\in w(S_i)}B^4(c)\right)\sqcup \coprod_{c\in\hat{w}(T)}B^4(c)\hookrightarrow B^4\bigl(\alpha h(T)\bigr).\]
\end{itemize}
\end{theorem}
(While \cite[Theorem 1.4]{CG} is stated for a single concave toric domain $S_1$, the proof---which closely follows the proof for ellipsoids in \cite{M}---extends without change to a collection of several disjoint such domains, as was already noted when all of the domains are ellipsoids in \cite[Proposition 3.5]{M10}.)

Let us introduce the following notation. If $[a_1,\ldots,a_m]$ is an unordered sequence of nonnegative real numbers, and if $t$ is another nonnegative real number, we will write \[ [a_1,\ldots,a_m]\preceq [t] \] if and only if \[ t\geq \inf\left\{u: \coprod_{i=1}^{m}B^4(a_i)\mbox{ symplectically embeds into }B^4(u)\right\} \]  Then Theorem \ref{cgthm} can be rephrased as stating that, for concave toric quadrilaterals $S_1,\ldots, S_k$ and a convex toric quadrilateral $T$, the statement that for all $\alpha>1$ there is a symplectic embedding $\coprod_{i=1}^{k}S_i\hookrightarrow \alpha T$ is equivalent to the statement that \[ w(S_1)\sqcup\cdots\sqcup w(S_k)\sqcup\hat{w}(T)\preceq [h(T)]. \]

\begin{remark}\label{cp2}
As follows from \cite{MP} and \cite{Liu}, if we denote by $H$ the hyperplane class and $E_1,\ldots,E_m$ the exceptional divisors of the manifold $X_m$ obtained by blowing up $\mathbb{C}P^2$ $m$ times, the statement that $[a_1,\ldots,a_m]\preceq [t]$ is equivalent to the statement that the Poincar\'e dual of the class $tH-\sum a_iE_i$ lies in the set $\bar{\mathcal{C}}_K$ given as the closure of the subset of $H^2(X_m;\mathbb{R})$ consisting of the cohomology classes of symplectic forms having associated canonical class Poincar\'e dual to $-3H+\sum E_i$.
\end{remark}

We will often find it useful to combine Theorem \ref{cgthm} with the following elementary but somewhat subtle fact.  In the special case that $b$ and $c$ are integer multiples of $a$  this has a well-known proof as in  \cite[Lemma 2.6]{M}; see also \cite[Lemma 2.6]{M10} for a corresponding statement about ECH capacities in a different special case.
\begin{prop}\label{tile}
Let $a,b,c\in (0,\infty)$.  Then for all $\alpha >1$ there is a symplectic embedding $E(a,b)\coprod E(a,c)\hookrightarrow \alpha E(a,b+c)$.
\end{prop}
\begin{proof}
For any $v,w>0$ let us write $\square(v,w)=(0,v)\times (0,w)$ and $\triangle(v,w)=\{(x_1,x_2)\in (0,\infty)^2|\frac{x_1}{v}+\frac{x_2}{w}<1\}$. Also for $A,B\subset \R^2$ write $A\times_LB$ for the ``Lagrangian product'' $\{(x_1+iy_1,x_2+iy_2)|(x_1,x_2)\in A,(y_1,y_2)\in B\}\subset \mathbb{C}^2$. 
Now the Traynor trick \cite[Corollary 5.3]{T} shows that for all $\gamma < 1$ there is a symplectic embedding $\gamma E(v, w)\hookrightarrow \triangle(v,w)\times_L \square(1,1)$.  Conversely $(x_1+iy_1,x_2+iy_2)\mapsto (\sqrt{\frac{x_1}{\pi}}e^{2\pi iy_1},\sqrt{\frac{x_2}{\pi}}e^{2\pi iy_2})$ defines a symplectic embedding $\triangle(v,w)\times_L \square(1,1)\to E(v,w)$.  Meanwhile the symplectomorphism of $\mathbb{C}^2$ given by $(x_1+iy_1,x_2+iy_2)\mapsto (v^{-1}x_1+ivy_1,w^{-1}x_2+iwy_2)$ maps $\triangle(v,w)\times_L \square(1,1)$ to $\triangle(1,1)\times_L \square(v,w)$.  Hence: \begin{align} \label{ellipsesquare} \mbox{For any $v,w>0$ and any $\gamma<1$, there are symplectic embeddings } \\ \nonumber \gamma E(v,w)\hookrightarrow \triangle(1,1)\times_L \square(v,w)\hookrightarrow E(v,w). \end{align}

The proof readily follows from this: if $\gamma<1$, we may symplectically embed 
\[
    E(\gamma a,\gamma b)\hookrightarrow\triangle(1,1)\times_L\square(a,b)=\{(x_1+iy_1,x_2+iy_2)|x_1,x_2,y_1,y_2>0,\,x_1+x_2<1,\,y_1<a,y_2<b\},
\]
and likewise, by composing an embedding as in (\ref{ellipsesquare}) with a translation in the $y_2$ direction, we may symplectically embed \[ E(\gamma a,\gamma c)\hookrightarrow \{(x_1+iy_1,x_2+iy_2)|x_1,x_2,y_1>0,\,x_1+x_2<1,\,y_1<a,\,b<y_2<b+c\}.\]  The images of these two embeddings are evidently disjoint, and their union is contained in $\triangle(1,1)\times_L\square(a,b+c)$, which symplectically embeds into $E(a,b+c)$.  We thus obtain, for any $\gamma<1$, a symplectic embedding $E(\gamma a,\gamma b)\coprod E(\gamma a,\gamma c)\hookrightarrow E(a,b+c)$; conjugation by a rescaling then gives the embeddings required in the proposition.
\end{proof}

The following family of embeddings is used in Case (i) of Theorem \ref{4d}; see Figure \ref{L5} for more context in a particular instance.

\begin{prop} \label{cvxaxy}
Let $a,b,x,y\in(0,\infty)$ with $x\leq a, y\leq b,$ and $a\leq b\leq x+y$.  Then for all $\alpha>1$ there is a symplectic embedding $E(a,x+y)\hookrightarrow \alpha T(a,b,x,y)$.
\end{prop}

\begin{proof}  It evidently suffices to prove the statement when $a,b,x,y\in\Q$. Then
by Theorem \ref{cgthm} the statement is equivalent to the statement that $\mathcal{W}(a,x+y)\sqcup\mathcal{W}(x+y-a,y)\sqcup\mathcal{W}(x+y-b,x)\preceq [x+y]$.  But another application of Theorem \ref{cgthm} shows that this, in turn, is equivalent to the statement that for all $\alpha>1$ there exists a symplectic embedding of a disjoint union of three ellipsoids: \[ E(a,x+y)\sqcup E(x+y-a,y)\sqcup E(x+y-b,x)\hookrightarrow \alpha B^4(x+y). \] Since we assume that $a\leq b$ (so $x+y-b\leq x+y-a$), we have symplectic embeddings \[ E(x+y-a,y)\sqcup E(x+y-b,x)\hookrightarrow E(x+y-a,y)\sqcup E(x+y-a,x)\hookrightarrow \sqrt{\alpha}E(x+y-a,x+y)\] where the first map is the inclusion and the second is given by Proposition \ref{tile}.  Combining this with another application of Proposition \ref{tile} yields: \begin{align*}
E(a,x+y)\sqcup & E(x+y-a,y)\sqcup  E(x+y-b,x)\hookrightarrow E(a,x+y)\sqcup \sqrt{\alpha}E(x+y-a,x+y)\\&\subset \sqrt\alpha\left(E(a,x+y)\sqcup E(x+y-a,x+y)\right) \hookrightarrow \alpha E(x+y,x+y)=\alpha B^4(x+y).
\end{align*}
\end{proof}

Similarly in the concave case, we obtain:

\begin{prop} \label{ccvaxy}
Let $a,b,x,y\in (0,\infty)$ with $x+y\leq a\leq b$.  Then for all $\alpha >1$ there is a symplectic embedding $T(a,b,x,y)\hookrightarrow E(b,x+y)$.
\end{prop}

\begin{proof} It again suffices to assume that $a,b,x,y\in\Q$.
Theorem \ref{cgthm} shows that the proposition is equivalent to the statement that $[x+y]\sqcup\mathcal{W}(a-x-y,y)\sqcup \mathcal{W}(b-x-y,x)\sqcup \mathcal{W}(b-x-y,b)\preceq [b]$, which in turn is equivalent to the existence of a symplectic embedding, for all $\alpha >1$,\[  B^4(x+y)\sqcup E(a-x-y,y)\sqcup E(b-x-y,x)\sqcup E(b,b-x-y)\hookrightarrow \alpha B^4(b).\]  Proposition \ref{tile} (together with the inclusion $E(a-x-y,y)\subset E(b-x-y,y)$) gives, for all $\nu>1$, embeddings $E(a-x-y,y)\sqcup E(b-x-y,x)\hookrightarrow\nu E(b-x-y,x+y)$ and then $B^4(x+y)\sqcup E(b-x-y,x+y)\hookrightarrow \nu E(b,x+y)$, and finally $E(b,x+y)\sqcup E(b,b-x-y)\hookrightarrow\nu E(b,b)=\nu B^4(b)$.  Combining these three embeddings (with $\nu=\alpha^{1/3}$) then implies the result.
\end{proof}

\begin{remark}
Note that the volume of $T(a,b,x,y)$ is $\frac{1}{2}(ay+bx)$, while that of $E(a,x+y)$ is $\frac{1}{2}a(x+y)$.  So in the case that $a=b$, the embeddings $E(a,x+y)\to \alpha T(a,a,x,y)$ (in the convex case) or $T(a,a,x,y)\to \alpha E(a,x+y)$ (in the concave case) fill all but an arbitrarily small proportion of the volumes of their targets as $\alpha\to 1$.
\end{remark}

Since $P(1,b)=T(1,b,1,b)$, a special case of Proposition \ref{cvxaxy} is that, for any $\alpha>1$ and $b\geq 1$, there is a symplectic embedding $E(1,1+b)\hookrightarrow \alpha P(1,b)$.  The following reproduces this embedding when $1\leq b<2$, and improves on it for $b=m+\eps\geq 2$.  The case that $\eps=0$ is well-known; see \cite[Remark 1.2(1)]{CFS}.

\begin{prop}\label{longembed} Let $m\in\mathbb{Z}_+$ and $0\leq \ep<1$.  Then for all $\alpha >1$ there is a symplectic embedding $E(1,2m+\ep)\hookrightarrow \alpha P(1,m+\ep)$.
\end{prop}

\begin{proof}
By Theorem \ref{cgthm} the statement is equivalent to the statement that $\mathcal{W}(1,2m+\ep)\sqcup \mathcal{W}(m+\ep,m+\ep)\sqcup\mathcal{W}(1,1)\preceq [m+1+\ep]$.  From the recursive description  of $\mathcal{W}(a,b)$ given earlier we see that $\mathcal{W}(1,2m+\ep)=\mathcal{W}(1,m)\sqcup \mathcal{W}(1,m+\ep)$, so this is equivalent to the existence, for all $\alpha >1$, of a symplectic embedding \[ E(1,m)\sqcup E(1,m+\ep)\sqcup E(m+\ep,m+\ep) \sqcup E(1,1)\hookrightarrow \alpha B^4(m+1+\ep).\]  But by Proposition \ref{tile} there are symplectic embeddings \[ E(1,m)\sqcup E(1,1)\hookrightarrow \sqrt{\alpha}E(1,m+1)\subset \sqrt{\alpha} E(1,m+1+\ep) \] and \[ E(1,m+\ep)\sqcup E(m+\ep,m+\ep)\hookrightarrow \sqrt{\alpha}E(m+\ep,m+1+\ep),\] and then another application of Proposition \ref{tile} gives a symplectic embedding \[ \sqrt{\alpha}E(m+\ep,m+1+\ep)\sqcup \sqrt{\alpha}E(1,m+1+\ep)\hookrightarrow \alpha E(m+1+\ep,m+1+\ep)=\alpha B^4(m+1+\ep),\] from which the result is immediate.
\end{proof}

The embeddings in Propositions \ref{cvxaxy}, \ref{ccvaxy}, and \ref{longembed} will give rise to many of the knotted embeddings described in the introduction.  Some of our other knotted embeddings require a somewhat less straightforward application of Theorem \ref{cgthm} and Proposition \ref{tile}.  The key additional (and standard) ingredient is the use of \emph{Cremona moves}, based on \cite[Proof of Lemma 3.4]{LL}.  As in Remark \ref{cp2} we regard the question of whether $[a_1,\ldots,a_m]\preceq [t]$ as equivalent to the question of whether the Poincar\'e dual of $tH-\sum a_iE_i$ lies in the closure $\bar{\mathcal{C}}_K$ of the appropriate connected component of the symplectic cone of the $m$-fold blowup $X_m$ of $\mathbb{C}P^2$.  Since $[a_1,\ldots,a_m]\preceq [t]$ if and only if $[a_1,\ldots,a_m,0]\preceq[t]$ we may without loss of generality assume that $m\geq 3$.  Then $X_m$ contains a sphere in the class $H-E_1-E_2-E_3$ of self-intersection $-2$ and Chern number zero; the cohomological action of a Dehn-Seidel twist in this sphere preserves $\bar{\mathcal{C}}_K$ and maps the Poincar\'e dual of $tH-\sum a_iE_i$ to the Poincar\'e dual of $(2t-a_1-a_2-a_3)H-(t-a_2-a_3)E_1-(t-a_1-a_3)E_2-(t-a_1-a_2)E_3-\sum_{i=4}^{m}a_iE_i$.  So we have:

\begin{prop}\cite{LL} \label{cremona}
Assume that $t\geq \max\{a_1+a_2,a_1+a_3,a_2+a_3\}$.  Then \[ [a_1,a_2,a_3,a_4,\ldots,a_m]\preceq [t] \quad \mbox{ if and only if } \]
\[ [t-a_2-a_3,t-a_1-a_3,t-a_1-a_2,a_4,\ldots,a_m] \preceq [2t-a_1-a_2-a_3] \] 
\end{prop}

The following will help us construct the  knotted polydisks from Case (iv) of Theorem \ref{4d}.

\begin{prop}\label{step2}
Let $a\leq y\leq b\leq 2a$.  Then for all $\alpha >1$ there is a symplectic embedding \[ E\left(\frac{a+b}{3},2a+y\right)\hookrightarrow \alpha T(a,b,a,y).\]
\end{prop}

\begin{proof}
As usual assuming that $a,b,y\in\Q$, by Theorem \ref{cgthm} the proposition is equivalent to the statement that \begin{equation}\label{goal} \mathcal{W}\left(\frac{a+b}{3},2a+y\right)\sqcup\mathcal{W}(y,y)\sqcup\mathcal{W}(a+y-b,a)\preceq [a+y].\end{equation}  Since $a\leq y$ and $b\leq 2a$ we have $a+b\leq 2a+y$, in view of which \[ \mathcal{W}\left(\frac{a+b}{3},2a+y\right)=\left[\frac{a+b}{3},\frac{a+b}{3},\frac{a+b}{3}\right]\sqcup \mathcal{W}\left(\frac{a+b}{3},a+y-b\right).\]  Meanwhile of course $\mathcal{W}(y,y)=[y]$, and (since $y\leq b$) $\mathcal{W}(a+y-b,a)=[a+y-b]\sqcup\mathcal{W}(a+y-b,b-y)$.  So (\ref{goal}) amounts to the statement that \[ \left[y,a+y-b,\frac{a+b}{3},\frac{a+b}{3},\frac{a+b}{3}\right]\sqcup\mathcal{W}\left(\frac{a+b}{3},a+y-b\right)\sqcup\mathcal{W}(b-y,a+y-b)\preceq [a+y].\]  
Applying Proposition \ref{cremona} and reordering the sequence in brackets shows that this is equivalent to the statement that \[ \left[\frac{a+b}{3},\frac{a+b}{3},\frac{2b-a}{3},\frac{2a-b}{3},b-y\right]\sqcup \mathcal{W}\left(\frac{a+b}{3},a+y-b\right)\sqcup\mathcal{W}(b-y,a+y-b)\preceq\left[\frac{2a+2b}{3}\right].\]
Then another application of Proposition \ref{cremona} shows that this last statement (and hence also (\ref{goal})) is equivalent to the statement that \[
\left[\frac{2a-b}{3},\frac{2a-b}{3},0,\frac{2a-b}{3},b-y\right]\sqcup\mathcal{W}\left(\frac{a+b}{3},a+y-b\right)\sqcup\mathcal{W}(b-y,a+y-b)\preceq[a].
\]  The left hand side above can be rewritten as \[ [0]\sqcup \mathcal{W}\left(\frac{2a-b}{3},2a-b\right)\sqcup\mathcal{W}\left(\frac{a+b}{3},a+y-b\right)\sqcup \mathcal{W}(b-y,a+y-b)\sqcup \mathcal{W}(b-y,b-y).\]  So by Theorem \ref{cgthm} it suffices to show that for all $\alpha >1$ there is a symplectic embedding \begin{equation}\label{embedgoal} E\left(\frac{2a-b}{3},2a-b\right)\sqcup E\left(\frac{a+b}{3},a+y-b\right)\sqcup E(b-y,a+y-b)\sqcup E(b-y,b-y)\hookrightarrow \alpha E(a,a).\end{equation}  We now repeatedly use Proposition \ref{tile}, obtaining for any $\nu>1$ symplectic embeddings: \begin{itemize} \item $E(b-y,a+y-b)\sqcup E(b-y,b-y)\hookrightarrow \nu E(b-y,a)$;
\item $E\left(\frac{2a-b}{3},2a-b\right)\sqcup E\left(\frac{a+b}{3},a+y-b\right)\subset E\left(\frac{2a-b}{3},a+y-b\right)\sqcup E\left(\frac{a+b}{3},a+y-b\right)\hookrightarrow \nu E(a,a+y-b)$ (since $a\leq y$);
\item $E(a,a+y-b)\sqcup E(b-y,a)\cong E(a+y-b,a)\sqcup E(b-y,a)\hookrightarrow \nu E(a,a)$.\end{itemize}  Combining these embeddings (with $\nu=\sqrt{\alpha}$) yields the embedding (\ref{embedgoal}) and hence proves the proposition.
\end{proof}

\subsection{Proof of Theorem \ref{dellbounds}} \label{dellproof}
We begin with the following easy observation, using the terminology and notation from Section \ref{intro}.
\begin{prop} Let $\Omega\subset [0,\infty)^n$ be any star-shaped domain such that $X_{\Omega}$ contains the origin in its interior.  Then \begin{equation}\label{maxdell} \dell(X_{\Omega})\leq \inf\left\{\left.\left\|\left(\frac{1}{a_1},\ldots,\frac{1}{a_n}\right)\right\|^{*}_{\Omega}\right|\mbox{There is a symplectic embedding $E(a_1,\ldots,a_n)\hookrightarrow X_{\Omega}^{\circ}$}\right\} \end{equation} and \begin{equation}\label{mindell} 
\dell(X_{\Omega})\leq \frac{1}{\sup\left\{\left.[(\frac{1}{a_1},\ldots,\frac{1}{a_n})]_{\Omega}\right|\mbox{There is a symplectic embedding $X_{\Omega}\hookrightarrow E(a_1,\ldots,a_n)^{\circ}$}\right\}}.
\end{equation}
\end{prop}

\begin{proof}
We first prove (\ref{maxdell}).  Suppose that there is a symplectic embedding $E(a_1,\ldots,a_n)\hookrightarrow X_{\Omega}^{\circ}$ and let $\alpha=\|(\frac{1}{a_1},\ldots,\frac{1}{a_n})\|^{*}_{\Omega}$.  So by definition, each point $(x_1,\ldots,x_n)\in \Omega$ obeys $\sum_i\frac{x_i}{a_i}\leq \alpha$.  But $\alpha E(a_1,\ldots,a_n)$ is precisely the preimage under $\mu$ of $\big\{(x_1,\ldots,x_n)\in[0,\infty)^n|\sum_i\frac{x_i}{a_i}\leq\alpha\big\}$, while $X_{\Omega}=\mu^{-1}(\Omega)$.  So we have $E(a_1,\ldots,a_n)\hookrightarrow X_{\Omega}^{\circ}$ and $X_{\Omega}\subset \alpha E(a_1,\ldots,a_n)$, and hence for $E=\alpha E(a_1,\ldots,a_n)$ there are symplectic embeddings $X_{\Omega}\hookrightarrow E\hookrightarrow \alpha X_{\Omega}^{\circ}$. Thus $\dell(X_{\Omega})\leq\alpha$.  Since $(a_1,\ldots,a_n)$ was arbitrary subject to the assumption that there is a symplectic embedding $E(a_1,\ldots,a_n)\hookrightarrow X_{\Omega}^{\circ}$, this proves (\ref{maxdell}).

Similarly, suppose that there is a symplectic embedding $X_{\Omega}\hookrightarrow E(a_1,\ldots,a_n)^{\circ}$ and let $\nu=\left[(\frac{1}{a_1},\ldots,\frac{1}{a_n})\right]_{\Omega}$.  Then for each $(x_1,\ldots,x_n)\in\overline{[0,\infty)^2\setminus\Omega}$   we have $\sum_i\frac{x_i}{a_i}\geq \nu$. 

So since $\Omega$ is closed  it then follows that $\big\{(y_1,\ldots,y_n)\in[0,\infty)^n|\sum_i\frac{y_i}{a_i}\leq\nu\big\}\subset \Omega$.  Taking preimages under $\mu$ then shows that $\nu E(a_1,\ldots,a_n)\subset X_{\Omega}$,  and hence $E(a_1,\ldots,a_n)\hookrightarrow \nu^{-1}X_{\Omega}$.  Thus $\dell(X_{\Omega})\leq\frac{1}{\nu}$, which implies (\ref{mindell}) since $\nu$ was arbitrary subject to the assumption that there is a symplectic embedding $X_{\Omega}\hookrightarrow E(a_1,\ldots,a_n)^{\circ}$.
\end{proof}

The proof of Theorem \ref{dellbounds} now follows almost immediately based on Propositions \ref{cvxaxy}, \ref{ccvaxy}, and \ref{step2}.  For part (a),
the hypotheses that $\hat{\Omega}$ is convex and that $(a,0),(0,b),(x,y)\in \Omega$ imply that also $(0,0),(0,y)\in \Omega$.   Since $(0,y)\in\Omega$ and since the right-hand-side of the desired inequality is independent of $b$,  there is no loss of generality in assuming that $b\geq y$, while the hypothesis of the theorem gives inequalities $x\leq a\leq b\leq x+y$. The fact that  $\Omega$ is a convex region containing $(a,0),(0,0),(0,b),(x,y)$  implies that the quadrilateral with these points as its vertices is contained in $\Omega$, and hence that $T(a,b,x,y)\subset X_{\Omega}$.  So for any $\alpha >1$ Proposition \ref{cvxaxy} gives a symplectic embedding $E\big(\alpha^{-1}a,\alpha^{-1}(x+y)\big)\hookrightarrow X_{\Omega}^{\circ}$, whence (\ref{maxdell}) yields Theorem \ref{dellbounds} (a).

Similarly in part (b), by hypothesis we have $(a,0),(x,y),(0,b)\in\overline{[0,\infty)^2\setminus\Omega}$, and moreover $[0,\infty)^2\setminus \Omega$ (and hence also its closure) is convex.  Since $\Omega$ is bounded, it follows that  $\overline{[0,\infty)^2\setminus\Omega}$ contains all points of form $t\vec{v}$ where $t\geq 1$ and $\vec{v}$ lies on  the  line segment from $(a,0)$ to $(x,y)$ or the line segment  from $(x,y)$ to $(0,b)$.  The preimage under $\mu$ of the set of all  such points is $\R^4\setminus T(a,b,x,y)^{\circ}$,  while the preimage under $\mu$ of $\overline{[0,\infty)^2\setminus\Omega}$ is $\R^4\setminus X_{\Omega}^{\circ}$, so this shows that $X_{\Omega}^{\circ}\subset T(a,b,x,y)^{\circ}$ and hence (recalling our convention that ``domains'' are closures of open subsets) that  $X_{\Omega}\subset T(a,b,x,y)$.  Thus part (b) of Theorem \ref{dellbounds} follows from Proposition \ref{ccvaxy} and (\ref{mindell}).

Part (c) of Theorem \ref{dellbounds} is an immediate application of Proposition \ref{step2} (applied to $P(a,b)=T(a,b,a,b)$) together with (\ref{maxdell}).

\subsection{An explicit construction}\label{explicitsection}

The embeddings from Propositions \ref{cvxaxy}, \ref{ccvaxy}, and \ref{step2} that underlie Theorem \ref{dellbounds} are obtained by very indirect methods and are difficult to understand concretely.  We will now explain a more direct construction that, for instance, leads to an explicit formula for a knotted embedding $P(1,1)\to \alpha P(1,1)^{\circ}$ for any $\alpha\in \left(\frac{1}{2-\sqrt{2}},2\right)$.

The key ingredient is a toric structure on the complement of the antidiagonal in $S^2\times S^2$ that appears (at least implicitly) in \cite[Example 1.22]{EP}, \cite{FOOO}, \cite[Section 2]{OU}.   View $S^2$ as the unit sphere in $\R^3$ and let $A=\{(v,w)\in S^2\times S^2\,|\,w=-v\}$ be the antidiagonal.  Define functions $F_1,F_2\co S^2\times S^2\to\R$ by \[ F_1(v,w)=v_3+w_3\qquad F_2(v,w)=\|v+w\|.\]  Now $F_2$ fails to be smooth along $A=F_{2}^{-1}(\{0\})$, but on $S^2\times S^2\setminus A$ the Hamiltonian flows of the functions $F_1$ and $F_2$  induce $S^1$-actions that commute with each other and are rather simple to understand: $F_1$ induces simultaneous rotation of the factors about the $z$-axis, and $F_2$ induces the flow which rotates the pair $(v,w)\in S^2\times S^2\setminus A$ about an axis in the direction of $v+w$.  In formulas:\begin{align}\label{f1flow} \phi_{F_1}^{t}&\bigl((v_1,v_2,v_3),(w_1,w_2,w_3)\bigr)\\ \nonumber &=\Bigl(\bigl((\cos t)v_1-(\sin t)v_2,(\sin t)v_1+(\cos t)v_2,v_3\bigr),\bigl((\cos t)w_1-(\sin t)w_2,(\sin t)w_1+(\cos t)w_2,w_3\bigr)\Bigr) \end{align} and \begin{equation}\label{f2flow} \phi_{F_2}^{t}\left(v,w\right)=\left(\frac{v+w}{2}+(\cos t)\frac{v-w}{2}+(\sin t)\frac{w\times v}{\|v+w\|},\frac{v+w}{2}+(\cos t)\frac{w-v}{2}+(\sin t)\frac{v\times w}{\|v+w\|}\right).\end{equation}

Define \[ J\co S^2\times S^2\to \R^2\quad\mbox{by}\quad J(v,w)=\left(2-\|v+w\|,\|v+w\|-v_3-w_3\right), \] \emph{i.e.} $J=(2-F_2,-F_1+F_2)$.    Then $J$ is smooth away from $A$, and its restriction to $S^2\times S^2\setminus A$ is the moment map for a Hamiltonian $T^2$-action.\footnote{Here we view $T^2$ as $(\R/2\pi\Z)^2$.  On the other hand the map $\mu(w,z)=(\pi|w|^2,\pi|z|^2)$ that we have considered elsewhere is the moment map for a Hamiltonian $(\R/\Z)^2$-action; to get a $(\R/2\pi\Z)^2$-action one would take $\frac{\mu}{2\pi}$.} It is not hard to see that $J$ has image equal to $\Delta:=\{(x,y)\in [0,\infty)^2|x/2+y/4\leq 1\}$, and that the preimage of $\{x/2+y/4=1\}$ is equal to $Q:=\{(v,w)\in S^2\times S^2|v_3+w_3=-\|v+w\|\}$.  (In other words, $Q$ is the locus of pairs $(v,w)\in S^2\times S^2$ such that $v+w$ is on the nonpositive $z$ axis.)

\begin{prop}\label{PhiS}
Let $\Delta^{\circ}=\left\{(x,y)\in[0,\infty)^2\left|\frac{x}{2}+\frac{y}{4}<1\right.\right\}$ and define $s\co \Delta^{\circ}\to S^2\times S^2$ by \begin{align*} s(x,y)=&\left(\left(\sqrt{x\left(1-\frac{x}{4}\right)},\sqrt{y\left(1-\frac{x}{2}-\frac{y}{4}\right)},1-\frac{x+y}{2}\right)
,\right.\\ &\phantom{coucou}\left.\left(-\sqrt{x\left(1-\frac{x}{4}\right)},\sqrt{y\left(1-\frac{x}{2}-\frac{y}{4}\right)},1-\frac{x+y}{2}\right)\right).\end{align*} Then, writing $E(4\pi,8\pi)^{\circ}=\left\{(w,z)\in \C^2\,|\,\frac{|w|^2}{4}+\frac{|z|^2}{8}<1\right\}$, the map \[ \Phi\bigl(|z_1|e^{i\theta},|z_2|e^{i\varphi}\bigr)=\phi_{F_1}^{\varphi}\left(\phi_{F_2}^{\theta-\varphi}\left(s\left(\frac{|z_1|^2}{2},\frac{|z_2|^2}{2}\right)\right)\right) \] defines a symplectomorphism $\Phi\co E(4\pi,8\pi)^{\circ}\to S^2\times S^2\setminus Q$ which satisfies $J\circ\Phi(z_1,z_2)=\left(\frac{|z_1|^2}{2},\frac{|z_2|^2}{2}\right)$.
\end{prop}

\begin{proof} First we observe that $s$ indeed takes values in $S^2\times S^2\subset \R^3\times \R^3$, which follows by computing \begin{align*} x &\left(1-\frac{x}{4}\right)+y\left(1-\frac{x}{2}-\frac{y}{4}\right)+\left(1-\frac{x+y}{2}\right)^2
\\&=x+y-\frac{x^2+y^2}{4}-\frac{xy}{2}+1-x-y+\frac{x^2+2xy+y^2}{4}=1.\end{align*}  Given $(x,y)\in\Delta^{\circ}$, if we write $(v,w)=s(x,y)$, then \[ \|v+w\|^2=4y\left(1-\frac{x}{2}-\frac{y}{4}\right)+(2-x-y)^2=x^2-4x+4=(2-x)^2,\] so (since $x<2$) \[ J\bigl(s(x,y)\bigr)=\bigl(2-\|v+w\|,-v_3-w_3+\|v+w\|\bigr)=(x,x+y-2+2-x)=(x,y).\]  In particular, the image of $s$ is contained in $S^2\times S^2\setminus Q=J^{-1}(\Delta^{\circ})$, and it intersects each fiber of $J|_{J^{-1}(\Delta^{\circ})}$ just once.

Moreover, since the image of $s$ is contained in $\big\{(v,Rv)\,|\,v\in S^2\big\}$ where $R$ is the reflection through the $v_2v_3$-plane and hence is antisymplectic, we see that $s^{\star}\Omega=0$ where $\Omega$ is the standard product symplectic form on $S^2\times S^2$.  Thus $s\co \Delta^{\circ}\to J^{-1}(\Delta^{\circ})$ is a Lagrangian right inverse to the moment map $J$.  

Write $\psi^{1}_{(\theta,\varphi)}(z_1,z_2)=(e^{-i\theta}z_1,e^{-i\varphi}z_2)$ for the standard $T^2$-action on $E(4\pi,8\pi)^{\circ}$ (with moment map $\frac{\mu}{2\pi}$ having image equal to $\Delta^{\circ}$; the negative signs in front of $\theta$ and $\varphi$ arise because our convention for Hamiltonian vector fields is $\omega_0(X_H,\cdot)=dH$).  Likewise write $\psi^{2}_{(\theta,\varphi)}=\phi_{F_1}^{-\varphi}\circ\phi_{F_2}^{\varphi-\theta}$ for the $T^2$-action on $S^2\times S^2\setminus Q$ induced by the moment map $J$.  Our map $\Phi$ maps the Lagrangian section of $\frac{\mu}{2\pi}$ given by the nonnegative real locus of $E(4\pi,8\pi)^{\circ}$ to the Lagrangian section of $J|_{S^2\times S^2\setminus Q}$ given by the image of $s$, and $\Phi$ obeys $J\circ\Phi=\frac{\mu}{2\pi}$ and, for all $(\theta,\varphi)\in T^2$, $\Phi\circ\psi^{1}_{(\theta,\varphi)}=\psi^{2}_{(\theta,\varphi)}\circ\Phi$.  These facts are easily seen to imply that $\Phi$ is a symplectomorphism, as it identifies action-angle coordinates on $E(4\pi,8\pi)^{\circ}$ with action-angle coordinates on $S^2\times S^2\setminus Q$.  The last statement is immediate from the formula for $\Phi$ and the facts that $J\circ s$ is the identity and that $J$ is preserved under the Hamiltonian flows of $F_1$ and $F_2$.
\end{proof}

\begin{remark}\label{altform} With sufficient effort, one can derive the following equivalent formula for the map $\Phi\co E(4\pi,8\pi)^{\circ}\to S^2\times S^2$ from Proposition \ref{PhiS}: regarding $S^2$ as the unit sphere in $\C\times \R$, we have \begin{align} \Phi & (w,z)=\bigl(\Gamma(w,z),\Gamma(-w,z)\bigr) \quad\mbox{where} \nonumber \\ \label{longeqn} \Gamma(w,z)&= \left(\frac{\sqrt{8-|w|^2}\left((8-2|w|^2-|z|^2)w+\bar{w}z^2\right)}{8(4-|w|^2)}+\frac{iz}{4}\sqrt{8-2|w|^2-|z|^2},\right.\\ & \qquad \qquad \left. 1-\frac{|w|^2+|z|^2}{4}-\frac{\sqrt{(8-|w|^2)(8-2|w|^2-|z|^2)}}{4(4-|w|^2)}Im(w\bar{z})\right).\nonumber \end{align}

Since $E(4\pi,8\pi)^{\circ}$ is precisely the locus where $2|w|^2+|z|^2<8$, this makes clear that $\Phi$ is a smooth (indeed even real-analytic) map despite the appearance of square roots in the formula for $s$ in Proposition \ref{PhiS}.  
\end{remark}

Now if $D(4\pi)$ denotes the open disk of area $4\pi$ (so radius $2$) in $\C$, there is a symplectomorphism $\sigma\co S^2\setminus\{(0,-1)\}\to D(4\pi)$ defined by \begin{equation}\label{sigma} \sigma(z,v_3)=\sqrt{\frac{2}{1+v_3}}z\end{equation} where as in Remark \ref{altform} we regard $S^2$ as the unit sphere in $\C\times\R$.

So if we let $\mathcal{I}=\left(\{(0,-1)\}\times S^2\right)\cup\left(S^2\times\{(0,-1)\}\right)$ then $\sigma\times \sigma$ defines a symplectomorphism $S^2\times S^2\setminus\mathcal{I}\cong P(4\pi,4\pi)^{\circ}=D(4\pi)\times D(4\pi)$.  

For $v=(z,v_3)\in S^2\subset \C\times\R$, we have \[ \|v+(0,-1)\|^2=|z|^2+v_{3}^{2}-2v_3+1=2-2v_3 \] and hence

\[ J\bigl(v,(0,-1)\bigr)=J\bigl((0,-1),v\bigr)=\left(2-\sqrt{2-2v_3},  \sqrt{2-2v_3}+(1-v_3)\right).\]  Thus \[ J(\mathcal{I})\subset \bigl\{(x,y)\in\R^2\,|\,(2-x)^2=2(x+y)-4\bigr\}=\bigl\{(x,y)\in\R^2\,|\,y=\frac{x^2}{2}-3x+4\bigr\}.\]   Since $\frac{\mu}{2\pi}=J\circ\Phi$, we have $\frac{\mu}{2\pi}(\Phi^{-1}(\mathcal{I}))=J(\mathcal{I})$. From this we obtain the following:

\begin{prop}\label{PhiEllipse}
Suppose that $X_{\Omega}$ is a convex toric domain where $\Omega\subset\big\{(2\pi x,2\pi y)\in [0,\infty)^2|y<\frac{x^2}{2}-3x+4\big\}$. Then there is an ellipsoid $E$ such that $X_{\Omega}\subset E^{\circ}$ and such that the map $\Phi$ from Proposition \ref{PhiS} maps $E$ to a subset of $S^2\times S^2\setminus \mathcal{I}$.  Hence $(\sigma\times\sigma)\circ\Phi|_E$ is a symplectic embedding from $E$ to $P(4\pi,4\pi)^{\circ}$.     
\end{prop}

\begin{proof}
The sets $\frac{1}{2\pi}\Omega$ and $S:=\bigl\{(x,y)\in [0,\infty)^2\,|\,y\geq \frac{x^2}{2}-3x+4\bigr\}$ are disjoint, closed, convex subsets of $\R^2$, and the first of these sets is compact, so the hyperplane separation theorem shows that they must be separated by a line $\ell$, which passes through the first quadrant since both sets are contained in the first quadrant. This line $\ell$ must have negative slope, since $S$ intersects all lines with positive slope and also intersects all horizontal or vertical lines that pass through the first quadrant.  So we can write the separating line as $\ell=\bigl\{(x,y)\in \R^2\,|\,\frac{x}{a}+\frac{y}{b}=1\bigr\}$ with $a,b>0$, and then it will hold that $\frac{1}{2\pi}\Omega\subset \bigl\{\frac{x}{a}+\frac{y}{b}<1\bigr\}$ and $S\subset \{\frac{x}{a}+\frac{y}{b}>1\}$.  The first inclusion shows that $X_{\Omega}\subset E(2\pi a,2\pi b)^{\circ}$.  Meanwhile since $(2,0),(0,4)\in S\subset\bigl\{\frac{x}{a}+\frac{y}{b}>1\bigr\}$, we have $a<2$ and $b<4$.  So $E(2\pi a,2\pi b)$ is contained in the domain of the map $\Phi$ from Proposition \ref{PhiS}, and by the discussion before the proposition the fact that $\ell\cap S=\varnothing$ implies that $E(2\pi a,2\pi b)\cap \Phi^{-1}(\mathcal{I})=\varnothing$.  Hence the proposition holds with $E=E(2\pi a,2\pi b)$. 
\end{proof}

\begin{cor}\label{explicit}
Suppose that $X_{\Omega}$ is a convex toric domain with $\Omega\subset \{(2\pi x,2\pi y)\in [0,\infty)^2|y<\frac{x^2}{2}-3x+4\}$, and that we have $P(4\pi,4\pi)\subset \alpha X_{\Omega}$ for some $\alpha<\dellu(X_{\Omega})$. Then $(\sigma\times \sigma)\circ\Phi|_{X_{\Omega}}\co X_{\Omega}\hookrightarrow P(4\pi,4\pi)^{\circ}\subset \alpha X_{\Omega}^{\circ}$ defines a knotted embedding $X_{\Omega}\hookrightarrow\alpha X_{\Omega}^{\circ}$.
\end{cor}

\begin{proof}
By Proposition \ref{PhiEllipse} we have an ellipsoid $E$ and a sequence $X_{\Omega}\hookrightarrow E^{\circ}\hookrightarrow P(4\pi,4\pi)^{\circ}\subset \alpha X_{\Omega}^{\circ}$ where the first map is the inclusion and the second map is $(\sigma\times \sigma)\circ \Phi|_{E}$.  So the corollary follows directly from the assumption that $\alpha<\dellu(X_{\Omega})$ and the definition of $\dellu$.
\end{proof}

We emphasize that this embedding $(\sigma\times\sigma)\circ\Phi$ is completely explicit: $\sigma$ is defined in (\ref{sigma}) and $\Phi$ is defined in Proposition \ref{PhiS} based partly on the formulas (\ref{f1flow}) and (\ref{f2flow}), or even more explicitly is given by (\ref{longeqn}).

\begin{figure}\label{explicitfig}
\includegraphics[width=\textwidth]{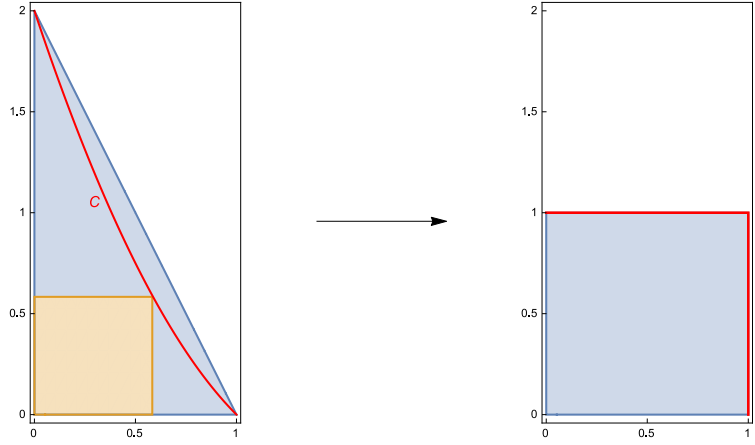}
\caption{After appropriate rescalings, the map $\Phi$ from Proposition \ref{PhiS} sends the interior of the ellipsoid $E(1,2)$ to a product of spheres of area $1$, with the preimage of $(S^2\times \{(0,0,-1)\})\cup(\{(0,0,-1)\}\times S^2)$ contained in $\mu^{-1}(C)$ where $C$ is the red curve at left.  Consequently the preimage under $\mu$ of any domain lying below $C$, such as the small square at left, is embedded into the polydisk $P(1,1)^{\circ}$ by a rescaling of $(\sigma\times\sigma)\circ\Phi$. This gives an explicit knotted embedding $P(c,c)\hookrightarrow P(1,1)^{\circ}$ for $1/2<c<2-\sqrt{2}$.}
\end{figure}  

\begin{ex} For instance, $\Omega$ could be taken to be a square $[0,2\pi c]^2$ with $c$ smaller than the smallest root of the polynomial $\frac{x^2}{2}-4x+4$, namely $
4-2\sqrt{2}$ (see Figure \ref{explicitfig}).  So we obtain an embedding $(\sigma\times\sigma)\circ\Phi\co P(2\pi c,2\pi c)\hookrightarrow P(4\pi,4\pi)^{\circ}=\frac{2}{c}P(2\pi c,2\pi c)^{\circ}$, which is knotted provided that $\frac{2}{c}<\dellu\bigl(P(2\pi c,2\pi c)\bigr)$. By Theorem \ref{dellubounds} we have $\dellu\bigl(P(a,a)\bigr)\geq 2$ for any $a$, so our embedding is knotted provided that $1<c<4-2\sqrt{2}$.  So after conjugating by appropriate rescalings our explicit embedding $(\sigma\times\sigma)\circ\Phi$ defines a knotted embedding $P(a,a)\hookrightarrow \alpha P(a,a)^{\circ}$ provided that $2>\alpha>\frac{1}{2-\sqrt{2}}\approx 1.71$.  For comparison, our less explicit construction based on Proposition \ref{cvxaxy} (leading to the bound $\dell(P(a,a))\leq 3/2$ from Theorem \ref{dellbounds}) gives knotted embeddings  $P(a,a)\hookrightarrow \alpha P(a,a)^{\circ}$ whenever $2>\alpha>1.5 $. 

Choosing the scaling so that the codomain is $P(4\pi,4\pi)^{\circ}$, the image of this embedding $\alpha^{-1}P(4\pi,4\pi)\hookrightarrow P(4\pi,4\pi)^{\circ}$ is not hard to describe explicitly as a subset of $P(4\pi,4\pi)^{\circ}$: it is given as the region \[ \bigl\{(z_1,z_2)\in P(4\pi,4\pi)^{\circ}|G_2(z_1,z_2)\geq 2-2/\alpha,\,-G_1(z_1,z_2)+G_2(z_1,z_2)\leq 2/\alpha\bigr\},\] where $G_i=F_i\circ(\sigma\times\sigma)^{-1}$, \emph{i.e.}, \[ G_1(z_1,z_2)=2-\frac{|z_1|^2+|z_2|^2}{2}\] and \begin{align*} G_2(z_1,z_2)^2=&\left(\sqrt{1-\frac{|z_1|^2}{4}}\operatorname{Re}(z_1)+\sqrt{1-\frac{|z_2|^2}{4}}\operatorname{Re}(z_2) \right)^2\\ &\, +\left(\sqrt{1-\frac{|z_1|^2}{4}}\operatorname{Im}(z_1)+\sqrt{1-\frac{|z_2|^2}{4}}\operatorname{Im}(z_2)  \right)^2+\left(2-\frac{|z_1|^2+|z_2|^2}{2}\right)^2.\end{align*}
\end{ex}

Corollary \ref{explicit} also applies to some other convex toric domains besides the cube $P(a,a)$, though it as not as broadly applicable as Theorem \ref{4d}.  For example the reader may check that, in Corollary \ref{explicit}, for appropriate $\alpha$ one can take $X_{\Omega}$ equal to a polydisk $P(1,a)$ with $1\leq a\leq 1.2$, or  to an appropriately rescaled $\ell^p$ ball as in Theorem \ref{4d} for $p\geq 10$.

\begin{remark}
By construction, the embedding $\Phi$ from Proposition \ref{PhiS} maps the torus $T_{\sqrt{2}}:=\left\{(w,z)\in \C^2\left||w|=|z|=\sqrt{2}\right.\right\}$ to the Lagrangian torus in $S^2\times S^2$ that is denoted $K$ in \cite[Example 1.22]{EP}, and which can be identified with the Chekanov-Schlenk twist torus $\Theta$, see \cite{CS},\cite{OU}. 
Since, as shown in \cite{EP}, there is no symplectomorphism mapping $K$ to the Clifford torus in $S^2\times S^2$ (\emph{i.e.}, to the image of $T_{\sqrt{2}}$ under the standard embedding $(\sigma\times \sigma)^{-1}$ of $P(4\pi,4\pi)^{\circ}$ into $S^2\times S^2$), one easily infers independently of our other results that $(\sigma\times \sigma)\circ\Phi\co P(2\pi c,2\pi c)\hookrightarrow P(4\pi,4\pi)^{\circ}$ must not be isotopic to the inclusion by a compactly supported Hamiltonian isotopy for $1<c<4-2\sqrt{2}$ (for such a Hamiltonian isotopy could be extended to $S^2\times S^2$, giving a symplectomorphism that would send $K$ to the Clifford torus).  However this argument based on Lagrangian tori does not seem to adapt to yield the full result that $(\sigma\times\sigma)\circ\Phi$ is knotted in the stronger sense of Definition \ref{knotdef}.

By the way, if $c<1$, our embedding $(\sigma\times\sigma)\circ \Phi \co P(2\pi c,2\pi c)\hookrightarrow P(4\pi,4\pi)^{\circ}$ is unknotted. Indeed in this case the ball $B^{4}(4\pi c)$ is contained  both in $P(4\pi,4\pi)^{\circ}$ and in $E(4\pi,8\pi)\setminus \Phi^{-1}(\mathcal{I})$, and so both $(\sigma\times\sigma)\circ \Phi|_{P(2\pi c,2\pi c)}$ and the inclusion $P(2\pi c,2\pi c)\hookrightarrow P(4\pi,4\pi)^{\circ}$ extend to embeddings  $B^{4}(4\pi c,4\pi c)\hookrightarrow P(4\pi,4\pi)^{\circ}$; these two embeddings of the ball are symplectically isotopic by \cite[Proposition 1.5]{CG}.  Thus a transition between knottedness and unknottedness occurs at the value $c=1$, which is precisely the first value for which $P(2\pi c,2\pi c)$ contains the torus $T_{\sqrt{2}}$ mentioned at the start of the remark. \end{remark} 

\begin{remark} A similar construction to that in Proposition \ref{PhiS}, using results from \cite[Section 3]{OU}, allows one to construct a symplectic embedding of $E(3\pi, 12\pi)^{\circ}$ into $\mathbb{C}P^2$ where the symplectic form on $\mathbb{C}P^2$ is normalized to give area $6\pi$ to a complex projective line, such that the torus $T_{\sqrt{2}}$ is sent to the $\mathbb{C}P^2$ version of the Chekanov-Schlenk twist torus $\Theta$.  Combining this with a symplectomorphism from the complement of a line in $\C P^2$ to a ball and restricting to $P(2\pi c,2\pi c)$ for $c$ slightly larger than $1$, we obtain a symplectic embedding $P(2\pi c,2\pi c)\hookrightarrow B^{4}(6\pi)^{\circ}$ which cannot be Hamiltonian isotopic to the inclusion because $\Theta$ is not Hamiltonian isotopic to the Clifford torus.  It is less clear whether this embedding $P(2\pi c,2\pi c)\hookrightarrow B^{4}(6\pi)^{\circ}$  is knotted in the sense of Definition \ref{knotdef}; the symplectic-homology-based methods in the present paper seem ill-equipped to address this because the filtered positive $S^1$-equivariant symplectic homology of $B^4(6\pi)$ does not have as rich a structure as that of the domains $X$ that appear in Theorem \ref{4d}.
\end{remark}

\section{More knotted polydisks}\label{polysect}

The lower bounds on $\dellu$ that are used to show that our embeddings $X\hookrightarrow \alpha X^{\circ}$ are knotted are generally based on showing that, for suitable $k,L$, the maps $\Phi^L\co CH_{k}^{L}(\alpha X^{\circ},\lambda_{0})\to CH^{L}_{k}(X^{\circ},\lambda_{0})$ have sufficiently large image and then appealing to Corollary \ref{rankcor}.  One can in principle use Corollary \ref{rankcor} to prove the knottedness of embeddings $X\hookrightarrow V$ for more general star-shaped open subsets $V$ which are not dilates of $X^{\circ}$; the main difficulty in this case is that one can no longer simply appeal to Lemma \ref{opentriangle} in order to estimate the rank of $\Phi^L\co CH_{k}^{L}(V,\lambda_{0})\to CH_{k}^{L}(X^{\circ},\lambda_{0})$.  

In this section we carry this procedure out when $X$ and $V$ are four-dimensional polydisks $P(a_0,b_0)$, $P(a_1,b_1)^{\circ}$, typically with
$\frac{b_0}{a_0}\neq \frac{b_1}{a_1}$.

A polydisk $P(a,b)$ is the toric domain associated to the rectangle $R_{a,b}=[0,a]\times[0,b]$, which has $\|(x,y)\|^{*}_{R_{a,b}}=ax+by$, so by Lemma \ref{lem:keyconvex} we see that $\dim CH^{L}_{3}\bigl(P(a,b)^{\circ},\lambda_{0}\bigr)=2$ whenever $\max\{a,b\}<L<a+b$, and that \begin{align}\label{polydiski} \imath_{L_1,L_2}\co CH^{L_1}_{3}\bigl(P(a,b)^{\circ},\lambda_{0}\bigr)\to CH^{L_2}_{3}\bigl(P(a,b)^{\circ},\lambda_{0}\bigr) &\mbox{ is an isomorphism for }\\ \nonumber & \max\{a,b\}<L_1<L_2<a+b.\end{align}
 
\begin{lemma} \label{shrink-long-side}
Assume that $a\leq b<b'$ and that $b'<L<a+b$.  Then the transfer map $\Phi^L\co CH^{L}_{3}\bigl(P(a,b')^{\circ},\lambda_{0}\bigr)\to CH^{L}_{3}\bigl(P(a,b)^{\circ},\lambda_{0}\bigr)$ is an isomorphism.
\end{lemma}

\begin{proof}
Choose $\delta>0$ such that $(1+\delta)L<a+b$ and such that, for some $N\in \mathbb{N}$, $(1+\delta)^Nb=b'$.  Consider any $c\in [b,b']$. Lemma \ref{opentriangle} gives a commutative diagram \[ \xymatrix{ & CH_{3}^{L}\bigl((1+\delta)^{-1}P(a,c)^{\circ},\lambda_{0}\bigr)\ar[dd]^{\cong} \\ CH_{3}^{L}\bigl(P(a,c)^{\circ},\lambda_{0}\bigr)\ar[ru]^{\Phi^L}\ar[rd]_{\imath_{L,(1+\delta)L}} & \\ & CH_{3}^{(1+\delta)L}\bigl(P(a,c)^{\circ},\lambda_{0}\bigr) } \] where $\imath_{L,(1+\delta)L}$ is an isomorphism by Lemma \ref{lem:keyconvex} since our assumptions give $a\leq b\leq c\leq b'<L<(1+\delta)L<a+b\leq a+c$.  Thus $\Phi^L\co CH_{3}^{L}\bigl(P(a,c)^{\circ},\lambda_{0}\bigr)\to CH_{3}^{L}\Bigl(P\bigl((1+\delta)^{-1}a,(1+\delta)^{-1}c\bigr),\lambda_{0}\Bigr)$ is an isomorphism.  But this latter map factors as a composition \[ \xymatrix{
CH_{3}^{L}\bigl(P(a,c)^{\circ},\lambda_{0}\bigr)\ar[r]^{\Phi^L} & CH_{3}^{L}\bigl(P(a,(1+\delta)^{-1}c)^{\circ},\lambda_{0}\bigr)\ar[r]^{\Phi^L} & CH_{3}^{L}\Bigl(P\bigl((1+\delta)^{-1}a,(1+\delta)^{-1}c\bigr)^{\circ},\lambda_{0}\Bigr),}\] so since all three vector spaces above have dimension two it follows that, again for any $c\in [b,b']$, $\Phi^L\co CH_{3}^{L}\bigl(P(a,c)^{\circ},\lambda_{0}\bigr)\to CH_{3}^{L}\bigl(P(a,(1+\delta)^{-1}c),\lambda_{0}\bigr)$ is an isomorphism.

Since $(1+\delta)^{-N}b'=b$, we may apply this successively with $c=b',(1+\delta)^{-1}b',\ldots,(1+\delta)^{-(N-1)}b'$ and appeal to the functoriality of $\Phi^L$ to see that $\Phi^L\co CH^{L}_{3}\bigl(P(a,b')^{\circ},\lambda_{0}\bigr)\to CH^{L}_{3}\bigl(P(a,b)^{\circ},\lambda_{0}\bigr)$ is an isomorphism.
\end{proof}

A similar argument gives:

\begin{lemma}\label{shrink-short-side}
Assume that $a<a'\leq b$ and that $b<L<a+b$.  Then the transfer map $\Phi^L\co CH^{L}_{3}\bigl(P(a',b)^{\circ},\lambda_{0}\bigr)\to CH^{L}_{3}\bigl(P(a,b)^{\circ},\lambda_{0}\bigr)$ is an isomorphism.
\end{lemma} 

\begin{proof}
Analogously to the proof of Lemma \ref{shrink-long-side}, choose $\delta>0$ such that $(1+\delta)L<a+b$ and, for some $N\in \mathbb{N}$, $(1+\delta)^Na=a'$.  Using Lemma \ref{opentriangle} and (\ref{polydiski}), we find that for all $c\in [a,a']$ the transfer map $\Phi^L\co CH^{L}_{3}\bigl(P(c,b)^{\circ},\lambda_{0}\bigr)\to CH^{L}_{3}\Bigl(P\bigl((1+\delta)^{-1}c,(1+\delta)^{-1}b\bigr)^{\circ},\lambda_{0}\Bigr)$ is an isomorphism.  Since this map factors through $CH^{L}_{3}\bigl(P(1+\delta)^{-1}c,b),\lambda_{0}\bigr)$, we deduce by dimensional considerations that $\Phi^L\co CH^{L}_{3}\bigl(P(c,b),\lambda_{0}\bigr)\to CH^{L}_{3}\bigl(P((1+\delta)^{-1}c,b),\lambda_{0}\bigr)$ is an isomorphism for all $c\in [a,a']$.  Just as in the proof of Lemma \ref{shrink-long-side}, iterating this for $c=a',(1+\delta)^{-1}a',\ldots,(1+\delta)^{-(N-1)}a'$ yields the result.  
\end{proof}

\begin{prop}\label{polydisk-rank}
Assume that $a_1\leq b_1$, that $a_0\leq a_1$ and that $b_0\leq b_1$.  For any $L$ with $b_1<L<a_0+b_0$, the transfer map $\Phi^L\co CH^{L}_{3}\bigl(P(a_1,b_1)^{\circ},\lambda_{0}\bigr)\to CH^{L}_{3}\bigl(P(a_0,b_0)^{\circ},\lambda_{0}\bigr)$ is an isomorphism (and so in particular has rank two).
\end{prop}

\begin{proof}
If $a_0\leq b_0$ and $\frac{b_1}{a_1}\geq \frac{b_0}{a_0}$, then we can factor $\Phi^L\co CH^{L}_{3}(P(a_1,b_1)^{\circ})\to CH^{L}_{3}(P(a_0,b_0)^{\circ})$ as a composition \[ \xymatrix{ CH^{L}_{3}(P(a_1,b_1)^{\circ})\ar[r]^{\Phi^L} & CH^{L}_{3}\left(P\left(a_1,\frac{a_1b_0}{a_0}\right)^{\circ}\right)\ar[r]^{\Phi^L} & CH^{L}_{3}(P(a_0,b_0)^{\circ})   } \] where the first map is an isomorphism by Lemma \ref{shrink-long-side} and the second is an isomorphism by (\ref{polydiski}) and Lemma \ref{opentriangle} (which identifies the map with the inclusion-induced map $\imath_{\frac{a_0}{a_1}L,L}\co CH^{\frac{a_0}{a_1}L}_{3}\bigl(P(a_0,b_0)^{\circ},\lambda_{0}\bigr)\to CH_{3}^{L}\bigl(P(a_0,b_0)^{\circ},\lambda_{0}\bigr)$). 

Similarly, if $a_0\leq b_0$ and $\frac{b_1}{a_1}\leq \frac{b_0}{a_0}$, then we can factor
$\Phi^L\co CH^{L}_{3}\bigl(P(a_1,b_1)^{\circ},\lambda_{0}\bigr)\to CH^{L}_{3}\bigl(P(a_0,b_0)^{\circ},\lambda_{0}\bigr)$ as a composition \[ \xymatrix{ CH^{L}_{3}\bigl(P(a_1,b_1)^{\circ},\lambda_{0}\bigr)\ar[r]^{\Phi^L} & CH^{L}_{3}\left(P\left(\frac{a_0b_1}{b_0},b_1\right)^{\circ},\lambda_{0}\right)\ar[r]^{\Phi^L} & CH^{L}_{3}\bigl(P(a_0,b_0)^{\circ},\lambda_{0}\bigr)   } \] where the first map is an isomorphism  by  Lemma \ref{shrink-short-side} and the second is an isomorphism   by (\ref{polydiski}) and Lemma \ref{opentriangle}.

We have now proven the result whenever $a_0\leq b_0$.  If instead $a_0>b_0$, then the hypotheses imply that $P(a_0,b_0)\subset P(a_0,a_0)\subset P(a_1,b_1)$, and that $a_0\leq b_1<L<a_0+b_0<2a_0$.  We can then factor the map in question as \[  \xymatrix{ CH^{L}_{3}\bigl(P(a_1,b_1)^{\circ},\lambda_{0}\bigr)\ar[r]^{\Phi^L} & CH^{L}_{3}\bigl(P(a_0,a_0)^{\circ},\lambda_{0}\bigr)\ar[r]^{\Phi^L} & CH^{L}_{3}\bigl(P(a_0,b_0)^{\circ},\lambda_{0}\bigr)   } \] where the first map is an isomorphism by a case of the present corollary that we have already proven, and the second is an isomorphism by Lemma \ref{shrink-short-side} (after conjugating by a symplectomorphism that switches the factors of $\mathbb{C}^2$).
\end{proof}

\begin{cor} \label{ellpolycor}
If $E\subset \mathbb{R}^4$ is an ellipsoid, if $g\co E\to P(c,d)^{\circ}$ is a symplectic embedding where $c\leq d$, and if $P(a,b)\subset E\cap P(c,d)^{\circ}$, then $g|_{P(a,b)}$ is knotted provided that $d<a+b$. 
\end{cor}

\begin{proof}
If $g|_{P(a,b)}$ were unknotted, we could apply Corollary \ref{rankcor} with $f$ equal to the inclusion, with $k=3$, and with $L$ equal to any number with $d<L<a+b$.  This would yield $\op{Rank}\Bigl(\Phi^L\co CH^{L}_{3}\bigl(P(c,d)^{\circ},\lambda_{0}\bigr)\to CH^{L}_{3}\bigl(P(a,b)^{\circ},\lambda_{0}\bigr)\Bigr)\leq 1$, in contradiction with Proposition \ref{polydisk-rank}.
\end{proof}

\begin{remark}\label{flip}
In the case that $\max\{a,b\}<c$, so that $P(c,d)^{\circ}$ contains both $P(a,b)$ and $P(b,a)$, then one example of a symplectic embedding $P(a,b)\hookrightarrow P(c,d)^{\circ}$ is $\sigma\co (z_1,z_2)\mapsto (z_2,z_1)$, which has image equal to $P(b,a)$.  In \cite[Theorem 4]{FHW} it is shown that, when $c=d<a+b$, this embedding is not Hamiltonian isotopic to the inclusion within $P(c,c)^{\circ}$.  However our definition of knottedness is such that (when $c=d$) this embedding would be considered unknotted, because the symplectomorphism of $P(c,c)^{\circ}$ which swaps the factors maps $P(a,b)$ to $P(b,a)$ (and we do not require our ambient symplectomorphisms to be induced by Hamiltonian isotopies supported in the codomain). Likewise when $a=b$ but $c\neq d$, $\sigma$ is unknotted according to our definition because we take knottedness to depend only on the image of the embedding.  

In the situation that both $a\neq b$ and $c<d$ (and still $\max\{a,b\}\leq c$ and $d<a+b$) it can be shown that the above embedding $\sigma\co P(a,b)\to P(c,d)$ with image $P(b,a)$ is knotted. More specifically, by using arguments like those in \cite[Section 3.3]{FHW} one can show that for $a<L_1<b$ and $c<L_2<d$ the inclusion-induced map $SH_{3}^{[L_1,L_2)}(P(c,d)^{\circ})\to SH^{[L_1,L_2)}(P(a,b)^{\circ})$ on action-window symplectic homology 
vanishes, while the inclusion-induced map $SH_{3}^{[L_1,L_2)}(P(c,d)^{\circ})\to SH^{[L_1,L_2)}(P(b,a)^{\circ})$ is nontrivial, which is sufficient to show that $P(a,b)$ cannot be mapped to $P(b,a)$ by a symplectomorphism of $P(c,d)^{\circ}$; we omit the details.

However because Proposition \ref{polydisk-rank} shows that, for $d<L<a+b$, the map $\Phi^L\co CH_{3}^{L}\bigl(P(c,d)^{\circ},\lambda_{0}\bigr)\to CH_{3}^{L}\bigl(P(b,a)^{\circ},\lambda_{0}\bigr)$ has rank two, the embeddings described by Corollary \ref{ellpolycor} (for which $\Phi^L$ has rank one) have different knot types from $\sigma$.  (In other words, the image of such an embedding is not taken by a symplectomorphism of $P(c,d)^{\circ}$ to either one of $P(a,b)$ or $P(b,a)$.)  In particular this comment applies to the embeddings in Corollaries \ref{longknot} and \ref{allpoly} in each case that the target contains the image of the domain under $\sigma$.
\end{remark}

\begin{cor}\label{longknot}
Let $m\in \Z_+$ and $0\leq \ep <1$. If $a+\frac{b}{2m+\ep}<1$ and $0\leq a\leq b<m+\ep<a+b$ then there is a knotted embedding of $P(a,b)$ into $P(1,m+\ep)^{\circ}$.   
\end{cor}

\begin{proof}
Choose $\mu$ such that $a+\frac{b}{2m+\ep}<\mu<1$; we then have $P(a,b)\subset (\mu E(1,2m+\ep))\cap P(1,m+\ep)^{\circ}$.  Proposition \ref{longembed} moreover gives a symplectic embedding $\mu E(1,2m+\ep)\hookrightarrow P(1,m+\ep)^{\circ}$.  The conclusion then follows from Corollary \ref{ellpolycor}.
\end{proof}

We conclude by restating and proving Theorem \ref{polythm}:

\begin{cor}\label{allpoly} Given any $y\geq 1$, there exist polydisks $P(a,b)$ and $P(c,d)$ and knotted embeddings of $P(a,b)$ into $P(1,y)^{\circ}$ and of $P(1,y)$ into $P(c,d)^{\circ}$.
\end{cor}

\begin{proof}
For a knotted embedding $P(a,b)\hookrightarrow P(1,y)^{\circ}$, write $y=m+\ep$ where $m\in \Z_+$ and $0\leq \ep <1$.  We can then set $a=\frac{1}{2}$ and $b=m+\frac{3\ep-1}{4}$ and apply Corollary \ref{longknot}.

For a knotted embedding $P(1,y)\hookrightarrow P(c,d)^{\circ}$, write $y=2k+\delta$ where $k\in \Z_+$ and $-1\leq \delta <1$.  If $\delta\geq 0$, then Corollary \ref{longknot} gives a knotted embedding of $P(\mu,\mu y)$ into $P\left(1,k+\frac{\delta}{2}\right)^{\circ}$ for any $\mu$ with \[ \frac{1}{2}\frac{2k+\delta}{2k+\delta+1}<\mu<\frac{4k+\delta}{8k+3\delta},\] and so conjugating by a rescaling by $\mu$ gives the desired embedding (with $c=\frac{1}{\mu}$, $d=\frac{1}{\mu}\left(k+\frac{\delta}{2}\right)=\frac{y}{2\mu}$).  If instead $-1\leq \delta< 0$, then for $1+\frac{\delta}{4k}<\alpha < 1+\frac{1+\delta}{2k}$ Corollary \ref{longknot} (with $m=k,\ep=0$) gives a knotted embedding of $\frac{1}{2\alpha}P(1,y)$ into $P(1,k)^{\circ}$, and so again conjugating by a rescaling gives the desired embedding with $c=2\alpha,d=2\alpha k$.
\end{proof}

\appendix
\section{Proof of Lemma \ref{complex-exists}}

The purpose of this appendix is to prove Lemma \ref{complex-exists}. A related statement is proven in \cite{GG} for a slightly different version of $S^1$-equivariant symplectic homology; the main difference between our result and theirs is that they construct a filtered complex after choosing a certain action interval and prove that their complex computes the filtered $S^1$-equivariant symplectic homology associated to this action interval, whereas we construct a single complex that works simultaneously for all action intervals.  One can in fact show based on arguments similar to those below that the filtration on the complex constructed in \cite{GG} in the case of the action interval $(0,\infty)$ does have filtered homologies that recover their version of filtered $CH$ in arbitrary action intervals, but since this is not explicitly proven in \cite{GG} we give a detailed proof in our case. 

The main ingredient is an algebraic lemma concerning  filtered complexes which shows that, up to isomorphism, the images of inclusion-induced maps between the filtered parts of the complexes can be recovered from the filtered homology of a new chain complex whose underlying vector space is the $E^1$ term of the spectral sequence associated to the original filtered complex.  This lemma is proven in the following section, and in the subsequent section we apply this together with results from \cite{gutt},\cite{GuH} to complete the proof of Lemma \ref{complex-exists}.
We assume that the reader is familiar with positive $S^1$-equivariant symplectic homology and we use the notation from \cite{GuH}.

\subsection{A lemma on filtered complexes}

In this section we consider a $\Z$-graded chain complex $(C_*,\partial)$ of vector spaces over a field $K$ equipped with a filtration \[ \{0\}=F_0C_{*}\subset F_1C_{*}\subset\cdots\subset F_r C_{*}\subset \cdots \subset C_{*} \] (where each $F_rC_{*}$ is a subcomplex of $C_{*}$) that is bounded below by zero and exhausting (\emph{i.e.} $F_{\infty}C_*:=\cup_r F_rC_{*}$ is equal to $C_{*}$).  
We extend the above filtration by $\N$ to a filtration by $\Z$ by setting $F_iC_{*}=\{0\}$ for $i<0$.  

Recall that the associated graded complex of $(C_*,\partial)$, denoted $\mathcal{G}(C_*)$, is the direct sum of quotient complexes $\bigoplus_{p\geq 1}\frac{F_pC_*}{F_{p-1}C_*}$, equipped with obvious boundary operator induced from $\partial$.  The homology $H_*\big(\mathcal{G}(C_*)\big)$ evidently splits as a direct sum \[ H_k\big(\mathcal{G}(C_{*})\big)=\bigoplus_{p\geq 1}H_k\left(\frac{F_p C_*}{F_{p-1}C_*}\right).\]

The following is the main algebraic input needed for Lemma \ref{complex-exists}:

\begin{lemma}\label{maincx} 
With notation and assumptions as above, there is a chain complex $(D_*,\delta)$ equipped with a filtration \[ \{0\}=F_0D_*\subset F_1D_{*}\subset \cdots\subset F_rD_* \subset\cdots \subset D_* \] where for each $r,k$ \begin{equation}\label{dfilter} F_rD_k=\bigoplus_{1\leq p\leq r}H_k\left(\frac{F_p C_*}{F_{p-1}C_*}\right) \end{equation} and $F_{\infty}D_*:=\cup_rF_rD_*=D_*$, such that the boundary operator $\delta$ on $D_*$ strictly lowers filtration in the sense that $\delta(F_rD_*)\subset F_{r-1}D_*$, and such that for $1\leq s\leq t\leq\infty$ there exists an isomorphism of vector spaces 
\[
\Img\big(H_k(F_sC_{*},\partial)\to H_k(F_tC_*,\partial)\big)\cong \Img\big(H_k(F_sD_*,\delta)\to H_k(F_tD_*,\delta)\big)
\] where the maps on both sides are induced by inclusion of filtered subcomplexes.  \end{lemma}

The proof of Lemma \ref{maincx} will occupy the rest of this section.  To begin,
let us recall from \cite[Section 5.4]{weibel} some ingredients in the construction of the spectral sequence associated to the filtration on $(C_{*},\partial)$.

For $p\in \Z$ write $\eta_p\co F_pC_{*}\to \frac{F_pC_*}{F_{p-1}C_*}$ for the natural projection, and for $p,q,r\in \Z$ define: \[ A_{p,q}^{r}=\left\{x\in F_pC_{p+q}|\partial x\in F_{p-r}C_{p+q-1}\right\},\] \[ \hat{Z}_{p,q}^{r}=\eta_p(A_{p,q}^{r}),\qquad \hat{B}_{p,q}^{r}=\eta_p\left(\partial(A_{p+r-1,q-r+2}^{r-1})\right).\]  

For any $r\geq 1$ one then has inclusions \[ \{0\}= \hat{B}_{p,q}^{0}\subset \hat{B}_{p,q}^{1}\subset \cdots\subset \hat{B}_{p,q}^{r}\subset \hat{B}_{p,q}^{r+1}\subset \hat{Z}_{p,q}^{r+1}\subset \hat{Z}_{p,q}^{r}\subset\cdots\subset \hat{Z}_{p,q}^{0}=\frac{F_pC_{p+q}}{F_{p-1}C_{p+q}}.\]

We also write \[ \hat{B}_{p,q}^{\infty}=\cup_{r=1}^{\infty}\hat{B}_{p,q}^{r}=\cup_{r=1}^{\infty}\eta_p\left(\partial(A_{p+r-1,q-r+2}^{r-1})\right).\]
Note also that since we assume that $F_iC_*=\{0\}$ for $i\leq 0$, we have \[ \hat{Z}_{p,q}^{r}=\hat{Z}_{p,q}^{p}=\eta_p\left(\ker \partial|_{F_pC_{p+q}}\right) \mbox{ for }r\geq p.\]  Accordingly if we let $\hat{Z}_{p,q}^{\infty}=\hat{Z}_{p,q}^{p}$ then we will have \[ \hat{Z}_{p,q}^{\infty}=\cap_{r=1}^{\infty} \hat{Z}_{p,q}^{r}.\]

As is standard, we write \[ E_{p,q}^{r}=\frac{\hat{Z}_{p,q}^{r}}{\hat{B}_{p,q}^{r}} \] for $r\in \N\cup\{\infty\}$. 
For the case that $r=1$, notice that $\hat{Z}_{p,q}^{1}$ is equal to the set of degree-$(p+q)$ cycles in the quotient complex $\frac{F_pC_*}{F_{p-1}C_*}$ and that $\hat{B}_{p,q}^{1}$ is equal to the set of degree-$(p+q)$ boundaries in  $\frac{F_pC_*}{F_{p-1}C_*}$; thus \begin{equation}\label{E1id} E_{p,q}^{1}=H_{p+q}\left(\frac{F_pC_*}{F_{p-1}C_*}\right).\end{equation}

The following is standard and easily-checked:
\begin{prop}(cf. \cite[Construction 5.4.6]{weibel})\label{partialr} For each $p,q,r$, the boundary operator $\partial$ induces a map \[ \hat{\partial}^{r}_{p,q}\co E_{p,q}^{r}\to E_{p-r,q+r-1}^{r} \]
such that   \[ \ker(\hat{\partial}^{r}_{p,q})=\pi(\hat{Z}_{p,q}^{r+1})\quad\mbox{and} \quad \Img(\hat{\partial}^{r}_{p+r,q-r+1})=\pi(\hat{B}_{p,q}^{r+1}),\] where $\pi\co \hat{Z}_{p,q}^{r}\to \frac{\hat{Z}_{p,q}^{r}}{\hat{B}_{p,q}^{r}}$ is the quotient projection. \end{prop}

We also have the following fact concerning the maps $H_k(F_pC_*,\partial)\to H_k(F_tC_*,\partial)$ for $p\leq t$ induced by inclusion of filtered subcomplexes; this is a slight extension of the familiar fact that the spectral sequence of a suitable filtered complex converges to the associated graded of the homology.  
\begin{prop} \label{computerank} Let $1\leq p\leq t\leq \infty$ with $p<\infty$.  Then there is an isomorphism \[ 
 \frac{\Img\left( H_k(F_pC_*,\partial)\to H_k(F_tC_*,\partial)\right)}{\Img\left( H_k(F_{p-1}C_*,\partial)\to H_k(F_tC_*,\partial)\right)} \cong \frac{\hat{Z}_{p,k-p}^{\infty}}{\hat{B}_{p,k-p}^{t-p+1}}.
\]
\end{prop}
(Here for the case $t=\infty$ we interpret $F_{\infty}C_{*}$ as $C_*$ and $\hat{B}_{p,k-p}^{\infty-p+1}$ as $\hat{B}_{p,k-p}^{\infty}$.)
\begin{proof}
There is an obvious surjective map \[\phi\co \ker(\partial|_{F_pC_k})\to 
 \frac{\Img\left( H_k(F_pC_*,\partial)\to H_k(F_tC_*,\partial)\right)}{\Img\left( H_k(F_{p-1}C_*,\partial)\to H_k(F_tC_{*},\partial)\right)} \] given by including $\ker(\partial|_{F_pC_k})$ into $\ker(\partial|_{F_tC_k})$, then taking homology classes, and then projecting.  We see that $x\in \ker(\phi)$ if and only if there is $y\in \ker(\partial|_{F_{p-1}C_k})$ such that $x$ and $y$ represent the same homology class in $H_k(F_t C_*,\partial)$; this holds if and only if we can write $x=y+\partial z$ with $z\in F_tC_{k+1}$, and in this case we would have $z\in A_{t,k-t+1}^{t-p}$ since $\partial z=x-y\in F_pC_{k}$.  Thus $\ker(\phi)=\ker(\partial|_{F_{p-1}C_k})+\partial(A_{t,k-t+1}^{t-p})$ and hence \begin{equation}\label{gpa}
\frac{\Img\left( H_k(F_pC_*,\partial)\to H_k(F_tC_*,\partial)\right)}{\Img\left( H_k(F_{p-1}C_*,\partial)\to H_k(F_tC_*,\partial)\right)} \cong \frac{ \ker(\partial|_{F_pC_k}) }{ \ker(\partial|_{F_{p-1}C_k})+\partial(A_{t,k-t+1}^{t-p})   }. \end{equation} (The above discussion implicitly assumes that $t<\infty$, but since $\cup_{s=1}^{\infty}F_sC_*=C_*$ the reasoning is equally valid for $t=\infty$ provided that we interpret the notation $A_{\infty,k-\infty+1}^{t-p}$ as $\cup_{p\leq s\in \N}A_{s,k-s+1}^{s-p}$, as we will continue to do below).  

On the other hand the projection $\eta_p\co F_pC_k\to \frac{F_pC_k}{F_{p-1}C_k}$ sends $\ker(\partial|_{F_pC_k})$ to $\hat{Z}_{p,k-p}^{\infty}$ and sends $ \ker(\partial|_{F_{p-1}C_k})+\partial(A_{t,k-t+1}^{t-p}) $ to $\hat{B}_{p,k-p}^{t-p+1}$, and it is easy to check that the resulting map \[ \eta\co \frac{ \ker(\partial|_{F_pC_k}) }{ \ker(\partial|_{F_{p-1}C_k})+\partial(A_{t,k-t+1}^{t-p})   }\to \frac{\hat{Z}_{p,k-p}^{\infty}}{\hat{B}_{p,k-p}^{t-p+1}} \] is an isomorphism.  Combining this isomorphism with (\ref{gpa}) proves the proposition.
\end{proof}

For $1\leq r\leq \infty$ let \[ B_{p,q}^{r}=\frac{\hat{B}_{p,q}^{r}}{\hat{B}_{p,q}^{1}}\qquad Z_{p,q}^{r}=\frac{\hat{Z}_{p,q}^{r}}{\hat{B}_{p,q}^{1}},\] so for $r<p$ we have a chain of inclusions  \[ \{0\}=B_{p,q}^{1}\subset\cdots\subset B_{p,q}^{r}\subset B_{p,q}^{r+1}\subset\cdots\subset B_{p,q}^{\infty}\subset Z_{p,q}^{\infty}=Z_{p,q}^{p}\subset Z_{p,q}^{r+1}\subset Z_{p,q}^{r}\subset\cdots\subset Z_{p,q}^{1}=E_{p,q}^{1}.\]  Projecting away $\hat{B}_{p,q}^{1}$ induces isomorphisms $E_{p,q}^{r}\cong \frac{Z_{p,q}^{r}}{B_{p,q}^{r}}$.  For each $p,q\in \Z$ and $r\geq 1$ let us choose: 
\begin{itemize} \item A complement $H_{p,q}^{r}$ to the subspace $B_{p,q}^{r}$ within the vector space $Z_{p,q}^{r}$, and 
\item A complement $M_{p,q}^{r}$ to the subspace $Z_{p,q}^{r+1}$ within the vector space $Z_{p,q}^{r}$.  \end{itemize}
Given these choices, the projection $Z_{p,q}^{r}\to E_{p,q}^{r}$ restricts to $H_{p,q}^{r}$ as an isomorphism, so the maps $\hat{\partial}_{p,q}^{r}$ from Proposition \ref{partialr} induce maps \[ \partial_{p,q}^{r}\co H_{p,q}^{r}\to H_{p-r,q+r-1}^{r} \] with \[ \ker\partial_{p,q}^{r}=Z_{p,q}^{r+1}\cap H_{p,q}^{r},\qquad \Img\partial_{p+r,q-r+1}^{r}=B_{p,q}^{r+1}\cap H_{p,q}^{r}.\]

(In particular, since $Z_{p,q}^{r+1}=Z_{p,q}^{r}$ for $r\geq p$, we have $\partial_{p,q}^{r}=0$ for $r\geq p$).

For any $r\geq 2$ the various direct sum decompositions $Z_{p,q}^{j-1}=Z_{p,q}^{j}\oplus M_{p,q}^{j-1}$ yield a direct sum decomposition \begin{align*} E_{p,q}^{1}&=Z_{p,q}^{1}=Z_{p,q}^{r}\oplus M_{p,q}^{r-1}\oplus\cdots\oplus M_{p,q}^{1} 
\\ &= H_{p,q}^{r}\oplus B_{p,q}^{r}\oplus M_{p,q}^{r-1}\oplus\cdots\oplus M_{p,q}^{1}.\end{align*}

(For $r=1$ we have $B_{p,q}^{1}=\{0\}$ and $H_{p,q}^{1}=E_{p,q}^{1}$ and the above direct sum decomposition degenerates to $E_{p,q}^{1}=H_{p,q}^{1}$).

We accordingly extend our map $\partial^{r}_{p,q}\co H_{p,q}^{r}\to H_{p,q}^{r}$ to a linear map (still denoted $\partial^{r}_{p,q}$) defined on all of $E_{p,q}^{1}$ by setting it equal to zero on the summands $B_{p,q}^{r},M_{p,q}^{r-1},\ldots,M_{p,q}^{1}$. We also regard the codomain of $\partial^{r}_{p,q}$ as $E_{p-r,q+r-1}^{1}$ rather than the subspace $H_{p-r,q+r-1}^{r}$.  
With this extended definition, we have \[ \ker\partial^{r}_{p,q}=(Z_{p,q}^{r+1}\cap H_{p,q}^{r})\oplus B_{p,q}^{r}\oplus M_{p,q}^{r-1}\oplus\cdots\oplus M_{p,q}^{1} = Z_{p,q}^{r+1}\oplus M_{p,q}^{r-1}\oplus \cdots\oplus M_{p,q}^{1}, \]
where we have used that $B_{p,q}^{r}\subset Z_{p,q}^{r+1}\subset Z^{r}_{p,q}=H_{p,q}^{r}\oplus B_{p,q}^{r}$, so that $(Z_{p,q}^{r+1}\cap H_{p,q}^{r})\oplus B_{p,q}^{r}=Z_{p,q}^{r+1}$.
Since we have a direct sum decomposition \[ E_{p,q}^{1}=Z_{p,q}^{r+1}\oplus M_{p,q}^{r}\oplus M_{p,q}^{r-1}\cdots\oplus M_{p,q}^{1} ,\] it follows that:

\begin{cor}\label{chop} The maps $\partial_{p,q}^{r}\co E^{1}_{p,q}\to E_{p-r,q+r-1}^{1}$ restrict as isomorphisms $M_{p,q}^{r}\to B_{p-r,q+r-1}^{r+1}\cap H_{p-r,q+r-1}^{r}$, and vanish identically on the complementary subspace $Z_{p,q}^{r+1}\oplus M_{p,q}^{r-1}\oplus\cdots\oplus M_{p,q}^{1}$ to $M_{p,q}^{r}$ in $E_{p,q}^{1}$.
\end{cor}

In particular, since for $j>r$ we have $M_{p,q}^{j}\subset Z_{p,q}^{j}\subset Z_{p,q}^{r+1}\subset \ker(\partial_{p,q}^{r})$, this shows that $\partial_{p,q}^{r}$ vanishes on $M_{p,q}^{j}$ for $j\neq r$, while it maps $M_{p,q}^{r}$ isomorphically to $B_{p-r,q+r-1}^{r+1}\cap H_{p-r,q+r-1}^{r}$.

Now for any $p,q$ let us write \[ \partial_{p,q}=\sum_{r\geq 1} \partial_{p,q}^{r}\co E_{p,q}^{1}\to\oplus_{r\geq 1}H_{p-r,q+r-1}^{r}\subset E_{p-r,q+r-1}^{1}.\]  (This has just finitely many nonzero terms since $\partial_{p,q}^{r}=0$ for $r\geq p$.)  Also define, for $k\in \Z$, \[ D_k=\bigoplus_{p+q=k}E_{p,q}^{1},\] and define $\delta_k\co D_k\to D_{k-1}$ as the map which restricts to $\partial_{p,q}$ on the respective summands $E_{p,q}^{1}$.  Each $D_k$ has a filtration given by \[ F_sD_k=\bigoplus_{p+q=k,p\leq s}E_{p,q}^{1},\] which is consistent with (\ref{dfilter}) by (\ref{E1id}). By definition, the map $\delta_k$ respects this filtration, and indeed satisfies the stronger property $\delta_k(F_sD_k)\subset F_{s-1}D_{k-1}$.  

We will now compute the kernel and image of $\delta_k$.  For a general element $x=\sum_{p}x_p\in D_k$ where each $x_p\in E_{p,k-p}^{1}$, 
the component of $\delta_kx$ in the summand $E_{m,k-1-m}^{1}\subset D_{k-1}$ is equal to \[ \sum_r \partial_{m+r,k-m-r}^{r}x_{m+r}.\]  Now $\partial_{m+r,k-m-r}^{r}x_{m+r}$ lies in the subspace $B_{m,k-1-m}^{r+1}\cap H_{m,k-1-m}^{r}$ of $E_{m,k-1-m}^{1}$.  
But these latter subspaces are independent as $r$ varies: indeed given finitely many elements $y_r\in B_{m,k-1-m}^{r+1}\cap H_{m,k-1-m}^{r}$ that are not all zero, if $r_{\max}$ is chosen maximal subject to the property that $y_{r_{\max}}\neq 0$ then the fact that $0\neq y_{r_{\max}}\in H_{m,k-1-m}^{r_{\max}}$ while for all $s<r_{\max}$ we have $y_s\in B_{m,k-1-m}^{s+1}\subset B_{m,k-1-m}^{r_{\max}}$ would imply that $\sum y_r\neq 0$ since $H_{m,k-1-m}^{r_{\max}}$ is complementary to $B_{m,k-1-m}^{r_{\max}}$.

The independence of these subspaces implies that, for $x_p\in E_{p,k-p}^{1}$, the component of $\delta_k\left(\sum_p x_p\right)$ in $E_{m,k-m-1}^{1}$ is zero only if each $\partial_{m+r,k-m-r}^{r}x_{m+r}$ separately vanishes.  Thus: \begin{equation}\label{allpr} \sum_px_p\in\ker\delta_k \Leftrightarrow (\forall p,r)(\partial_{p,k-p}^{r}x_p=0).\end{equation}  Now fixing $p$ and recalling that $Z_{p,k-p}^{p}=Z_{p,k-p}^{\infty}$ and $\partial_{p,k-p}^{r}=0$ for $r\geq p$, note that we have \[ E_{p,k-p}^{1}=Z_{p,k-p}^{\infty}\oplus M_{p,k-p}^{p-1}\oplus\cdots\oplus M_{p,k-p}^{1}.\]  Moreover, for $r<p$, $\partial_{p,k-p}^{r}$ vanishes on $Z_{p,k-p}^{r+1}\supset Z_{p,k-p}^{\infty}$ and on each $M_{p,k-p}^{j}$ for $j\neq r$ while restricting injectively to $M_{p,k-p}^{r}$.  Hence $\partial_{p,k-p}^{r}x_p=0$ for \emph{all} $r$ if and only if $x_p\in Z_{p,k-p}^{\infty}$.  In combination with (\ref{allpr}) this shows:

\begin{prop}\label{prop:ker} 
\[ \ker(\delta_k\co D_k\to D_{k-1})=\bigoplus_p Z_{p,k-p}^{\infty} \] and, for each $s\in \N$, \[ \ker(\delta_k|_{F_sD_k})=\bigoplus_{p\leq s}Z_{p,k-p}^{\infty}.\]
\end{prop}

Next we will show:
\begin{prop}\label{prop:im}
\[ \Img(\delta_k\co D_k\to D_{k-1})=\bigoplus_p B_{p,k-1-p}^{\infty} \] and, 
for $s\in \N$, \[ \Img(\delta_k|_{F_sD_k})=\bigoplus_{p<s}B_{p,k-1-p}^{s-p+1}.
\]
\end{prop}
\begin{proof} As noted earlier the summand of $\delta_k\left(\sum_px_p\right)$
 in $E_{m,k-1-m}^{1}$ is $\sum_r \partial_{m+r,k-m-r}^{r}x_{m+r}$, which is a 
sum of terms in the mutually independent subspaces $B_{m,k-1-m}^{r+1}\cap H_{
m,k-1-m}^{r}$.  Note that, for fixed $k,m$ and any $t\in \N$, \begin{equation}
\label{bdecomp} \bigoplus_{1\leq r\leq t}\left(B_{m,k-1-m}^{r+1}\cap H_{m,k-1-
m}^{r}\right)=B_{m,k-1-m}^{t+1}: \end{equation} indeed using the inclusions \[
 B_{m,k-1-m}^{r}\subset B_{m,k-1-m}^{r+1}\subset Z_{m,k-1-m}^{r}=H_{m,k-1-m}^{
r}\oplus B_{m,k-1-m}^{r}\]  we see that $B_{m,k-1-m}^{r+1}=(B_{m,k-1-m}^{r+1}
\cap H_{m,k-1-m}^{r})\oplus B_{m,k-1-m}^{r}$; applying this inductively 
starting from $B_{m,k-1-m}^{1}=\{0\}$ yields (\ref{bdecomp}).  The same 
reasoning shows that $\bigoplus_{r=1}^{\infty}\left(B_{m,k-1-m}^{r+1}\cap H_{
m,k-1-m}^{r}\right)=B_{m,k-1-m}^{\infty}$. Thus to prove the proposition it 
suffices to show that, given $p\in \N$ and elements $y_{r}\in B_{p-r,k+r-p-1}^
{r+1}\cap H_{p-r,k+r-p-1}^{r}$ for $1\leq r< p$, we can find a single $x\in
 E_{p,k-p}^{1}$ with $\partial_{p,k-p}^{r}x=y_r$ for each $r$.  But this is 
an easy consequence of Corollary \ref{chop}: using the decomposition $E_{p,k-p
}^{1}=Z_{p,k-p}^{\infty}\oplus M_{p,k-p}^{p-1}\oplus \cdots\oplus M_{p,k-p}^
{1}$ we can take $x$ to be an element with trivial component in $Z_{p,k-p}^{
\infty}$ and with component in each respective $M_{p,k-p}^{r}$ equal to a preimage of $y_r$ 
under $\partial_{p,k-p}^{r}$.\end{proof}

\begin{cor}\label{dhomology}
Let $D_*=\oplus_{k}D_k$ and $\delta=\oplus_k\delta_k$.  Then $(D_*,\delta)$ is a filtered chain complex whose total homology is given by \[ H_k(D_*,\delta)=\frac{\oplus_pZ_{p,k-p}^{\infty}}{\oplus_p B_{p,k-p}^{\infty}}.\]
Moreover, for $s\in \N,t\in\N\cup\{\infty\}$ with $s\leq t$ we have \[ \Img\big(H_k(F_sD_*,\delta)\to H_k(F_tD_*,\delta)\big)= \frac{\bigoplus_{p\leq s}Z_{p,k-p}^{\infty}}{\bigoplus_{p\leq s}B_{p,k-p}^{t-p+1}}.\] 
\end{cor}

\begin{proof}
That $(D_*,\delta)$ is a chain complex simply results from  Propositions \ref{prop:ker} and \ref{prop:im} and the fact that $B_{p,q}^{\infty}\subset Z_{p,q}^{\infty}$; the computation of $H_k(D_*,\delta)$ likewise follows immediately. The computation of $\Img\big(H_k(F_sD_*,\delta)\to H_k(F_tD_*,\delta)\big)$  also follows because this image is essentially by definition equal to the quotient of $\ker(\delta|_{F_sD_k})$ by $\Img(\delta|_{F_tD_{k+1}})\cap F_sD_k$.  (For the case that $t=s$, it perhaps also bears noting that $B_{s,k-s}^{1}=\{0\}$, so that $\oplus_{p<s} B_{p,k-p}^{s-p+1}=\oplus_{p\leq s}B_{p,k-p}^{s-p+1}$). \end{proof}

Lemma \ref{maincx} now follows almost immediately from Corollary \ref{dhomology} and Proposition \ref{computerank}.  Indeed, projecting away $\hat{B}_{p,k-p}^{1}$ gives isomorphisms $\frac{\hat{Z}_{p,k-p}^{r}}{\hat{B}_{p,k-p}^{r}}\cong \frac{Z_{p,k-p}^{r}}{B_{p,k-p}^{r}}$ so Corollary \ref{dhomology} and Proposition \ref{computerank} show that we have, whenever $s\in \N$ and $1\leq s\leq t\leq \infty$, \[ 
\Img(H_k(F_sD_*,\delta)\to H_k(F_tD_*,\delta) )\cong  \bigoplus_{p=1}^{s}
 \frac{\Img\left( H_k(F_pC_*,\partial)\to H_k(F_tC_*,\partial)\right)}{\Img\left( H_k(F_{p-1}C_*,\partial)\to H_k(F_tC_*,\partial)\right)}.\]  Since $F_0C_*=\{0\}$, we can then iteratively choose complements to $\Img\left( H_k(F_{p-1}C_*,\partial)\to H_k(F_tC_*,\partial)\right)$ in $\Img\left( H_k(F_{p}C_*,\partial)\to H_k(F_tC_*,\partial)\right)$ to obtain an isomorphism $\Img\big(H_k(F_sD_*,\delta)\to H_k(F_tD_*,\delta)\big)\cong \Img\big(H_k(F_sC_*,\partial)\to H_k(F_tC_*,\partial)\big)$.  Moreover in the case that $t=\infty$, as $s$ varies this can be done in such a way that if $s<s'$ then the isomorphism $\Img\big(H_k(F_{s'}D_*,\delta)\to H_k(D_*,\delta)\big)\cong \Img\big(H_k(F_{s'}C_*,\partial)\to H_k(C_*,\partial)\big)$ restricts to  $\Img\big(H_k(F_{s}D_*,\delta)\to H_k(D_*,\delta)\big)$ as the already-chosen isomorphism  $\Img\big(H_k(F_sD_*,\delta)\to H_k(D_*,\delta)\big)\cong \Img\big(H_k(F_sC_*,\partial)\to H_k(C_*,\partial)\big)$; hence by taking the union over $s$ we obtain an isomorphism $H_k(D_*,\delta)\cong H_k(C_*,\delta)$ (corresponding to the case $s=t=\infty$ in Lemma \ref{maincx}).

Since we have already seen that our complex $(D_*,\delta)$ satisfies the other required properties, this completes the proof of Lemma \ref{maincx}.

\subsection{Construction of $CC_*(X,\lambda)$}

Since we assume that the Reeb flow on the boundary of $(X,\lambda)$ is nondegenerate, the set of actions (equivalently, periods) of the Reeb orbits on $\partial X$ is discrete; of course every element of this set is positive, so let us denote by $T_1<T_2<\cdots<T_r<\cdots$ the numbers which arise as actions of Reeb orbits on $\partial X$.  Also write  $T_0=0$.  By \cite[Proposition 3.1]{GuH}, the maps $\imath_{L_1,L_2}\co CH^{L_1}(X,\lambda)\to CH^{L_2}(X,\lambda)$ give a directed system (\emph{i.e.} $\imath_{L_2,L_3}\circ \imath_{L_1,L_2}=\imath_{L_1,L_3}$), and $\imath_{L_1,L_2}$ is an isomorphism if the interval $(L_1,L_2]$ does not contain any of the actions $T_i$.  So in particular if $L\leq L'$ with $L\in [T_i,T_{i+1}),\,L'\in [T_j,T_{j+1})$ then there is a commutative diagram \begin{equation}\label{intervaldiag} \xymatrix{ CH^L(X,\lambda)\ar[r]^{\imath_{L,L'}} & CH^{L'}(X,\lambda) \\ CH^{T_i}(X,\lambda) \ar[u]^{\imath_{T_i,L}}_{\cong} \ar[r]_{\imath_{T_i,T_j}} & CH^{T_j}(X,\lambda)\ar[u]_{\imath_{T_j,L'}}^{\cong}} \end{equation} where both vertical arrows are isomorphisms.  So to understand the maps $\imath_{L_1,L_2}$ it suffices to understand the maps $\imath_{T_i,T_j}$.

By definition (\cite[Definition 6.1]{GuH}), we have \[ CH^L(X,\lambda)=\varinjlim_{N,H}HF^{S^1,N,+,\leq L}(H,J) \] where the direct limit is taken over parametrized Hamiltonians $H\co S^1\times \hat{X}\times S^{2N+1}\to \R$ on the Liouville completion $\hat{X}$ of $X$ that satisfy a certain admissibility condition, with the structure maps being given by parametrized versions of continuation maps associated to pairs $(N,H),(N',H')$ with $N\leq N',H\leq H'|_{S^1\times \hat{X}\times S^{2N+1}}$.  Here $HF^{S^1,N,+,\leq L}(H,J)$ is the homology of the subcomplex (which for brevity we will denote by $C(N,H)^L$)  generated by orbits of symplectic action at most $L$ of the positive equivariant Floer complex $C(N,H)^{\infty}:=\frac{CF^{S^1,N,+}(N,H)}{CF^{S^1,N,+,\leq \epsilon}(N,H)}$ where $0<\epsilon\ll T_1$.  The maps $\imath_{L_1,L_2}\co CH^{L_1}(X,\lambda)\to CH^{L_2}(X,\lambda)$ are by definition the maps induced on the direct limit by the maps $HF^{S^1,N,+,\leq L_1}(N,H)\to HF^{S^1,N,+,\leq L_2}(N,H)$ given by the inclusion of subcomplexes $C(N,H)^{L_1}\hookrightarrow C(N,H)^{L_2}$.

Suppose that $\{(N_i,H_i)\}_{i=1}^{\infty}$ is any cofinal, linearly ordered subset of the partially ordered set of pairs $(N,H)$ used to define $CH^L(X,\lambda)$.  We can then form the direct limit of the chain complexes $C(N_i,H_i)^{\infty}$, using as structure maps the compositions of chain level continuation maps $C(N_i,H_i)^{\infty}\to C(N_{i+1},H_{i+1})^{\infty}$.  Denote this direct limit by $\underrightarrow{C}$.   Since the continuation maps preserve the filtration by symplectic action,  for any $L\in \R$ we likewise have a direct limit $\underrightarrow{C}^L=\varinjlim_i C(N_i, H_i)^{L}$, and the $\underrightarrow{C}^L$ form an $\R$-valued filtration of $\underrightarrow{C}$.  

Let us coarsen this $\R$-filtration to an $\N$ filtration by, for each $p\in \N$, choosing $T'_p$ with $T_p<T'_p<T_{p+1}$, and letting \[ F_p\underrightarrow{C}=\underrightarrow{C}^{T'_p}\] (Recall our notation that $T_0=0$ and the $T_p$ for $p>0$ are the distinct actions of Reeb orbits along $\partial X$, in increasing order.) As in \cite[Remark 5.6]{GuH}, for $i$ sufficiently large every generator of $C(N_i,H_i)$ will have filtration level greater than $T'_0$, so that $C(N_i,H_i)^{T'_0}=\{0\}$ for $i$ sufficiently large and so $F_0\underrightarrow{C}=\{0\}$. The fact that $\cup_{p}C(N_i,H_i)^{T'_p}=C(N_i,H_i)$ for each $i$ implies that likewise $\cup_pF_p \underrightarrow{C}=\underrightarrow{C}$.  All of our complexes are $\Z$-graded because of the assumption that $c_1(TX)|_{\pi_2(X)}=0$.  Thus Lemma \ref{maincx} applies to the filtered complex $\underrightarrow{C}$, producing a filtered complex $(D_*,\delta)$ with \[ F_rD_{*}=\bigoplus_{1\leq p\leq r}H_*\left(\frac{\varinjlim_iC(N_i,H_i)^{T'_{p}}}{\varinjlim_iC(N_i,H_i)^{T'_{p-1}}}\right) \] such that for each $k\in \Z,s\leq t$ we have \[ \Img\big(H_k(F_sD_{*},\delta)\to H_k(F_tD_*,\delta)\big)\cong \Img\left(H_k(\underrightarrow{C}^{s})\to H_k(\underrightarrow{C}^{t})\right);\] note that (for finite $t$) the right-hand side is precisely the image of $\imath_{T_s,T_t}$ in grading $k$.  Also, since $\varinjlim$ is an exact functor, we have \[ H_*\left(\frac{\varinjlim_iC(N_i,H_i)^{T'_{p}}}{\varinjlim_iC(N_i,H_i)^{T'_{p-1}}}\right)\cong \varinjlim_iH_*\left(\frac{C(N_i,H_i)^{T'_p}}{C(N_i,H_i)^{T'_{p-1}}}\right). \]  Thus we have a filtered complex $(D_*,\delta)$ whose $r$-filtered part is \[ F_rD_*=\varinjlim_i\bigoplus_{1\leq p\leq r}H_*\left(\frac{C(N_i,H_i)^{T'_p}}{C(N_i,H_i)^{T'_{p-1}}}\right) \] and such that, for $1\leq s\leq t<\infty$, \begin{equation}\label{chdiso} \Img\left(\imath_{T'_s,T'_t}\co CH^{T'_s}_{k}(X,\lambda)\to CH^{T'_t}_{k}(X,\lambda) \right)\cong \Img\big(H_k(F_sD_{*},\delta)\to H_k(F_tD_*,\delta)\big).\end{equation}

The foregoing discussion applies to an arbitrary cofinal linearly ordered subset $\{(N_i,H_i)\}_{i=1}^{\infty}$ of the set of admissible pairs $(N,H)$.  For a particular choice of such a cofinal subset consisting of Hamiltonians as described in \cite[Section 3.1]{gutt} and \cite[Remark 5.15]{GuH}, the homologies $H_*\left(\frac{C(N_i,H_i)^{T'_p}}{C(N_i,H_i)^{T'_{p-1}}}\right)$ are computed in \cite[Section 3.2]{gutt},\cite[Section 6.7]{GuH}.  Namely, the space $H_*\left(\frac{C(N_i,H_i)^{T'_p}}{C(N_i,H_i)^{T'_{p-1}}}\right)$ is generated by elements $\widecheck{\gamma}$ and $u^{N_i}\otimes \widehat{\gamma}$ as $\gamma$ ranges over good Reeb orbits having action equal to $T_p$; writing $CZ$ for the Conley--Zehnder index, the grading of $\widecheck{\gamma}$ is $CZ(\gamma)$ and that of $u^{N_i}\otimes \widehat{\gamma}$ is $CZ(\gamma)+2N_i+1$.  The continuation maps $H_*\left(\frac{C(N_i,H_i)^{T'_p}}{C(N_i,H_i)^{T'_{p-1}}}\right)\to H_*\left(\frac{C(N_{i+1},H_{i+1})^{T'_p}}{C(N_{i+1},H_{i+1})^{T'_{p-1}}}\right)$ moreover map $\widecheck{\gamma}$ to $\widecheck{\gamma}$ and $u^{N_i}\otimes\widehat{\gamma}$ to zero, as one can see based on \cite[Remark 3.7]{BOjta}.  Thus in any given degree $k$ the direct limit $\varinjlim_i H_k\left(\frac{C(N_i,H_i)^{T'_p}}{C(N_i,H_i)^{T'_{p-1}}}\right)$  has basis in bijection with the good Reeb orbits on $\partial X$ of action $T_p$ and Conley-Zehnder index $k$. 

So the $\N$-filtered complex $(D_*,\delta)$ produced by Lemma \ref{maincx} has the property that $F_rD_k$ is the span of a set of generators in bijection with the good Reeb orbits on $\partial X$ of Conley--Zehnder index $k$ and action at most $T_r$.  The complex $CC_{*}(X,\lambda)$ promised in Lemma \ref{complex-exists} is then given by converting $(D_*,\delta)$ into an $\R$-filtered complex by taking the $L$-filtered part $CC^{L}_{*}(X,\lambda)$ to be equal $F_rD_*$ where $r$ is maximal subject to the condition that $T_r\leq L$.  In particular we have equalities $CC^{L}_{*}(X,\lambda)=CC^{L'}_{*}(X,\lambda)$ whenever $L,L'\in [T_r,T_{r+1})$.  Since $\delta$ strictly decreases the $\N$-filtration on $D_*$, it likewise strictly decreases this $\R$-filtration.

 By (\ref{chdiso}), we have isomorphisms \[ \Img\left(\imath_{T'_s,T'_t}\co CH^{T'_s}_{k}(X,\lambda)\to CH^{T'_t}_{k}(X,\lambda)\right)\cong \Img\left(H_k\big(CC^{T'_s}_{*}(X,\lambda)\big)\to H_k\big(CC^{T'_t}_{*}(X,\lambda)\big)\right)
 \] for $s\leq t$, and then by applying (\ref{intervaldiag}) we obtain a similar isomorphism with $T'_s,T'_t$ replaced by arbitrary $L,L'$ with $L\leq L'$. The special case that $L=L'$ shows that $CH_{k}^{L}(X,\lambda)$ is isomorphic to $H_k\big(CC_{*}^{L}(X,\lambda)\big)$ since in this case the relevant inclusion-induced map is the identity. This completes the proof that the filtered complex $CC_{*}(X,\lambda)=\cup_L CC^{L}_{*}(X,\lambda)$ with boundary operator $\delta$ satisfies the properties required by Lemma \ref{complex-exists}.


\begin{thebibliography}{99}
\bibitem[BCE07]{BCE} F. Bourgeois, K. Cieliebak, and T. Ekholm. \emph{A note on Reeb dynamics on the tight $3$-sphere}. J. Mod. Dyn. \textbf{1} (2007), no. 4, 597--613
\bibitem[BO10]{BOjems} F. Bourgeois and A. Oancea. \emph{Fredholm theory and transversality for the parametrized and for the $S^1$-equivariant symplectic action}. J. Eur. Math. Soc. (JEMS) \textbf{12} (2010), no. 5, 1181--1229.
\bibitem[BO13a]{BOjta} F. Bourgeois and A. Oancea. \emph{The Gysin exact sequence for $S^1$-equivariant symplectic homology}. J. Topol. Anal. \textbf{5} (2013), no. 4, 361--407.
\bibitem[BO13b]{BOind} F. Bourgeois and A. Oancea. \emph{The index of Floer moduli problems for parametrized action functionals}. Geom. Dedicata \textbf{165} (2013), 5--24.
\bibitem[BO16]{bo} F. Bourgeois and A. Oancea. \emph{$S^1$-equivariant symplectic homology and linearized contact homology}. Int. Math. Res. Notices (2016), article ID rnw029, 1--89.
\bibitem[CS10]{CS} Y. Chekanov and F. Schlenk.  \emph{Notes on monotone Lagrangian twist tori}. Electron. Res. Announc. Math. Sci. \textbf{17} (2010), 104--121.
\bibitem[CCGF+14]{concave} K. Choi, D. Cristofaro-Gardiner, D. Frenkel, M. Hutchings, and V. Ramos. \emph{Symplectic embeddings into four-dimensional concave toric domains}. J. Topol. \textbf{7} (2014), 1054--1076.
\bibitem[C14]{CG} D. Cristofaro-Gardiner. \emph{Symplectic embeddings from concave toric domains into convex ones}. arXiv:1409.4378.
\bibitem[CFS17]{CFS} D. Cristofaro-Gardiner, D. Frenkel, and F. Schlenk. \emph{Symplectic embeddings of four-dimensional ellipsoids into integral polydisks}. Algebr. Geom. Topol. \textbf{17} (2017), 1189--1260.
\bibitem[EP09]{EP} M. Entov and L. Polterovich. \emph{Rigid subsets of symplectic manifolds}. Compos. Math. \textbf{145} (2009), no. 3, 773–826.
\bibitem[FOOO12]{FOOO} K. Fukaya, Y.-G. Oh, H. Ohta, and K. Ono. \emph{Toric degeneration and nondisplaceable Lagrangian tori in $S^2\times S^2$}. Int. Math. Res. Not. IMRN \textbf{2012}, no. 13, 2942–2993.
\bibitem[FHW94]{FHW} A. Floer, H. Hofer, and K. Wysocki. \emph{Applications of symplectic homology I}. Math. Z. \textbf{217} (1994), 577--606.
\bibitem[FM15]{FM} D. Frenkel and D. M\"uller. \emph{Symplectic embeddings of $4$-dimensional ellipsoids into cubes}. J. Symplectic Geom. \textbf{13} (2015), no. 4, 765--847.
\bibitem[GG16]{GG} V. Ginzburg and B. G\"urel. \emph{Lusternik-Schnirelmann theorem and closed Reeb orbits}, arXiv:1601.03092.
\bibitem[G15]{gutt} J. Gutt. \emph{The positive $S^1$-equivariant symplectic homology as an invariant for some contact manifolds}, arXiv:1503.01443, to appear in J. Symplectic Geom.
\bibitem[GH17]{GuH} J. Gutt and M. Hutchings. \emph{Symplectic capacities from $S^1$-equivariant symplectic homology}, arXiv:1707.06514.
\bibitem[H13]{Hi} R. Hind. \emph{Symplectic folding and nonisotopic polydisks}. Algebr. Geom. Topol. \textbf{13} (2013), no. 4, 2171--2192.
\bibitem[LiLi02]{LL} B.-H. Li and T.-J. Li. \emph{Symplectic genus, minimal genus and diffeomorphisms}. Asian J. Math. \textbf{6} (2002), no. 1, 123--144. 
\bibitem[LiLiu01]{Liu} T.-J. Li and A.-K. Liu. \emph{Uniqueness of symplectic canonical class, surface cone and symplectic cone of 4-manifolds with $b\sp +=1$}. J. Differential Geom. \textbf{58} (2001), no. 2, 331--370. 
\bibitem[M91]{M91} D. McDuff. \emph{Blow ups and symplectic embeddings in dimension 4}. Topology \textbf{30} (1991), no. 3, 409-421.
\bibitem[M09]{M} D. McDuff. \emph{Symplectic embeddings of 4-dimensional ellipsoids}. J. Topol. \textbf{2} (2009), no. 1, 1--22.
\bibitem[M11]{M10} D. McDuff. \emph{The Hofer conjecture on embedding symplectic ellipsoids}. J. Differential Geom. \textbf{88} (2011), no. 3, 519--532.
\bibitem[MP94]{MP} D. McDuff and L. Polterovich. \emph{Symplectic packings and algebraic geometry}. Invent. Math. \textbf{115} (1994), no. 3, 405--434. 
\bibitem[MS12]{MS} D. McDuff and F. Schlenk. \emph{The embedding capacity of $4$-dimensional symplectic ellipsoids}. Ann. Math. (2) \textbf{175} (2012), 1191--1282.
\bibitem[OU16]{OU} J. Oakley and M. Usher. \emph{On certain Lagrangian submanifolds of $S^2\times S^2$ and $CP^n$}.
Algebr. Geom. Topol. \textbf{16} (2016), no. 1, 149--209. 
\bibitem[S]{S} F. Schlenk. \emph{Embedding problems in symplectic geometry}. de Gruyter, 2005.
\bibitem[T95]{T} L. Traynor. \emph{Symplectic packing constructions}. J. Differential Geom. \textbf{42} (1995), no. 2, 411--429.
\bibitem[Vit99]{V} C. Viterbo. \emph{Functors and computations in Floer homology with applications. I}. Geom. Funct. Anal. \textbf{9} (1999), no. 5, 985--1033.
\bibitem[Wei94]{weibel} C. Weibel. \emph{An introduction to homological algebra}. Camb. Univ. Press, 1994.
\end{thebibliography}
\end{document}